\documentclass{amsart}

\usepackage[utf8]{inputenc}
\usepackage{hyperref}
\usepackage{listings}
\usepackage{amssymb}
\usepackage{graphicx}
\usepackage{stmaryrd}
\usepackage{tikz}
\usetikzlibrary[positioning]
\usepackage{url}
\usepackage{amsthm}
\usepackage{amsmath}
\usepackage{wasysym}
\usepackage{rotating} 
\usepackage{lscape}
\usepackage{diagrams}
\usepackage[noend]{algpseudocode}

\usepackage{pgfplots}
\usepgfplotslibrary{polar}
\pgfplotsset{compat=newest}

\usepackage{chngcntr}
\counterwithout{footnote}{section}
\counterwithin{figure}{section}

\theoremstyle{plain}
\newtheorem*{thm*}{Theorem}
\newtheorem{thm}{Theorem}[section]
\newtheorem{cor}[thm]{Corollary}
\newtheorem*{cor*}{Corollary}
\newtheorem{prop}[thm]{Proposition}
\newtheorem*{prop*}{Proposition}
\newtheorem{lem}[thm]{Lemma}

\theoremstyle{definition}
\newtheorem{defn}[thm]{Definition}
\newtheorem{ex}[thm]{Example}

\theoremstyle{remark}
\newtheorem{rmk}[thm]{Remark}

\newcommand{\jint}[1]{\textrm{Int}\left({#1}\right)}
\newcommand{\jext}[1]{\textrm{Ext}\left({#1}\right)}
\newcommand{\D}{\mathcal{D}} 
\newcommand{\J}{\mathcal{J}} 
\newcommand{\tc}[2]{\textrm{TC}_{#1}\left(#2\right)}
\newcommand{\zcl}[2]{\textrm{zcl}_{#1}\left(#2\right)}
\newcommand{\ls}[1]{\textrm{cat}\left(#1\right)}
\newcommand{\zz}{\mathbb{Z}}
\newcommand{\rr}{\mathbb{R}}

\newcommand{\op}[1]{{#1}^{\textrm{op}}} 
\newcommand{\dist}[3]{d_{#1}\left(#2,#3\right)} 
\newcommand{\set}[1]{\left\{#1\right\}}
\newcommand{\abs}[1]{\left|#1\right|}
\newcommand{\paren}[1]{\left(#1\right)}
\newcommand{\simp}[1]{\mathcal{K}\left(#1\right)}
\newcommand{\greal}[1]{\left|\simp{#1}\right|}
\newcommand{\cots}[1]{\left[#1\right]}
\newcommand{\img}[1]{\textrm{im}\left(#1\right)}

\newcommand{\five}{
\foreach \x in {1,2,3,4,5}
	{
		\draw[lightgray, thin] (+\x,1)--(\x,5);
		\draw[lightgray, thin] (1,\x)--(5,\x);
	}
\foreach \x in {1,3,5}
	{
		\draw[lightgray, thin] (\x,1)--(5,6-\x);
		\draw[lightgray, thin] (\x,5)--(5,\x);
		\draw[lightgray, thin] (1,\x)--(6-\x,5);
		\draw[lightgray, thin] (\x,1)--(1,\x);
	}

\foreach \x in {1,3,5}
	\foreach \y in {1,3,5}
	{
		\draw (\x,\y) node [shape=circle,fill=white,draw]{};
	}
\foreach \x in {2,4}
	\foreach \y in {2,4}
	{
		\draw (\x,\y) node [shape=rectangle,fill=white,draw]{};
	}
\foreach \x in {2,4}
	\foreach \y in {1,3,5}
	{
		\draw (\x,\y) node [shape=circle,fill=black,scale=.5,draw]{};
		\draw (\y,\x) node [shape=circle,fill=black,scale=.5,draw]{};
	}
}	

\title{The Topological Complexity of Spaces of Digital Jordan Curves}
\author{Shelley Kandola}
\date{June 2019}

\begin{document}

\maketitle

\begin{abstract}
    This research is motivated by studying image processing algorithms through a topological lens.
    The images we focus on here are those that have been segmented by digital Jordan curves as a means of image compression.
    The algorithms of interest are those that continuously morph one digital image into another digital image.
    Digital Jordan curves have been studied in a variety of forms for decades now.
    Our contribution to this field is interpreting the set of digital Jordan curves that can exist within a given digital plane as a finite topological space.
    Computing the topological complexity of this space determines the minimal number of continuous motion planning rules required to transform one image into another, and determining the motion planners associated to topological complexity provides the specific algorithms for doing so.
    The main result of Section \ref{chap:DigJC} is that our space of digital Jordan curves is connected, hence, its topological complexity is finite.
    To build up to that, we use Section \ref{chap:COTS} to prove some results about paths and distance functions that are obvious in Hausdorff spaces, yet surprisingly elusive in $T_0$ spaces.
    We end with Section \ref{chap:app}, in which we study applications of these results.
    In particular, we prove that our interpretation of the space of digital Jordan curves is the only topologically correct interpretation.
\end{abstract}

\section{Preliminaries}

Understanding the topology of digital images has applications in a wide variety of fields.
Khalimsky, Kopperman, and Meyer's motivation in \cite{Khalimsky1990a} is image compression.
In \cite{Kong1992}, Kong et al. discuss isomorphisms of digital images as a means of thinning, border-finding, and rotating digital images.
Applications of image processing include document reading, image segmentation (e.g., recognizing separate components of an image), and even artificial intelligence \cite{Eckhardt1994}.
Studying the topology of Jordan curves is a natural starting point for tackling these problems.
In the Euclidean setting, a Jordan curve  is a non-self-intersecting continuous loop in the plane.
The Jordan curve theorem of \cite{Berg1975} states the following:
\begin{thm*}
Let C be a Jordan curve in the plane $\rr^2$. 
Then its complement, $\rr^2 - C$, consists of exactly two connected components. 
One of these components is bounded (the interior) and the other is unbounded (the exterior), and the curve $C$ is the boundary of each component.
\end{thm*}
A digital image can be separated into simple loops, each of whose interiors are precisely one color.
Understanding these components could help differentiate the foregrounds of images from their backgrounds.
After segmenting an image by Jordan curves, one could store the color of the interior of each Jordan curve as a means of image compression.

In \cite{Rosenfeld2007}, Rosenfeld proves a digital Jordan curve theorem using a graph-theoretic approach.
In \cite{Kong1992}, Kong, Roscoe and Rosenfeld prove a digital Jordan curve theorem that uses continuous analogues of digital pictures in Euclidean space.
This article uses the Jordan curve theorem from \cite{Khalimsky1990a}, which they prove using an axiomatic approach that defines a digital plane as a finite topological space.
Their method makes no appeal to the graph-theoretic or continuous approaches, and is purely topological in nature.
Later on in \cite{Slapal2006}, \v{S}lapal proves a digital Jordan curve theorem for a topology on $\mathbb{Z}^2$ that is not the Khalimsky topology, and allows for digital Jordan curves that turn at $\frac{\pi}{4}$ angles (see Example \ref{ex:acute} for why this is not possible with the Khalimsky topology).

In 1990, Khalimsky et al. published \cite{Khalimsky1990a}, in which they developed a finite analog of the Jordan curve theorem.
Their Jordan curve theorem exists in the context of a digital plane, $\D$, which is a model of a computer screen as a finite rectangular array.
Such a finite rectangular array only admits one $T_2$ topology, which is the discrete topology.
Their digital plane, however, is a conntected $T_0$ space that is the product of two finite connected ordered topological spaces.
This construction allows for the defining of paths, arcs, and curves that are finite analogs of their Euclidean counterparts.

In Sections \ref{chap:COTS} and \ref{chap:DigJC}, we build the tools necessary for looking at the set of all digital Jordan curves that can exist in a given finite digital plane equipped with the Khalimsky topology.
We introduce parameterizations of Jordan curves that allow two digital Jordan curves to be homotopic to each other with respect to the order topology on the digital plane.
The topology that $Map(S^1,\D)$ inherits from the digital plane
lends itself to defining the set of all digital Jordan curves within Khalimsky's digital plane as a finite topological space, $\J$.
We end Section \ref{chap:DigJC} with the main result:
\begin{thm}\label{main}
The space of digital Jordan curves $\J$ is path-connected.
\end{thm}
In Section \ref{chap:app}, we prove a satisfying justification:
\begin{thm}\label{thm:onlychoice}
Khalimsky's topology on $\D$ is the only topology for which $\J(\D)$ is path-connected.
\end{thm}

Our proof of Theorem \ref{main} is a corollary of two original results.
First, we show that any digital Jordan curve can be continuously deformed to a minimal digital Jordan curve enclosing one of its interior points.
Second, we show that there exists a fence of homotopies between any two minimal digital Jordan curves in the digital plane.

Later in Section \ref{chap:DigJC}, we explore the shapes of these spaces for varying sizes of digital planes.
In particular, we will prove that for a connected ordered topological space $X$, the space of digital Jordan curves in the digital plane $X \times X$ is contractible for $|X|=3,4,5$.

Our motivation for understanding the shapes of these spaces lies in topological complexity, introduced by Farber in \cite{Farber2001}.
Topological complexity has its roots in robot motion planning;
i.e., it answers the question, ``What is the minimal number of continuous motion planning rules required to instruct a robot to move from one position into another position?''
Applying this concept to the space of digital Jordan curves (and, eventually, to the space of more complex digital images) is analogous to understanding the complexity of an algorithm for morphing one digital image into another digital image.

In the next three sections, we cover the basics of finite topology, topological complexity, and digital topology.

\subsection{Finite Topology}

A finite topological space is a topological space with finitely many points.
Finite spaces are also studied as Alexandroff spaces, first introduced in \cite{Alexandroff1937}.
An \textbf{Alexandroff space} (or $A$-space, as in \cite{McCord1966}) is a topological space in which the arbitrary intersection of open sets is open.
The \textbf{minimal open neighborhoods} $U_x$ are given by 
\[U_x = \bigcap \set{U \mid U \ni x \textrm{ is open}},\]
and the $U_x$ form a basis for the topology on $X$
These $U_x$ also give rise to a preorder.
A {\bf preorder}  $\leq$ is a binary relation that is reflexive ($x \leq x$) and transitive ($x \leq y, y\leq z \implies x \leq z$).
If $x \leq y$ and $y \leq x$ implies $x=y$, then $\leq$ is a \textbf{partial order}.
We say $x \leq y$ if and only if $U_x \subseteq U_y$, and the sets $U_x$ form a basis for the topology on $X$. \footnote{We include these definitions to avoid confusion about the preorder of the finite $T_0$ spaces mentioned in this article.
For example, in \cite{El-Atik2002} and \cite{Khalimsky1990}, $x \leq y$ if and only if $x$ is in the closure of $y$.
In works such as \cite{McCord1966},\cite{Stong1966}, and \cite{Barmak2012}, however, $x \leq y$ if and only if $x$ is in the minimal open neighborhood of $y$.
The latter approach is what we will use here.}
One way to visualize this data is with a Hasse diagram.
The \textbf{Hasse diagram} of a partially ordered finite topological space $X$ is a directed graph whose vertices are the points of $X$ and whose directed edges are \[E(X) = \set{x \to y \mid y < x \textrm{ and } y \leq z \leq x \implies y=z \textrm{ or }x=z}.\]
Typically, the graph will be displayed without the directed edges, with the implied direction being ``down.''
See Figure \ref{fig:J4x4} for an example.

It is common to use $U_x$ or $U(x)$ (see \cite{Stong1966}, \cite{McCord1966}, \cite{Barmak2012}, \cite{El-Atik2002}, and many others) or $N(x)$ (see \cite{Khalimsky1990} or \cite{Khalimsky1990a}) to denote the minimal open neighborhood of $x$.
Because of the shape of open sets in the Hasse diagram, and to free up the variable $U$ for later computations, we adopt the notation
\[x^\downarrow = U_x=U(x)=\set{y \in X \mid y \leq x}.\]
Similarly, $x^\uparrow = \{y \in X \mid y \geq x\}$ is the \textbf{closure} of $x$, often denoted $F_x$ or $F(x)$.
\begin{ex}
To see that the closure of $x$ is given by $x^\uparrow=\set{y\in X \mid y\geq x}$, recall that by definition, $U_x \subseteq U_y \iff x \in U_y \iff x \leq y$.
We seek to show that $y \geq x$ implies $y \in F_x$.
If $y \not\in F_x$, then $U_y-F_x$ is an open set containing $y$ that does not contain $x$.
Since $x \in U_y$, however, every open set containing $y$ must also contain $x$, a contradiction.
Hence, $y \geq x$ implies $y$ is in the closure of $x$, so $F_x = \set{y \in X \mid y \geq x}$.
\end{ex}
We call these sets \textbf{down-sets} and \textbf{up-sets}, respectively.
The {\bf adjacency set} of a point $x \in X$ is 
\[A(x) = \set{y \in X, y \neq x \mid x \leq y \textrm{ or }y \leq x}.\]
In this case, $x$ and $y$ are {\bf adjacent} to each other, that is, $x \in A(y)$ and $y \in A(x)$.
Notice that $x \not\in A(x)$.
If $Y \subseteq X$ is a subspace, we say $A(Y) = \bigcup \set{A(y)\mid y\in Y}$.
Later on, we will use the notation $x \lessgtr y$ to mean that either $x \leq y$, or $x \geq y$.

\begin{table}
    \centering
    \begin{tabular}{c|c}
         Separation Axiom& Condition  \\
         \hline
         $T_0$ & $x^\downarrow \neq y^\downarrow \iff x \neq y$\\
         $T_\frac{1}{2}$ & $x^\downarrow=\{x\}$ or $x^\uparrow = \{x\}$ $\forall x \in X$\\
         $T_1$ & $x^\uparrow = \{x\} \forall x \in X$\\
         $T_2$ & $x^\downarrow \cap y^\downarrow = \emptyset \forall x\neq y \in X$
    \end{tabular}
    \caption{Summary of separation axioms in terms of minimal open neighborhoods and closures}
    \label{tab:T}
\end{table}

In Table \ref{tab:T}, we list the first few separation axioms in terms of the down-sets and up-sets of points in a finite space.
If $X$ is $T_0$, then each of $x^\downarrow$ and $x^\uparrow$ have a unique maximal or minimal element, respectively, and so both $x^\downarrow$ and $x^\uparrow$ are contractible.
If $x^\downarrow-\{x\}$ has a unique maximal element or $x^\uparrow-\{x\}$ has a unique minimal element, then $x$ is {\bf down beat} or {\bf up beat}, respectively.\footnote{In \cite{Stong1966}, up beat points are called {linear}, and down beat points are called {colinear}.}
If $x$ is either down beat or up beat, it is simply called {\bf beat}.
If either $x^\downarrow-\{x\}$ or $x^\uparrow=\{x\}$ is contractible, then $x$ is {\bf weak}.
Removing the beat points yields a strong deformation retract, as shown in Proposition 1.3.4 of \cite{Barmak2012}.
\begin{prop*}
If $X$ is a finite $T_0$ space and $x \in X$ is beat, then $X-\{x\}$ is a strong deformation retract of $X$.
\end{prop*}
Removing beat points from $X$ until none are left results in the {\bf core} of $X$.
In Corollary 4 of \cite{Stong1966}, Stong proves the following.
\begin{cor*}
A finite $T_0$ space $X$ is contractible if and only if the core of $X$ is a point.
\end{cor*}
In the absence of beat points, removing weak points from a finite space yields a weak homotopy equivalence (see Proposition 4.2.4 of \cite{Barmak2012}).
\begin{prop*}
If $x$ is a weak point of a finite $T_0$ space $X$, then the inclusion map $\iota:X-\{x\} \hookrightarrow X$ is a weak homotopy equivalence.
\end{prop*}
Recall that a map $f:X \to Y$ \textbf{weak homotopy equivalence} if it induces isomorphisms on all homotopy groups of $X$ and $Y$.
In this case, we may say $X \stackrel{we}{\simeq} Y$.

A function $f:X \to Y$ between two preordered sets is {\bf order-preserving} if $x \leq x'$ implies $f(x)\leq f(x')$ for all $x,x'\in X$.
In Proposition 7 of \cite{Stong1966}, Stong proves the following.
\begin{prop*}
A function $f:X \to Y$ between finite spaces is {continuous} if and only if it is order-preserving.
\end{prop*}
A {\bf fence} in $X$ is a sequence $x_0, x_1,\hdots , x_n$
of points such that any two consecutive points are comparable.
If there exists a fence between any two points in $X$, then $X$ is {\bf order-connected}.
In a finite space $X$, the notions of order-connected, connected, and path-connected are all equivalent (see Proposition 1.2.4 of \cite{Barmak2012}).
If $x$ and $y$ are comparable points in a finite space $X$, then there exists a path $\gamma:I \to X$ such that $\gamma(0)=x$ and $\gamma(1)=y$.

Consider $Y^X = Map(X,Y)$ with the compact-open topology.
When $X$ and $Y$ are finite, we may also consider the {\bf pointwise order} on $Y^X$: given two $f,g \in Y^X$ $f \leq g$ if $f(x) \leq g(x)$ for all $x \in X$.
By Proposition 9 of \cite{Stong1966} and Proposition 1.2.5 of \cite{Barmak2012}, we have the following.
\begin{prop*}
Let $X$ and $Y$ be two finite $T_0$ spaces.
Then the pointwise order on $Y^X$ corresponds to the compact-open topology.
\end{prop*}

By Corollary 3 of \cite{Stong1966}, if $f\leq g$, then they are \textbf{homotopic}, denoted $f\simeq g$.
A {\bf fence} in $Y^X$ is a sequence $f_0, f_1, \hdots, f_n$ of functions $f_i\in Y^X$ such that either $f_i \leq f_{i+1}$ or $f_i \geq f_{i+1}$ for $0 \leq i < n$.
In Proposition 14 of \cite{Stong1966}, Stong shows that $X$ need not always be finite.
If $f,g:X \to Y$ are maps from any topological space $X$ to a finite space $Y$, $f \leq g$ implies $f \simeq g$ in the traditional sense:
If $f \leq g$ for $f,g:X\to Y$, there exists a map $H:X \times I \to Y$ such that $H(x,0)=f(x)$ and $H(x,1)=g(x)$.

A finite $T_0$ space $X$ generates a simplicial complex called the \textbf{order complex} $\mathcal{K}(X)$, whose simplices are chains in $X$.
The $n$-simplices of $\simp{X}$ are determined by the $n$-chains $x_0 < x_1 < \hdots < x_n$ in $X$.
The points in the {\bf geometric realization} $\greal{X}$ are of the form
\[\alpha = \sum_{t=0}^n t_ix_i\]
where $x_0<x_1<\hdots<x_n$ is an $n$-chain in $X$, $t_i>0$ for $0\leq i \leq n$, and $\sum_{i=0}^n t_i=1$.
There exists a weak homotopy equivalence \[\mu_X: |\mathcal{K}(X)| \to X\] called the {\bf $\mathcal{K}$-McCord map} that sends a point $\alpha \in |\mathcal{K}(X)| $ to $\min(\textrm{support}(\alpha)) \in X$.
For $\alpha$ as written above, $\mu_X(\alpha)=x_0 \in X$.

Given a finite space $X$, its \textbf{dual} $\op{X}$ is the space whose topology is given by the closed subspaces of $X$.
That is, \[\op{X} \supseteq x^\downarrow = x^\uparrow \subseteq X\] for all $x \in X$.

\begin{prop}\label{prop:dualx}
Dualizing a finite space distributes across products.
That is,
\[\op{(X \times Y)} = \op{X} \times \op{Y}.\]
\end{prop}
\begin{proof}
Recall the product topology: if $X$ and $Y$ are finite spaces, open sets of the form $(x,y)^\downarrow = x^\downarrow \times y ^\downarrow$ form a basis for the topology on $X \times Y$ for $x \in X$ and $y \in Y$.
Let $(x,y) \in \op{(X \times Y)}$. 
Then:
\begin{eqnarray*}
\op{(X \times Y)} \supseteq (x,y)^\downarrow &=& (x,y)^\uparrow \subseteq X \times Y\\
&=& x^\uparrow \times y^\uparrow \subseteq X \times Y\\
&=&x^\downarrow \times y^\downarrow \subseteq \op{X} \times \op{Y}\\
&=& (x,y)^\downarrow \subseteq \op{X} \times \op{Y}.
\end{eqnarray*}
Hence, the open sets of $\op{(X \times Y)}$ are the open sets of $\op{X} \times \op{Y}$, so $\op{(X \times Y)} = \op{X} \times \op{Y}$.
\end{proof}

\subsection{Topological Complexity}

\subsubsection{Farber's Topological Complexity}
In \cite{Farber2001}, Farber introduces the notion of topological complexity as it relates to motion planning in robotics. 
Informally, the topological complexity $\tc{}{X}$ of a path-connected space $X$ is the minimal number of ``motion planning rules'' required to move from one point in the configuration space to another. 
Despite the seminal paper being published in 2003, topological complexity draws from the Schwarz' genus of a fiber space, described in \cite{Schwarz1958} in 1958.

\begin{defn}
The {\bf Schwarz genus} $\mathfrak{g}(p)$ of a fibration $p:E \to B$ is the minimal number $k$ such that there exists an open covering $U_1,\hdots,U_k$ of $B$ where each set $U_i$ admits a local $p$-section.
That is, each $U_i$ has an associated map $s_i$ such that $p \circ s_i \simeq 1_B$.
\end{defn}

\begin{figure}
\begin{diagram}
X&\rTo^\alpha&X^{[0,1]}\\
 &\rdTo_\Delta&\dTo_\pi\\
&&X \times X
\end{diagram}
\caption{The commutative diagram for $\tc{}{X}$}
\label{TC2}
\end{figure}

Consider Figure \ref{TC2}.
Here, $\alpha$ is the map that takes a point $x \in X$ and maps it to the constant path at $x$ in $X^{[0,1]}$, which is the path space of $X$.
The map $\Delta$ is the diagonal map sending $x \mapsto (x,x)$, and
$\pi$ is the fibrant replacement of $\Delta$ that sends a path $\varphi \in X^{[0,1]}$ to the ordered pair $(\varphi(0),\varphi(1))$, thereby recording its endpoints.

\begin{defn}
Formally, the \textbf{topological complexity} of a space $X$ is given by $\tc{}{X}=\mathfrak{g}(\pi)$, where $X$ is a path-connected space.
If $\tc{}{X}=k$, there exist $k$ open subsets covering $X \times X = U_1 \cup U_2 \cup\hdots \cup U_{k}$ such that for any $i \in \{1,2,\hdots,k\}$, there exists a continuous section $s_i:U_i \to X^{[0,1]}$ satisfying $\pi \circ s_i(u) = u$. The $U_i$ are called local domains, and the $s_i$ are called local rules, and $(U_i,s_i)$ is a local motion planner on $X$.
\end{defn}

\begin{rmk}
In Farber's original paper, he defines $\tc{}{X} = \mathfrak{g}(\pi)$. It has since become common practice to take $\tc{}{X}=\mathfrak{g}(\pi)-1$, in order to make bounds of topological complexity behave nicely.
In this article we use Farber's original unreduced definition.
\end{rmk}

Less abstractly, we are interested in the configuration space $X$ of a mechanical system, such as the arm of a robot. 
For example, if our robot has one jointless arm with a circular range of motion 
(e.g., a security camera whose range of motion is $360^\circ$), 
we compute the topological complexity of $S^1$ to give us the minimal number of rules required to get from one point on the circle to any other. 
Such a motion planner is described in \cite{Farber2001} and \cite{Gonzalez2010}:

\begin{ex}\label{s1mp}
It is well-known that $\tc{}{S^1}=2$, meaning that a motion planner for $S^1$ requires two sets covering $S^1 \times S^1$ that each admit local $\pi$-sections.
One example for the local domains of a motion planner for $S^1$ is given by
\begin{align*}
U_1&=\set{(x,-x) \in S^1 \times S^1}\\
U_2&=\set{(x,y) \in S^1 \times S^1\mid x\neq-y}.
\end{align*}
It is easiest to see a motion planning rule on $U_2$:
if $y\neq -x$, there exists a unique shortest arc connecting $x$ to $y$ in $X$, so $s_2(x,y)$ is the path at constant speed along that arc.
If $y=-x$, then no shortest arc exists, so  $s_1:U_1 \to \left(S^1\right)^{[0,1]}$ is the path at constant speed from $x$ to $-x$, along the semicircle determined by $V_x$, where $V$ is some fixed nonzero tangent vector field of $S^1$.
Hence $\tc{}{S^1}=2$.
\end{ex}

\begin{rmk}
It may appear that this construction does not satisfy the definition of topological complexity since $U_1$ is not open, however, it follows from Proposition 2.2 of \cite{Rudyak2010} that the open sets covering $X \times X$ may be replaced with (not necessarily open) Euclidean neighborhood retracts.
\end{rmk}

In general, topological complexity is difficult to compute, and explicit motion planners are rare.
In practice, the most successful way to determine the topological complexity of a space is by upper and lowerbounds.
The most simple bounds are given by the Lusternik-Schnirelmann category, first defined in \cite{LS34}.

\begin{defn}
The \textbf{Lusternik-Schnirelmann category} of a space $X$, denoted $\ls{X}$, is the minimal number $k$ such that $X$ can be covered in $k$ open sets $U_i \subseteq X$ whose inclusion maps $\iota_i: U_i \hookrightarrow X$ are nullhomotopic.
\end{defn}

For finite spaces, the Lusternik-Schnirelmann category provides the following bounds:
\[\ls{X} \leq \tc{}{X} \leq \ls{X\times X} \leq \ls{X}^2.\]
The inequality $\ls{X}\leq\tc{}{X}\leq \ls{X\times X}$ is proven in \cite{Farber2001}, and is true for all path-connected spaces $X$.
For finite path-connected spaces, the inequality $\ls{X\times X} \leq \ls{X}^2$ is proven independently in \cite{Kandola2018} and \cite{Tanaka2018}.
If $X$ is compact and $T_2$, the inequality is improved:
\[\ls{X} \leq \tc{}{X} \leq 2\ls{X}-1,\] as shown in Theorem 5 of \cite{Farber2001}.\footnote{In \cite{Farber2001}, Farber states Theorem 5 to be true for path-connected paracompact spaces. All finite spaces are compact and therefore paracompact, however, it is common in literature to define paracompact spaces such that they are always Hausdorff.}

A more precise cohomological lowerbound for topological complexity comes from Definition 6 of \cite{Farber2001}.
The map $\Delta:X \to X \times X$ in Figure \ref{TC2} induces a map in cohomology, $\Delta^*: H^*(X \times X) \cong H^*(X)\bigotimes H^*(X) \to X^*(X)$.
The longest non-vanishing product of nontrivial elements of $\ker(\Delta^*)$ is called the {\bf zero divisors cup length} and denoted $\zcl{}{X}$.
By Theorem 7 of \cite{Farber2001}, this provides a strict lowerbound for topological complexity.

\begin{thm*}
Let $X$ be a CW-complex and $k$ a field. Then $\zcl{}{X}<\tc{}{X}$.
\end{thm*} 

Let $X$ be a finite $T_0$ space.
Because $\mu_X:\greal{X}\to X$ is a weak homotopy equivalence, there is an isomorphism of homology groups $\mu_{X*}:H_n(\greal{X}) \to H_n(X)$.
By the universal coefficient theorem, this also induces an isomorphism in cohomology groups (see Theorem 3.2 of \cite{Hatcher2002}).
Then $\zcl{}{\greal{X}}=\zcl{}{X}$.
From this and Corollary 4.10 of \cite{Tanaka2018}, we have the following inequality:
\[\zcl{}{X}=\zcl{}{\greal{X}}<\tc{}{\greal{X}}\leq \tc{}{X}.\]
Because of this strict inequality, many papers adopt a reduced definition of topological complexity such that $\zcl{}{X}\leq \tc{}{X}$.
Because $\zcl{}{X}$ is not more useful than $\tc{}{\greal{X}}$ for a finite space $X$ when computing $\tc{}{X}$, we use the unreduced notion of topological complexity.
In part, this is because there are currently no known finite spaces $X$ such that $\tc{}{X}$ is known and $\tc{}{\greal{X}}$ is not.

\subsubsection{Topological Complexity of Discretized Spaces}
\label{subsec:dtc}

Studying the topological complexity of discretized topological spaces may be more useful in real-life applications.
Consider Example \ref{s1mp} as it relates to a rotating security camera.
Rotating $n^\circ$ is indistinguishable from rotating $(n+\varepsilon)^\circ$ for some sufficiently small $\varepsilon \in S^1$.
Realistically, the camera can only rotate into finitely many positions.
In this section, we review three notions of topological complexity that have been adapted to discretized spaces.
The first two notions, simplicial complexity and discrete topological complexity, we will only cover in brief since they are not used elsewhere in this article.

In \cite{Gonzalez2017}, Gonz\'{a}lez defines an analog of topological complexity for simplicial complexes, called {\bf simplicial complexity}\footnote{The definition of simplicial complexity in \cite{Gonzalez2017} is reduced, and we use unreduced values in this article.} and denoted $\text{SC}(K)$ for a simplicial complex $K$.
The computation of simplicial complexity depends on taking repeated barycentric subdivisions of the space.
Gonz\'{a}lez' definition is adapted from Iwase and Sakai's intepretation of topological complexity as a fibrewise Lusternik-Schnirelmann category, introduced in \cite{Iwase2012}.
Their notion agrees with Farber's topological complexity of the geometric realization of $K$, as proven in Theorem 1.6 of \cite{Gonzalez2017}:
\[SC(K)=\tc{}{|K|}\]

As a consequence, $\text{SC}(K)=2$ for any complex whose realization has the homotopy type of an odd sphere.
In particular, this includes $K$ such that $|K|\simeq S^1$.
They demonstrate this in Section 3 of \cite{Gonzalez2017} with $S^1$ modeled by the 1-skeleton of the 2-dimensional simplex $\Delta^2$, denoted $S_1$.
The open sets of $S_1 \times S_1$ admitting continuous motion planning rules follow Farber's construction closely. One collapses to $\set{(x,x) \in S_1\times S_1}$, and one to $\set{(x,-x) \in S_1 \times S_1}$.
While Gonz\'{a}lez' definition of topological complexity for simplicial complexes involves taking repeated barycentric subdivisions, the definition of {\bf discrete topological complexity}\footnote{The definition of discrete topological complexity in \cite{Fernandez-Ternero2018} is reduced, and we use the unreduced values in this article.} (DTC) in \cite{Fernandez-Ternero2018} is defined in purely combinatorial terms.
Fern\'{a}ndez-Ternero, et al. prove in Example 4.9 of \cite{Fernandez-Ternero2018} that $\textrm{DTC}\paren{S_1}=3$.
For larger simplicial models of $S^1$, Theorem 5.6 of that paper proves the discrete topological complexity drops back down to 2.

Tanaka introduces {\bf combinatorial complexity}\footnote{The definition of combinatorial complexity in \cite{Tanaka2018} is unreduced, and no changes have been made from the values in that paper.} (CC) in \cite{Tanaka2018} as an analog of topological complexity for finite spaces.
Tanaka's notion is most useful to us because it can be applied to finite spaces that do not have an underlying simplicial structure.
The definition of combinatorial complexity differs from Farber's topological complexity in that they consider finite models of the interval in place of $I$.
Denote by $J_m$ a finite space with $m+1$ points whose order is given by
\[0<1>2<\hdots \lessgtr m.\]
In the next section, we will show why this is an appropriate model of a line interval.
For example, the finite space in Figure \ref{cots} would be referred to as $J_8$ in \cite{Tanaka2018}.
\begin{defn}
Given a finite $T_0$ space $X$, $\textrm{CC}_m(X)$ is the smallest positive integer $n$ such that $X \times X$ can be covered in $n$ open sets $U_i$ such that each $U_i$ admits a continuous section $s_i:U_i \subseteq X\times X \to X^{J_m}$. 
The {\bf combinatorial complexity} of $X$ is given by \[\textrm{CC}(X):=\lim_{m \to \infty} \textrm{CC}_m(X).\]
\end{defn}

Theorem 3.2 of \cite{Tanaka2018} proves the following:
\[\textrm{It holds that $\tc{}{X} = \text{CC}(X)$ for any connected finite space $X$.}\]
Because connected finite spaces are path-connected by Proposition 1.2.4 of \cite{Barmak2012}, this is sufficient for defining a notion of topological complexity.
In Example 4.5 of that paper, Tanaka proves that for the minimal finite model $\mathbb{S}^1$ of $S^1$, $\tc{}{\mathbb{S}^1}=4$; this is the result that motivated \cite{Kandola2018}.

\subsection{Digital Topology}
\label{sec:digtop}

Digital topology arose as a topological tool in image processing.
Every digital topology includes some model of a computer screen.
These models can be graphs, imbeddings into $\rr^2$, or axiomatic (see \cite{Rosenfeld2007}, \cite{Kong1992}, and \cite{Khalimsky1990a}, respectively).
While most authors call this model a ``digital plane'' (see \cite{Kovalevsky1989}, \cite{Khalimsky1990a}, \cite{Eckhardt1994}, \cite{Kiselman2000}, \cite{Slapal2006}), it is also called a ``digital picture space'' (see \cite{Kong1992}).
Many papers approaching digital topology from a computer science perspective treat the digital plane as $\zz^2$ with prescribed adjacencies (see \cite{Marcus1970},\cite{Kong1992}, \cite{Eckhardt2003},\cite{Boxer2005}, \cite{Slapal2006}, for example).
In this next section, we focus on the digital plane of Khalimsky, Kopperman, and Meyer in \cite{Khalimsky1990a}, followed by a review of the other models.

\subsubsection{The Khalimsky Topology}

Because digital planes are discretized, we need discretized analogs of common structures in topology.
Much like a line segment, one can think of a finite connected ordered topological space as a finite topological space that has two endpoints, each with one neighbor, and all other points each have two neighbors.
The notion of COTS-arcs and COTS-paths was first brought up in \cite{Khalimsky1990a}.

\begin{defn}
A {\bf COTS} (connected ordered topological space) is a connected topological space $X$ with the following property:
if $Y \subseteq X$ is a 3-point subset, then there exists a $y \in Y$ such that $Y-\{y\}$ has two nonempty components.
Colloquially, every 3-point subset $Y$ of $X$ has one point that separates the other two.
\end{defn}

When a COTS is finite, the points alternate between open and closed, as shown in Lemma 2.8 of \cite{Khalimsky1990a}.
We provide an alternative explanation here.
This can also be seen from the Hasse diagram of a COTS, an example of which is shown in Figure \ref{fig:Hcots}.
Notice how $x_0$ and $x_8$ have exactly one neighbor in the space, and all other $x_i$ have exactly two neighbors.

\begin{figure}
    \centering
     \begin{tikzpicture}[mixed/.style={shape=circle,fill=black,scale=.5,draw}]
    \foreach \x in {1,3,5,7}{
    \node[mixed,label=above:{$x_{\x}$}] (x\x) at (\x,1){};
    }
    \foreach \x in {0,2,4,6,8}{
    \node[mixed,label=below:{$x_{\x}$}] (x\x) at (\x,0){};
    }
    \draw (x0)--(x1)--(x2)--(x3)--(x4)--(x5)--(x6)--(x7)--(x8);
        
    \end{tikzpicture}
    \caption{The Hasse diagram of a COTS with 9 points}
    \label{fig:Hcots}
\end{figure}
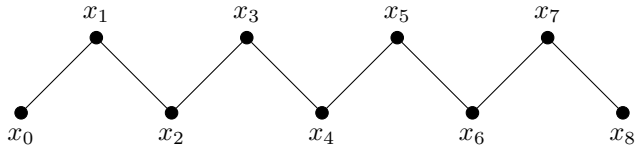

\begin{prop}
Every point of a COTS is either open or closed, and no two points of the same type are adjacent.
\end{prop}
\begin{proof}
To see this, recall that for any three-point subset $Y$ of a COTS $X$, there is a $y \in Y$ such that $Y$ intersects with two connected components of $X-\{y\}$.
Since $X$ is connected, there exists a connected 3-point subset $Y:=\{x,y,z\}$ of $X$
such that $y$ is the point that separates the other two.
Then $Y \cap (X-\{y\}) = \{x,z\} = Y-\{y\}$,
where $x$ includes into one component of $X-\{y\}$,
and $z$ includes into another component of $X-\{y\}$.
As subspaces of $Y-\{y\}$, each of $\{x\}$ and $\{z\}$ are open.
Within $Y$, the open sets are
\[\{x\}\textrm{ or }\{x,y\}\]
along with
\[\{z\}\textrm{ or }\{y,z\}.\]
Note that $\{x\},\{y\}$, and $\{z\}$ cannot all be open sets of $Y$, or else $Y$ would have the discrete topology, and $Y$ would not be connected.
So there are four cases to consider.

If the open sets of $Y$ are $\{x\}$ and $\{y,z\}$, then $x^\downarrow \not\subseteq y^\downarrow, x^\downarrow \not\subseteq z^\downarrow$, and $y^\downarrow = z^\downarrow \not\subseteq x^\downarrow$, so $x$ is not connected to any point of $Y$, a contradiction. 
Similarly, if the open sets of $Y$ are $\{x,y\}$ and $\{z\}$, then $z$ is not connected to any point of $Y$, a contradiction.

If the open sets of $Y$ are $\{x\}$ and $\{z\}$, then $y$ is a closed point and $x$ and $z$ are each open points.
If the open sets of $Y$ are $\{x,y\}$ and $\{y,z\}$, then $\{x,y\} \cap \{y,z\} = \{y\}$ is an open point, and $x$ and $z$ are each closed points.
In the latter two cases, the points alternate between open and closed, as shown in Figure \ref{cots}.
\end{proof}

Notice that the definition of a COTS does not specify that a COTS must be finite.
It is easy to see, for example, that $\rr$ is a COTS: for any three-point, one of the points separates the other two.
In this document, however, we will only be concerned with finite COTS.
Example 1.6 of \cite{El-Atik2002} lists, up to homeomorphism, all nine topologies on a three-point topological space.
The topology $\tau_7$ of that example is given by the set $\set{x,y,z}$ with minimal open sets $\set{x,y}$, $\set{y,z}$, and $\set{y}$.
That is the only topology on three points that satisfies the definition of a COTS.
Figure \ref{cots} displays a COTS with nine points whose Hasse diagram is shown in Figure \ref{fig:Hcots}.
In the context of digital topology, we adopt the convention of \cite{Khalimsky1990a} and \cite{Khalimsky1990} to use squares for closed points and circles for open points.
Later on, we use solid black dots to represent points that are neither open nor closed.

\begin{figure}
\centering
\begin{tikzpicture}
\draw (0,0) ellipse (40pt and 10pt);
\draw (2,0) ellipse (40pt and 10pt);
\draw (4,0) ellipse (40pt and 10pt);
\draw (6,0) ellipse (40pt and 10pt);
\draw (-1,0) circle (10pt);
\draw (1,0) circle (10pt);
\draw (3,0) circle (10pt);
\draw (5,0) circle (10pt);
\draw (7,0) circle (10pt);

\draw (-1,0)--(7,0);
\path (-1,0) node[shape=circle,fill=white,draw]{}
(0,0) node[shape=rectangle,fill=white,draw]{}
(1,0) node[shape=circle,fill=white,draw]{}
(2,0) node[shape=rectangle,fill=white,draw]{}
(3,0) node[shape=circle,fill=white,draw]{}
(4,0) node[shape=rectangle,fill=white,draw]{}
(5,0) node[shape=circle,fill=white,draw]{}
(6,0) node[shape=rectangle,fill=white,draw]{}
(7,0) node[shape=circle,fill=white,draw]{};
\end{tikzpicture}
\caption{cots A finite COTS $X$ with the minimal open set of each point circled}
\label{cots}
\end{figure}
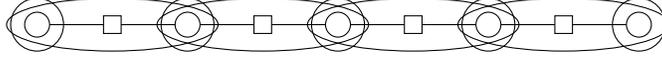

A Hausdorff representation of a computer screen might be in the form of $[a,b]\times[c,d]$, for some $a<b$ and $c<d$ with $a,b,c,d \in \mathbb{R}$.
A finite COTS is a representation of a line segment, and it can be used similarly to define a finite representation of a rectangle.

\begin{defn}
If $X$ and $Y$ are finite COTS with $|X|\geq 3$ and $|Y| \geq 3$, then a space $X \times Y$ equipped with the product topology is called a {\bf digital plane}.
Throughout this paper, $\D$ will refer to a digital plane that is sufficiently large, unless otherwise specified.
\end{defn}

Because the COTS $X$ and $Y$ are finite and $T_0$, they yield a partial order whose product gives rise to a partial order on $\D$.
If $X = \set{x_0,x_1,\hdots,x_{|X|-1}}$ and $Y=\set{y_0,y_1,\hdots,y_{|Y|-1}}$, each point of $\D$ is of the form $(x_i,y_j)$ for some integer $0 \leq i < |X|$ and $0\leq j < |Y|$.
For computational purposes, we will refer to points $(x_i,y_j)$ by their indices $(i,j)$, which can also be thought of as integer coordinates in $\zz^2$.

If $|X|=|Y|=3$ and $X=Y$, then $X \times Y$ is the digital plane shown in either Figure \ref{fig:pure} or Figure \ref{fig:purec}. 
If $|X|=|Y|=5$ and $X=Y$, then $X \times Y$ is either Figure \ref{fig:plane} or its dual, depending on whether the endpoints of $X$ are open or closed.
For example, if $X$ is the COTS in Figure \ref{cots}, then $\op{X}$ is also a COTS comprising nine points, however, it has five closed points and four open points.
The lines in these figures are not part of the digital planes themselves; they indicate which points are adjacent according to the topology on the space.

\begin{figure}
\centering
\begin{tikzpicture}
\foreach \x in {1,2,3}
	{
		\draw[lightgray, thin] (\x,1)--(\x,3);
		\draw[lightgray, thin] (1,\x)--(3,\x);
	}
		\draw[lightgray, thin] (1,1)--(3,3);
		\draw[lightgray, thin] (1,3)--(3,1);
	\draw (2,2) node [shape=circle,fill=white,draw]{};

	\draw (1,1) node [shape=rectangle,fill=white,draw]{};
	\draw (1,3) node [shape=rectangle,fill=white,draw]{};
	\draw (3,1) node [shape=rectangle,fill=white,draw]{};
	\draw (3,3) node [shape=rectangle,fill=white,draw]{};

		\draw (1,2) node [shape=circle,fill=black,scale=.5,draw]{};
		\draw (2,1) node [shape=circle,fill=black,scale=.5,draw]{};
		\draw (2,3) node [shape=circle,fill=black,scale=.5,draw]{};
		\draw (3,2) node [shape=circle,fill=black,scale=.5,draw]{};	
\end{tikzpicture}
\caption{An open point and its adjacency set}
\label{fig:pure}
\end{figure}
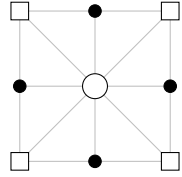

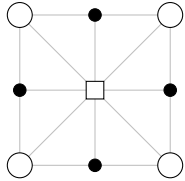
\begin{figure}
\centering
\begin{tikzpicture}
\foreach \x in {1,2,3}
	{
		\draw[lightgray, thin] (\x,1)--(\x,3);
		\draw[lightgray, thin] (1,\x)--(3,\x);
	}
		\draw[lightgray, thin] (1,1)--(3,3);
		\draw[lightgray, thin] (1,3)--(3,1);
	\draw (2,2) node [shape=rectangle,fill=white,draw]{};

	\draw (1,1) node [shape=circle,fill=white,draw]{};
	\draw (1,3) node [shape=circle,fill=white,draw]{};
	\draw (3,1) node [shape=circle,fill=white,draw]{};
	\draw (3,3) node [shape=circle,fill=white,draw]{};

		\draw (1,2) node [shape=circle,fill=black,scale=.5,draw]{};
		\draw (2,1) node [shape=circle,fill=black,scale=.5,draw]{};
		\draw (2,3) node [shape=circle,fill=black,scale=.5,draw]{};
		\draw (3,2) node [shape=circle,fill=black,scale=.5,draw]{};	
\end{tikzpicture}
\caption{A closed point and its adjacency set}
\label{fig:purec}
\end{figure}

\begin{figure}
\centering
\begin{tikzpicture}

	\draw[lightgray, thin] (2,1)--(1,2)--(2,3)--(3,2)--(2,1)--(2,3)--(1,2)--(3,2);
	\draw (2,2) node [shape=circle,fill=black,scale=.5,draw]{};
	\draw (1,2) node [shape=circle,fill=white,draw]{};
	\draw (3,2) node [shape=circle,fill=white,draw]{};
	\draw (2,1) node [shape=rectangle,fill=white,draw]{};
	\draw (2,3) node [shape=rectangle,fill=white,draw]{};

\end{tikzpicture}
\caption{A mixed point and its adjacency set}
\label{fig:amixed}
\end{figure}
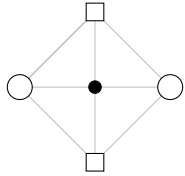

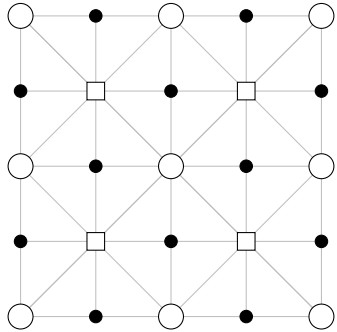
\begin{figure}
\centering
\begin{tikzpicture}
\foreach \x in {1,2,3,4,5}
	{
		\draw[lightgray, thin] (\x,1)--(\x,5);
		\draw[lightgray, thin] (1,\x)--(5,\x);
	}
\foreach \x in {1,3,5}
	{
		\draw[lightgray, thin] (\x,1)--(5,6-\x);
		\draw[lightgray, thin] (\x,5)--(5,\x);
		\draw[lightgray, thin] (1,\x)--(6-\x,5);
		\draw[lightgray, thin] (\x,1)--(1,\x);
	}

\foreach \x in {1,3,5}
	\foreach \y in {1,3,5}
	{
		\draw (\x,\y) node [shape=circle,fill=white,draw]{};
	}
\foreach \x in {2,4}
	\foreach \y in {2,4}
	{
		\draw (\x,\y) node [shape=rectangle,fill=white,draw]{};
	}
\foreach \x in {2,4}
	\foreach \y in {1,3,5}
	{
		\draw (\x,\y) node [shape=circle,fill=black,scale=.5,draw]{};
		\draw (\y,\x) node [shape=circle,fill=black,scale=.5,draw]{};
	}

\end{tikzpicture}
\caption{A $5 \times 5$ digital plane}
\label{fig:plane}
\end{figure}

To account for the parity of the points, we present the following:

\begin{thm} 
\label{thm:dswap}
Consider a digital plane $\D$.
Its dual space $\op{\D}$ is also a digital plane, whose open and closed points have been swapped.
\end{thm}
\begin{proof}
Consider two COTS $X,Y$ and their resulting digital plane $\D=X \times Y$.
As mentioned above, $\op{X}$ and $\op{Y}$ are also COTS, so $\op{X}\times\op{Y}$ is also a digital plane.
By Proposition \ref{prop:dualx}, $\op{X}\times\op{Y} = \op{(X \times Y)}=\op{\D}$.
By the definition of dualizing, the open sets of $\D$ are the closed sets of $\op{\D}$, so the open points of one are the closed points of the other.
\end{proof}

Just as paths and arcs in a space $X$ are the images of maps from $I$ to $X$, we have analogs for finite spaces.

\begin{defn}
\label{def:COTSpa}
If $Y$ is a topological space, a {\bf COTS-arc} in $Y$ is a homeomorphic image of a COTS in $Y$, 
and a {\bf COTS-path} is a continuous mapping of a COTS into $Y$.
\end{defn}

\begin{ex}\label{ex:acute}
Figure \ref{COTSpath} shows a mapping of a finite COTS into $\D$ whose image is not a COTS-arc, however, it is a COTS-path.
Recall that in a COTS, any point that is not an endpoint has precisely two neighbors.
Enumerating the points from left to right $c_0,\hdots c_6$, observe that $c_2$ has four neighbors in the COTS-path: $c_1, c_3,c_4,c_5$.
In Proposition \ref{prop:path-arc}, we prove that any COTS-path contains a COTS-arc as a subset.
Applying this proposition to the COTS-path in Figure \ref{COTSpath} would yield the COTS-arc $\{c_0,c_1,c_2,c_5,c_6\}$, shown in Figure \ref{COTSarc}.
Note how the $45^\circ$ turn at $c_3$ in Figure \ref{COTSpath} is not possible for a COTS-arc.

\begin{figure}
\centering
\begin{minipage}{.5\textwidth}
\centering
\begin{tikzpicture}
\foreach \x in {1,2,3,4,5}
	{
		\draw[lightgray, thin] (\x,1)--(\x,5);
		\draw[lightgray, thin] (1,\x)--(5,\x);
	}
\foreach \x in {1,3,5}
	{
		\draw[lightgray, thin] (\x,1)--(5,6-\x);
		\draw[lightgray, thin] (\x,5)--(5,\x);
		\draw[lightgray, thin] (1,\x)--(6-\x,5);
		\draw[lightgray, thin] (\x,1)--(1,\x);
	}

\foreach \x in {1,3,5}
	\foreach \y in {1,3,5}
	{
		\draw (\x,\y) node [shape=circle,fill=white,draw]{};
	}
\foreach \x in {2,4}
	\foreach \y in {2,4}
	{
		\draw (\x,\y) node [shape=rectangle,fill=white,draw]{};
	}
\foreach \x in {2,4}
	\foreach \y in {1,3,5}
	{
		\draw (\x,\y) node [shape=circle,fill=black,scale=.5,draw]{};
		\draw (\y,\x) node [shape=circle,fill=black,scale=.5,draw]{};
	}

\draw[red, ultra thick] (1,4)--(2,4)--(3,3)--(4,4)--(4,3)--(4,1);
\end{tikzpicture}
  \caption{A COTS-path}
  \label{COTSpath}
\end{minipage}%
\begin{minipage}{.5\textwidth}
  \centering
  \begin{tikzpicture}
\foreach \x in {1,2,3,4,5}
	{
		\draw[lightgray, thin] (\x,1)--(\x,5);
		\draw[lightgray, thin] (1,\x)--(5,\x);
	}
\foreach \x in {1,3,5}
	{
		\draw[lightgray, thin] (\x,1)--(5,6-\x);
		\draw[lightgray, thin] (\x,5)--(5,\x);
		\draw[lightgray, thin] (1,\x)--(6-\x,5);
		\draw[lightgray, thin] (\x,1)--(1,\x);
	}

\foreach \x in {1,3,5}
	\foreach \y in {1,3,5}
	{
		\draw (\x,\y) node [shape=circle,fill=white,draw]{};
	}
\foreach \x in {2,4}
	\foreach \y in {2,4}
	{
		\draw (\x,\y) node [shape=rectangle,fill=white,draw]{};
	}
\foreach \x in {2,4}
	\foreach \y in {1,3,5}
	{
		\draw (\x,\y) node [shape=circle,fill=black,scale=.5,draw]{};
		\draw (\y,\x) node [shape=circle,fill=black,scale=.5,draw]{};
	}

\draw[blue, ultra thick] (1,4)--(2,4)--(4,2)--(4,1);
\end{tikzpicture}
  \caption{A COTS-arc that is a subset of the COTS-path in Figure \ref{COTSpath}}
  \label{COTSarc}
\end{minipage}
\end{figure}
\end{ex}

The notion of COTS-arcs lends itself to that of COTS-Jordan curves, which is a digital Jordan curve in the Khalimsky plane.
(See \cite{Kiselman2000}, for example.)

\begin{defn}
If $J \subseteq \D$ has $|J|\geq 4$, and $J-\{j\}$ is a COTS-arc for all $j \in J$, then $J$ is a {\bf COTS-Jordan curve}.
\end{defn}

The main theorem of \cite{Khalimsky1990a} is that the complement of a digital Jordan curve which does not meet the border of a digital plane has two components: the component that touches the border is called the outside or {\bf exterior}, and the other component is called the inside or {\bf interior}, which we denote $\jext{J}$ and $\jint{J}$, respectively. 
Note that $\jext{J} \cap \jint{J} = \emptyset$, and $\D = \jext{J} \sqcup J \sqcup \jint{J}$.
The notion of ``interior'' used in this article is not the traditional definition of interior.
For example, if $c \in \D$ is a closed point of the Khalimsky digital plane, then $A(c)$ is a Jordan curve by Lemma 5.2(b) of \cite{Khalimsky1990a}.
Then the interior $\jint{A(c)}=\{c\}$ is a closed subset of $\D$.
Furthermore, COTS-Jordan curves are not closed subsets of $\D$, as is true in the Euclidean case.

\subsubsection{Other Digital Topologies}
\label{sec:ottop}

Jordan curve theorems have been defined for more than just Khalimsky's digital plane.
In \cite{Slapal2006}, \v{S}lapal explores Jordan curves in digital planes that have topologies different from that in \cite{Khalimsky1990a}.
The Jordan curve theorem has been proven many times in a variety of digital settings (See \cite{Kong1992}, \cite{Rosenfeld2007}, \cite{Kiselman2000}, \cite{Slapal2006}).
Kong et al. prove in \cite{Kong1992} that the digital fundamental groups of digital Jordan curves are actually isomorphic to the fundamental groups of their continuous counterparts.

Digital planes can also be interpreted as subsets of $\zz^2$.
To see how this aligns with the Khalimsky topology, one can treat $\zz$ as a COTS of infinite length whose minimal open sets are given by
\[
x^\downarrow = \left\{\begin{array}{ll}
\{x-1,x,x+1\}, & x\textrm{ is even} \\
\{x\},  & \textrm{otherwise.}
\end{array}\right.
\]
By this interpretation, for example, $\{0\} \in \zz$ is a closed point of $\zz$ with the COTS topology.
Taking the product $\zz \times \zz$ with this topology yields the Khalimsky topology on $\zz^2$.

Given two distinct points $\set{(x_1,y_1),(x_2,y_2)}\in \zz^2$, they are \textbf{4-adjacent} if exactly one of the following hold:
\begin{itemize}
    \item $\abs{x_2-x_1}=0$ and $\abs{y_2-y_1}=1$, or
    \item $\abs{y_2-y_1}=0$ and $\abs{x_2-x_1}=1$.
\end{itemize}
That is, $(x_2,y_2)$ is either above, below, left, or right of $(x_1,y_1)$.
They are \textbf{8-adjacent} if:
\begin{itemize}
    \item $\abs{x_2-x_1}\in\{0,1\}$ and $\abs{y_2-y_1} \in \{0,1\}$, and
    \item $\abs{x_2-x_1}\neq 0$ or $\abs{y_2-y_1}\neq 0$.
\end{itemize}
Equivalently, $(x_2,y_2)$ is either a 4-neighbor or a diagonal neighbor of $(x_1,y_1)$.
Given a topology $\tau$ on a digital plane $\zz^2$ and a $\kappa \in \{4,8\}$, a set $X \subset (\zz^2,\tau)$ is \textbf{$\kappa$-connected} if for all $x,y\in X$, there exists a sequence of points connecting $x$ to $y$ such that each consecutive pair of points is both $\kappa$-adjacent and adjacent with respect to $\tau$.
Such a sequence is called a $\kappa$-path.
Notice that the points in Khalimsky's plane can be either $4$-connected or $8$-connected.
If $x$ and $y$ are $\kappa$-adjacent, we may also say $x \in A_\kappa(y)$ and $y \in A_\kappa(x)$.
This notion can be generalized, as in \cite{Boxer2017a}: 
for points $p$ and $q$ in an $n$-dimensional integer lattice $\zz^n$, $p=(p_1,\hdots,p_n)$ and $q=(q_1,\hdots,q_n)$ are $c_u$-adjacent if there are at most $u$ coordinates $i$ such that $\abs{p_i-q_i}=1$, and $\abs{p_j-q_j}=0$ for all $j\neq i$.
In \cite{Eckhardt2003}, they list two conditions that summarize general intuition about ``nearness'' in a digital plane:
\begin{enumerate}
    \item If a set in $\zz^2$ is 4-connected, then it is topologically connected.
    \item If a set in $\zz^2$ is not 8-connected, then it is not topologically connected.
\end{enumerate}

In Section 4 of \cite{Chassery1979}, they show that there is no topology on $\zz^2$ such that every connected set is 8-connected.
Indeed, the Jordan curve theorem of \cite{Rosenfeld1979} is only true when the Jordan curve and its complement have different topologies:
if the Jordan curve is $\kappa$-connected, then its complement consists of two $4+(\kappa \bmod{8})$-connected components.
In Theorem 3.1 of \cite{Eckhardt2003}, they prove that the Khalimsky topology defined above and the Marcus-Wyse topology defined below are the only topologies on $\zz^2$ that satisfy these intuitive conditions.

Perhaps the first digital topology was described in \cite{Marcus1970}.
In 1970, Marcus proposed the following question in The American Mathematical Monthly: ``Is it possible to topologize the integers in such a way that the connected sets are the sets of consecutive integers? Generalize to the lattice points of $n$-space.''
In this context, two points in $\zz^n$ are ``consecutive'' if they differ by 1 in one coordinate and agree on all other coordinates.
The Cleveland State University Problem Solving Group, advised by Frank Wyse, devised the following solution.
The resulting basis for the topology on $\zz^2$ is of the form
\[(x,y)^\downarrow = \left\{\begin{array}{ll}
     \{(x,y), (x\pm1,y),(x,y\pm1)\},& x+y \textrm{ is even} \\
     \{(x,y)\},& \textrm{else.} 
\end{array}\right.,
\]
where $(x,y) \in \zz^2$.
A portion of $\zz^2$ with the {\bf Marcus-Wyse topology} is shown in Figure \ref{fig:marcus}. 
The squares are the closed points, and the circles are the open points.
Notice that every connected set in $\zz^2$ with the Marcus-Wyse topology is $4$-connected.
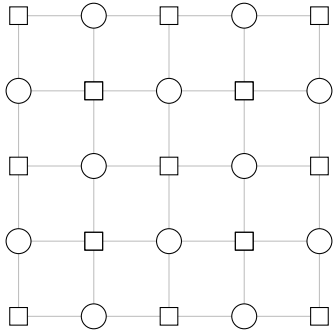
\begin{figure}
    \centering
    \begin{tikzpicture}
    \draw[lightgray, thin] (0,0) grid (4,4);
    \foreach \x in {0,2,4}{
        \foreach \z in {0,2,4}{
            \draw (\x,\z) node [shape=rectangle,fill=white,draw]{};
        }
        \foreach \y in {1,3}{
            \draw (\x,\y) node [shape=circle,fill=white,draw]{}; 
            \draw (\y,\x) node [shape=circle,fill=white,draw]{};
            \foreach \w in {1,3}{
                \draw (\y,\w) node [shape=rectangle,fill=white,draw]{};
            }
        }
    }
    \end{tikzpicture}
    \caption{A subset of $\zz^2$ with the Marcus-Wyse topology}
    \label{fig:marcus}
\end{figure}

In \cite{Slapal2006}, \v{S}lapal introduces topologies on $\zz^2$ that allow Jordan curves to have features that Khalimsky's cannot.
In particular, Jordan curves in \cite{Slapal2006} may turn at an angle of $\frac{\pi}{4}$.
\v{S}lapal denotes one such space as $(\zz^2,w)$, where $w: \mathcal{P}(\zz^2) \to \mathcal{P}(\zz^2)$ is the Kuratowski closure operator that maps a set to its closure.
A tile of $(\zz^2,w)$ is shown in Figure \ref{fig:Zw}.
In this figure, let us denote the bottom-left point by $(0,0) \in \zz^2$ such that the top-right point is $(4,4)\in\zz^2$.
See, for example, that $C:=\{(1,0),(0,0),(1,1)\}$ is a COTS-arc in $(\zz^2,w)$, because each of $(1,0)$ and $(1,1)$ have exactly one neighbor in $C$, and $(0,0)$ has exactly two neighbors in $C$.
\v{S}lapal defines a second topology on the digital plane denoted $(\zz^2,\hat{w})$, a tile of which is shown in Figure \ref{fig:Zw'}.
In both of Figures \ref{fig:Zw} and \ref{fig:Zw'}, the squares represent closed points and the circles reperesent open points.

\begin{figure}
    \centering
    \begin{tikzpicture}
    \draw[lightgray, thin] (0,0)--(0,4)--(4,4)--(4,0)--(0,0)--(1,1)--(3,1)--(3,3)--(1,3)--(1,1)--(4,4);
    \draw[lightgray, thin] (0,4)--(4,0);
    \draw[lightgray, thin] (0,1)--(1,2)--(0,3);
    \draw[lightgray, thin] (4,1)--(3,2)--(4,3);
    \draw[lightgray, thin] (1,4)--(2,3)--(3,4);
    \draw[lightgray, thin] (1,0)--(2,1)--(3,0);
        \foreach \x in {1,3}{
        \draw (\x,2) node [shape=circle,fill=white,draw]{}; 
        \draw (2,\x) node [shape=circle,fill=white,draw]{};
            \foreach \y in {0,1,3,4}{
            \draw (\x,\y) node [shape=rectangle,fill=white,draw]{};
            \draw (\y,\x) node [shape=rectangle,fill=white,draw]{};
            }
        }
        \foreach \x in {0,2,4}{
        \foreach \y in {0,2,4}{
            \draw (\x,\y) node [shape=circle,fill=white,draw]{};
            \draw (\y,\x) node [shape=circle,fill=white,draw]{}; 
        }}
    \end{tikzpicture}
    \caption{A tile of the connectedness graph of $(\zz^2,w)$}
    \label{fig:Zw}
\end{figure}
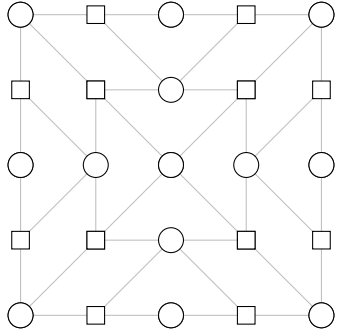

\begin{figure}
    \centering
    \begin{tikzpicture}
    \draw[lightgray, thin,step=2] (0,0) grid (4,4);
    \draw[lightgray, thin] (0,0)--(4,4);
    \draw[lightgray, thin] (0,4)--(4,0);
    \draw[lightgray, thin] (0,2)--(2,4)--(4,2)--(2,0)--(0,2);
        
        \foreach \x in {1,3}{
            \foreach \y in {0,1,...,4}{
            \draw (\x,\y) node [shape=circle,fill=white,draw]{};
            \draw (\y,\x) node [shape=circle,fill=white,draw]{};
            }
        }
        \foreach \x in {0,2,4}{
        \foreach \y in {0,2,4}{
            \draw (\x,\y) node [shape=rectangle,fill=white,draw]{};
            \draw (\y,\x) node [shape=rectangle,fill=white,draw]{}; 
        }}
    \end{tikzpicture}
    \caption{A tile of the connectedness graph of $(\zz^2,\hat{w})$}
    \label{fig:Zw'}
\end{figure}
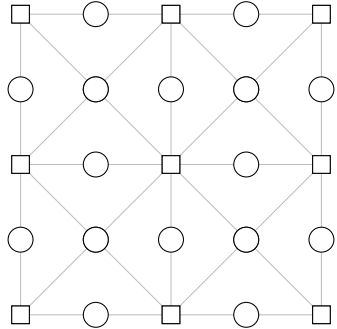

It is also worth noting that in \cite{Slapal2009}, \v{S}lapal discusses a pretopology $\paren{\zz^2,u}$ in which every cycle is a Jordan curve.
Recall that for a {\bf pretopology} on $X$, $p:\mathcal{P}(X) \to \mathcal{P}(X)$, the following are fulfilled:
\begin{enumerate}
    \item $p\emptyset=\emptyset$
    \item $A \subseteq pA$ for all $A \subseteq X$
    \item $p(A\cup B) = pA \cup p(B)$ for all $A,B \subseteq X$
\end{enumerate}
The pretopology $\paren{\zz^2,u}$ is given by the following, for any $z=(x,y) \in \zz^2$.
\[u(z) = \left\{ \begin{array}{rl}
    A_4(z), & (x\bmod{2})=(y\bmod{2})=1 \textrm{ or } (x,y)=(4k+2\ell,2\ell+2), k,\ell \in \zz\\
    A_8(z), & (x,y)=(4k+2\ell,2\ell), k,\ell \in \zz \\
    \{z\}, & \textrm{otherwise.}
    \end{array}
\right.\]
This fails to be a topology because $(7,4) \not\in u(6,2) \neq u(u(6,2))\ni (7,4)$, for example.

In \cite{Ptak1997} and \cite{Eckhardt2003}, they discuss the possibility of a $6$-adjacency structure on $\zz^2$.
In such a structure, every point of $\zz^2$ has 6 neighbors, as shown in Figure \ref{fig:6a}.
Ptak, et al. prove in Theorem 4 of \cite{Ptak1997} that there is no topology on $\zz^2$ that is compatible with 6-adjacency.

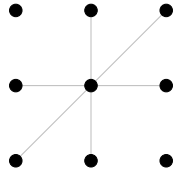
\begin{figure}
    \centering
    \begin{tikzpicture}
    \draw[lightgray, thin] (1,1)--(3,3);
    \draw[lightgray, thin] (2,1)--(2,3);
    \draw[lightgray, thin] (1,2)--(3,2);
        \foreach \x in {1,2,3}{
        \foreach \y in {1,2,3}{
            \draw (\x,\y) node [shape=circle, fill=black,scale=.5,draw]{};
        }}
    \end{tikzpicture}
    \caption{The 6-neighbors of a point in $\zz^2$}
    \label{fig:6a}
\end{figure}

For each of the digital planes described above, digital Jordan curves all have some similar properties.
In particular, they are all cycles in the graph determined by the underlying topology.
The converse, however, is not true; there are many cycles in each of Figures \ref{fig:marcus}, \ref{fig:Zw}, and \ref{fig:Zw'} that do not divide the plane into interior and exterior regions.
In Section \ref{sec:diftop}, we will prove that none of these topologies are suitable for our research.

It is worth noting that there is a notion of {\bf digital topological complexity}, defined in \cite{Karaca2018}.
The digital topology used in that paper is that of \cite{Boxer1999}.
Their digital topology is not truly a topology; it is defined by its adjacency relations, such as the $\kappa$-adjacency described earlier in this section.
Furthermore, that paper describes the topological complexity of digital images themselves, and not a space of digital images, which is the focus of this article.

\section{Distance Functions in Finite Spaces}
\label{chap:COTS}

There are very few results about metrics (i.e., distance functions) in finite topological spaces.\footnote{There are distance functions in graph theory, however, graphs make up a small portion of finite spaces.}
For the most part, this is because finite metric spaces are discrete.
To see this, consider the following.
In a metric space $X$ with distance function $d$, open sets of the form
\[D(x,\varepsilon) = \set{y \in x \mid d(x,y)<\varepsilon}\]
form a basis for the topology on $X$, where $\varepsilon \in \rr_{>0}$.
Suppose $X$ is a connected finite topological space, and let $x\in X$.
Take $\varepsilon = \min\set{d(x,y)\mid y \in A(x)}$.
Then the only point in $D\paren{x,\frac{\varepsilon}{2}}$ is $x$ itself.
Hence, every point of $X$ is open, and it has the discrete topology.

\subsection{Paths and Arcs in Finite Spaces}

Recall from Definition \ref{def:COTSpa} the definitions of COTS, COTS-paths, and COTS-arcs.
Throughout this document, any COTS, COTS-path, or COTS-arc $C$ denoted $\set{c_0,\hdots,c_n}$ is assumed to have a preorder in that there is a fence from $c_0$ to $c_n$, e.g., $c_i \lessgtr c_{i+1}$ and $c_i,c_j$ are not comparable for $|i-j|>1$. 
The endpoints of the $C$ are $c_0$ and $c_n$. 
When $C$ is already defined, a subinterval $C' \subseteq C$ given by
\[\cots{c_i,c_j} = \set{c_\ell \mid i \leq \ell \leq j}\]
 is the unique COTS/COTS-path/COTS-arc whose endpoints are $c_i$ and $c_j$ (see Theorem 3.2(b) of \cite{Khalimsky1990a}).
The sub-interval inherits the subspace topology.
Figure \ref{cots} depicts a COTS of length nine.
Adhering to the visuals used in \cite{Khalimsky1990a}, the circles are open points, and the squares are closed points.

In \cite{Tanaka2018}, they prove that $\textrm{CC}(X) = \tc{}{X}$ for a finite space $X$.
For this reason, many of our computations will use the interval $\cots{0,1}$ in place of a finite fence (i.e., a finite COTS).

\begin{prop}\label{COTSI}
Every COTS admits a parameterization by $[0,1]$.
\end{prop}
\begin{proof}
Let $C=\set{c_0,c_1,\hdots,c_n}$ be a COTS such that 
\[A(c_i) = \set{c_{i-1},c_{i+1}\}\cap C\textrm{ for }i \in \{0,1,\hdots,n}.\]
We start to define a path $f:[0,1]\to C$ on its midpoints $\cots{c_1,c_{n-1}}$:

\[f^{-1}(c_i):=
\left\{\begin{array}{rl}
    \paren{\frac{i}{n+1},\frac{i+1}{n+1}},&c_i \textrm{ is open, and}\\
    \cots{\frac{i}{n+1},\frac{i+1}{n+1}},&c_i \textrm{ is closed.} 
\end{array}
\right.
\]
For the endpoints $c_0$ and $c_n$,
\[
f^{-1}(c_0):=
\left\{\begin{array}{lr}
    \left[0,\frac{1}{n+1}\right),&c_0 \textrm{ open}\\
    \cots{0,\frac{1}{n+1}},&c_0 \textrm{ closed} 
\end{array}
\right.
\textrm{ and }
f^{-1}(c_n):=
\left\{\begin{array}{rr}
    \left(\frac{n}{n+1},1\right],&c_n \textrm{ open}\\
    \cots{\frac{n}{n+1},1},&c_n \textrm{ closed} 
\end{array}
\right.
.\]
By construction, $f$ is defined on all of $[0,1]$, and the preimages of open points of $C$ are open in $[0,1]$.
\end{proof}

Our intuition about continuous paths extends to COTS-paths, as will be shown throughout the next few sections.

\begin{prop}\label{prop:concat-paths}
The union of two COTS-paths that share an endpoint is a COTS-path.
\end{prop}
\begin{proof}
Let $\alpha,\beta \subseteq X$ be COTS-paths such that the end point of $\alpha$ is the start point of $\beta$.
Since $\alpha$ is a COTS-path, it is the image of a continuous map $f:C \to X$ where $C = \set{c_0,\hdots,c_n}$ is a COTS.
Suppose without loss of generality that $c_0$ is an open point of $C$.
Similarly, $\beta$ is the image of a continuous map $g:D \to X$, where $D = \set{d_0,\hdots,d_m}$ is a COTS.
Note that $f(c_n) = g(d_0)$ by assumption.

If $c_n$ and $d_0$ do not have the same parity, 
take \[E:=\set{e_0,\hdots,e_{n+m+1}}\] to be a COTS whose first point is open.
Then $e_{n+1}$ has the same parity as $d_0$.
By abuse of notation, we use $f \circ g$ for concatenation of functions, rather than composition.
We define $f \circ g : E \to X$ by:
$$
f \circ g(e_i) = \left\{
\begin{array}{ll}
f(c_i) & 0\leq i \leq n\\
g\left(d_{i-(n+1)}\right) & n+1 \leq i \leq n+m+1
\end{array}
\right..
$$

If $c_n$ and $d_0$ have the same parity, 
take $E:=\cots{e_0,\hdots,e_{n+m}}$ to be a COTS whose first point is open.
Then $e_{n+1}$ has the same parity as $d_1$.

We define $f \circ g : E \to X$ by:
$$
f \circ g(e_i) = \left\{
\begin{array}{ll}
f(c_i) & 1\leq i \leq n\\
g(d_{i-n}) & n+1\leq i \leq n+m
\end{array}
\right..
$$

Each construction is a continuous mapping from a COTS into the digital plane whose image is $\alpha \cup \beta \subseteq X$.
In the former case, the last point of $C$ and the first point of $D$ do not have the same parity.
Since the points of a COTS must alternate, we concatenate $C$ and $D$ to create $E$ such that $E$ is still a COTS whose points alternate between open and closed.
In the latter case, the last point of $C$ and the first point of $D$ have the same parity.
This implies that the penultimate point of $C$ and the next point of $D$ also have the same parity.
Then we can create $E = C \cup D/(c_n \sim d_0)$ such that the resulting space is still a COTS.
\end{proof}

In Lemma 1 of \cite{Eckhardt1994}, Eckhardt proves that every $\kappa$-path in $\zz^2$ contains a $\kappa$-arc as a subset, where $\kappa\in \{4,8\}$.
While the proof of that lemma is graph-theoretic, we can use a similar approach to generalize this result to any finite topological space.

\begin{prop}
\label{prop:path-arc}
Every COTS-path contains, as a subset, a COTS-arc with the same start and end points.
\end{prop}

\begin{proof}
Let $C = \set{c_0,\hdots,c_n} \subseteq X$ be a COTS-path that is the continuous image of a finite COTS into the digital plane.
Let $i \in \{0,\hdots,n-1\}$ be the lowest index such that if $|A(c_i)\cap C| > 1$ then $i=0$, or $|A(c_i) \cap C|>2$.
This index $i$ marks the start of a ``loop'' at $c_i$.
Note that loop cannot start at $c_n$, or else $|C| > n+1$.
Let $j \in \{i+1,\hdots,n\}$ be the highest index such that $c_j \in A(c_i)\cap C$.
This index $j$ marks the end of the loop at $c_i$.
We eliminate the extra points between $c_i$ and $c_j$ to form $C'$.
Take $C':=\cots{c_1,c_i} \cup \cots{c_j,c_n}$.
If $i=0$, then $|A(c_0) \cap C'| = |\{c_j\}|=1$, and if $i>1$,
then $|A(c_i)\cap C'| = |\{c_{i-1},c_j\}| = 2$, so the loop at $c_i$ has been removed.
Repeating the process for all subsequent values of $i$ such that $|A(c_i)\cap C|>2$ yields a COTS-arc.

Note also that at each step, we remove at least one point from $C$, so $|C'| \leq |C|$.
\end{proof}

We suspect that a technique similar to that used in Algorithm \ref{alg:shrink} could yield a homotopy between parameterizations of $C$ and $C'$, however, we are unaware of such an algorithm at this time.

\subsubsection{The COTS-distance Function}

The distance function we have defined for finite spaces is very natural, and is based on the length of the shortest COTS-arc connecting two points.

\begin{defn}
Given $x,y \in X$ for $X$ a finite topological space and $A \subseteq X$ a subspace, we assign a {\bf COTS-distance function} $d_A:X \times X \to \mathbb{Z}_{\geq 0}$ as one less than the magnitude of a shortest COTS-arc in $A$ whose endpoints are $x$ and $y$.
If $x,y$ are not connected by a path in $A$, set $\dist{A}{x}{y} = \infty$.
If the $A$ is omitted from notation, then $A=X$.
Note that the shortest COTS-arc between two points may not be unique.

Given two subsets $A,B \subseteq X$ and $x \in X$, we define
\[d_X(x,A) = \max_{a \in A}\set{\dist{X}{x}{a}},\] 
and 
\[\dist{X}{A}{B} = \max_{a \in A, b \in B}\set{\dist{X}{a}{B},\dist{X}{b}{A}}.\]
\end{defn} 

These definitions are natural, as they mimic the definitions of distances between sets in Hausdorff spaces.

\begin{prop}\label{metric}
The COTS-distance function $d$ is a metric.
\end{prop}

\begin{proof}

Let $x,y,z \in X$ a finite topological space equipped with distance function $d$ such that $x,y$, and $z$ are path-connected.

\begin{enumerate}
\item As a subset of $X$, $\{x\}$ is a COTS-arc of length one in $X$ and $1-1=0$, so $d(x,x)=0$.

\item If $C = \{c_0, c_1\hdots,c_n\} \subset X$ is a shortest COTS-arc whose endpoints are $x$ and $y$, then it is also a shortest COTS-arc whose endpoints are $y$ and $x$, so $d(x,y) = d(y,x) = n$.

\item Let $C = \{x, c_1,\hdots,c_{n-1},y\}$ be a COTS-arc with endpoints $x,y$ such that $|C|=n+1$.
Similarly, let $D = \{y,d_2,\hdots,d_{m-1},z\}$ be a COTS-arc of length with endpoints $y$ and $z$ such that $|D|=m+1$.
By Proposition \ref{prop:concat-paths}, $C \cup D$ is a COTS-path.
By Proposition \ref{prop:path-arc}, $C \cup D$ contains a COTS-arc $E$ as a subset with $|E| \leq |C|+|D|$.
Hence, $d(x,z) \leq d(x,y)+d(y,z)$.
\end{enumerate}

Hence $d$ is a metric.
\end{proof}

\begin{defn}
Given a finite $T_0$ space $X$ with COTS-distance function $d$, take
\[S_X(x,n) = \{y \in X \mid d_X(x,y)=n\}\]
to be a finite analog of the hollow {\bf sphere of radius $n$}, and 
\[D_X(x,n) = \{y \in X \mid d_X(x,y) \leq n\}\]
to be a finite analog of the solid {\bf disk of radius $n$}.
Futhermore, 
\[\textrm{diam}(X) = \max_{x,y \in X}\left\{d_X(x,y)\right\}\]
is the {\bf diameter} of $X$.
\end{defn}

\begin{prop}
\label{cont_dist}
Let $X$ be a finite topological space.
If $C=\{c_0,c_1,\hdots,c_n\} \subseteq X$ is a shortest COTS-arc containing $c_0$ and $c_n$, then $d(c_0,c_i)=i$ for $i \in \{0,1,\hdots,n\}$.
\end{prop}
\begin{proof}
Suppose for sake of contradiction that $d(c_0,c_i)=j$ for some $j \neq i$.
Then there exists a shortest COTS-arc $C'$ containing $c_0$ and $c_i$ such that $|C'|=j+1$.
Because $|\{c_0,c_1,\hdots,c_i\}|=i+1$, $j \leq i$ or else $C'$ would not be minimal.
Consider $D:=C' \cup \{c_{i+1},c_{i+2},\hdots,c_n\}$.
If $j < i$, then $|D|=j+1+(n-i) < n+1 =|C|$, a contradiction.
Hence $d(c_0,c_i)=i$.
\end{proof}

\begin{prop}
\label{sub_dist} 
Let $X$ be a finite topological space and $A \subseteq X$ a subspace.
Given $x,y \in A \subseteq X$, $d_A(x,y) \geq d_X(x,y)$.
\end{prop}
\begin{proof}
Suppose for sake of contradiction that $d_A(x,y) = m$ and $d_X(x,y)=n$ with $m < n$.
Let $C \subseteq A$ and $D\subseteq X$ be COTS-arc realizing those distances, respectively.
If $C \subseteq A \subseteq X$, then $C$ is a COTS-arc in $X$ of length $m$ whose endpoints are $x$ and $y$, contradicting the assumption that $d_X(x,y) >d_A(x,y)$.

Furthermore, if $d_X(x,y)=\infty$, then $x$ and $y$ are in different components of $X$, so they are in different components of $A$, so $d_A(x,y)=\infty$ as well.
\end{proof}

\subsubsection{COTS-paths in the Khalimsky Plane}
\label{sec:COTS-paths}

In this section we prove some properties of COTS-paths that are specific to Khalimsky's digital plane.

\begin{prop}
\label{n-ball}
Let $\D$ be a sufficiently large Khalimsky digital plane, and let $p=(i,j)\in\D$ be a pure point represented by integer coordinates.
If $d(p,q):= d_{\D}(p,q)\leq n$ for some $q =(k,\ell )\in \D$, then $|k-i|\leq n$, and $|\ell-j| \leq n$.
\end{prop}
\begin{proof}
We will prove this via induction.
If $d(p,q)=0$, then $p=q=(i,j)$ by Proposition \ref{metric}.
If $d(p,q)=1$, then $p$ and $q$ are adjacent, so $|k-i|\leq 1$ and $|\ell-j|\leq 1$.
This can also be seen explicitly in Lemma 4.2 of \cite{Khalimsky1990a}, and in Figures \ref{fig:pure} and \ref{fig:purec}.

Suppose that $d(p,q) \leq n-1$ implies $|k-i|\leq n-1$ and $|\ell-j| \leq n-1$.
Consider $q'=(k',\ell')$ such that $d(p,q')=n$.
Then there exists a (not necessarily unique) shortest COTS-arc $C=\{p,c_1,\hdots,c_{n-1},q'\}$ connecting $p$ and $q'$.
By Proposition \ref{cont_dist}, $d(p,c_{n-1})=n-1$, so $c_{n-1}=(x,y)$ satisfies $|x-i|\leq n-1$, and $|y-j|\leq n-1$ by the inductive hypothesis.
Since $c_n$ and $q'$ are adjacent, $|k'-x|\leq 1$ and $|\ell'-y|\leq 1$.
Then
\begin{eqnarray*}
|k'-i|&=&|k'-x+x-i|\\
&=&|(k'-x) + (x-i)|\\
&\leq &|k'-x| + |x-i|\\
&\leq& 1 + (n-1)\\
&=&n.
\end{eqnarray*}
Similarly, $|\ell'-j| \leq n$.
\end{proof}

In \cite{Chassery1979}, the author points out that the $d_1$ metric can be used on $4$-connected digital planes, and the $d_\infty$ metric can be used on $8$-connected digital planes.
In this context, given two points $p=(i,j)$ and $q=(k,\ell)$ in $\zz^2$, $d_1(p,q) = \abs{k-i}+\abs{\ell-j}$ and $d_\infty(p,q)=\max\set{\abs{k-i},\abs{\ell-j}}$.
Since Khalimsky's digital plane is not homogeneously connected, the distance between points is not immediately calculable.

\begin{prop}
\label{pure_dist_max}
If $p=(i,j)$ and $p'=(i',j')$ are pure points in $\D$ represented by integer coordinates, then $\dist{\D}{p}{p'} = \max\set{|i'-i|,|j'-j|}$.
\end{prop}
The proof idea here is to travel from $p$ to $p'$ as far as possible while only travelling diagonally, and then to move either horizontally or vertically to make up the remaining distance.
\begin{proof}
Let $m :=\min\{|i'-i|,|j'-j|\}$ and $m':=\max\{|i'-i|,|j'-j|\}$.
We can construct a COTS-arc $C=\set{c_0,c_1,\hdots,c_m}$ with start-point $p=c_0$ given in integer coordinates by
\[c_k = \paren{i_k,j_k} = \paren{i+\frac{(i'-i)k}{|i'-i|},j+\frac{(j'-j)k}{|j'-j|}},\]
for $k \in \{0,1,\hdots,m\}$.
At $k=m$, at least one of either $i'=i_m$ or $j'=j_m$.
Suppose without loss of generality that $i'=i_m$.
If $m=m'$, then $j'=j_m$ and the construction of a COTS connecting $p$ and $p'$ stops here; notice that all of the points in $C$ are pure and alternate between open and closed.

Otherwise, assume $j' \neq j_m$. 
We can construct a second COTS-arc \[C' = \set{c'_0,c'_1,\hdots,c'_{m'-m}}\] with start-point $c_m =c_0'$ and end-point $p'=c_{m'-m}'$ given by
\[c'_k = \left(i', j_m + \frac{(j'-j)k}{|j'-j|}\right),\]
for $k \in \{0,1,\hdots,m'-m\}$.
Since $c'_0=c_m$ and $c_{m'-m}'=(i',j')$, $C \cup C'$ exhibits a COTS-arc of length $m'$ connecting $p$ and $p'$.
By Proposition \ref{prop:path-arc}, this yields a COTS-arc with endpoints $p$ and $p'$ whose length is less-than-or-equal-to $m'$.
Hence, $d(p,p') \leq m'$.

We will show next that this length is necessary.
If $d(p,p') < m'$, then $|i'-i| \leq m'-1$ and $|j' - j| \leq m'-1$ by Proposition \ref{n-ball}.
By the definition of $m'$, either $|i'-i|= m'$, or $|j'-j|= m'$, but neither of these values are in the ranges mentioned above, a contradiction.
\end{proof}

Note that result of Proposition \ref{pure_dist_max} may not be true for two pure points $p,q \in A \subseteq \D$, where $A$ is a subspace.
Futhermore, if $p$ and $q$ are mixed, and $C=\set{p,c_1,c_2,\hdots,c_{n-1},q}$ is a shortest COTS-arc connecting them, $c_1$ and $c_{n-1}$ must be pure, since no two mixed points are adjacent.
Then $\dist{\D}{p}{q} = \dist{\D}{c_1}{c_{n-1}}+2$ by Proposition \ref{cont_dist}.
This is a considerably better upperbound than the one given by the $d_1$ metric.

\begin{prop}
\label{prop:pure_dist}
Let $p \in \D$ be pure.
For all $q \in \D$ open and all $r \in \D$ closed, $d(p,q) \neq d(p,r)$.
\end{prop}
\begin{proof}
Without loss of generality, suppose $p$ is closed.
Suppose there exist pure points $q,r \in \D$ such that $q$ is open and $r$ is closed, and that $d(p,q)= d(p,r)=n$.
Because $d(p,q) = d(p,r)=n$, the COTS-arc between each pair of points is the image of a COTS $c_0 \geq c_1 \leq \hdots \lessgtr c_n$.
Note that $c_0$ must be closed since $p$ is closed.
Then $c_n$ is open if $n$ is odd and closed if $n$ is even.
Since $q$ is open and $r$ is closed, they cannot both be the image of $c_n$, a contradiction.
\end{proof}

\begin{prop}
\label{pure_unique}
If $p,q\in\D$ are on the same diagonal, then the shortest COTS-arc $C$ connecting $p$ and $q$ is unique.
Furthermore, the points of $C$ are a subset of the diagonal containing $p$ and $q$.
\end{prop}
\begin{proof}
Let $p,q \in \D$ be pure points on the same diagonal.
Translate and reflect the coordinate system on $\D$ such that $p=(0,0)$ and $q=(n,n)$ for some positive integer $n$.
By Proposition \ref{pure_dist_max}, $d(p,q)=n$.
We will prove via induction that the shortest COTS-arc connecting $p$ and $q$ is unique.

If $q=(1,1)$, then $d(p,q)=1$.
Any COTS-arc from $p$ to $q$ must contain both $p$ and $q$.
Since $q \in A(p)$, $\{p,q\}$ is the smallest set containing them, and it is unique.\footnote{It is worth noting that the definition of COTS from \cite{Khalimsky1990a} does not mention two-point COTS, however, the COTS-arc $\set{p,q}$ is still a finite model of a line segment.}

Suppose that the shortest COTS-arc $C$ between two pure points of the form $(0,0)$ and $(n-1,n-1)$ is unique and only contains points on the same diagonal, that is, $C=\{(i,i)\}_{i=0}^{n-1}$.

Let $q=(n,n)$ be a pure point of $\D$.
If $q$ is closed, consider the neighbors of $q$ in Figure \ref{fig:Aq}.
(If $q$ is open, the parity of the points will be flipped, and their coordinates will remain the same.)
By Proposition \ref{pure_dist_max}, $d(p,r_0)=n-1$, and the shortest COTS-arc containing $p$ and $r_0$ is unique by the inductive hypothesis.
If $C=\{(i,i)\}_{i=0}^n$ is not the unique COTS-arc containing $p=(0,0)$ and $q=(n,n)$, then there exists another shortest COTS-arc $C'=\{p,c_1,c_2,\hdots,c_{n-1},q\}$.

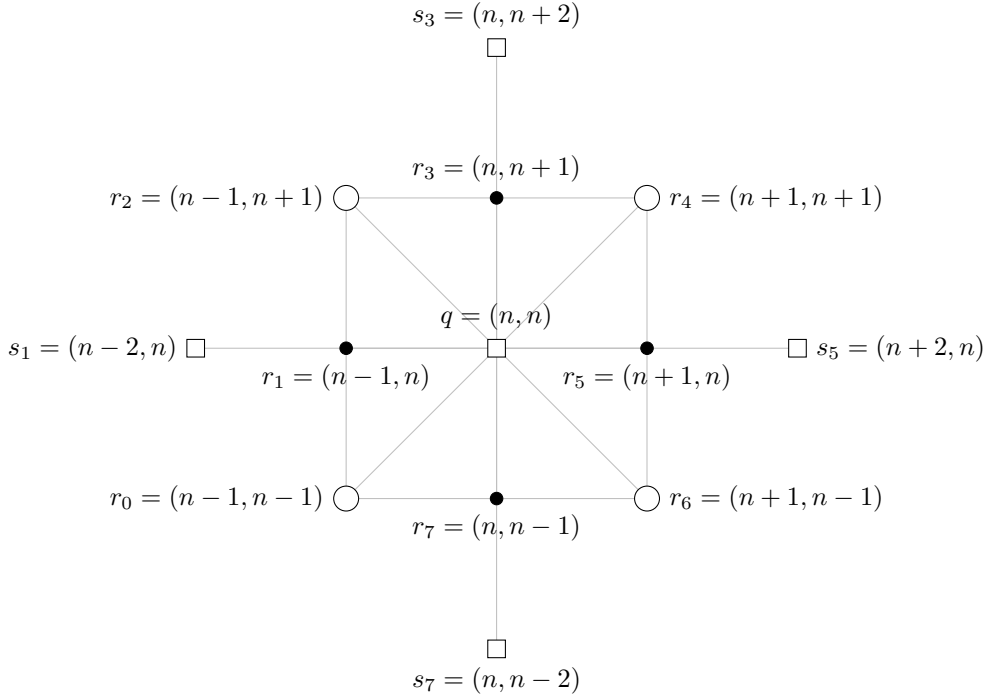
\begin{figure}
\centering
\begin{tikzpicture}[open/.style={shape=circle,fill=white},mixed/.style={shape=circle,fill=black,scale=.5}]

\draw[lightgray, thin, step=2.0] (-2,-2) grid (2,2);
\draw[lightgray, thin] (-2,-2)--(2,2);
\draw[lightgray, thin] (-2,2)--(2,-2);
\draw[lightgray, thin] (0,-4)--(0,4);
\draw[lightgray, thin] (-4,0)--(4,0);

\draw (0,0) node [shape=rectangle,fill=white,label={$q=(n,n)$},draw]{};

\draw (-2,-2) node [open,label=left:{$r_0=(n-1,n-1)$},draw]{};
\draw (-2,2) node [open,label=left:{$r_2=(n-1,n+1)$},draw]{};
\draw (2,2) node [open,label=right:{$r_4=(n+1,n+1)$},draw]{};
\draw (2,-2) node [open,label=right:{$r_6=(n+1,n-1)$},draw]{};

\draw (-2,0) node [mixed,label=below:{$r_1=(n-1,n)$},draw]{};
\draw (0,2) node [mixed,label=above:{$r_3=(n,n+1)$},draw]{};
\draw (2,0) node [mixed,label=below:{$r_5=(n+1,n)$},draw]{};
\draw (0,-2) node [mixed,label=below:{$r_7=(n,n-1)$},draw]{};

\draw (-4,0) node [shape=rectangle,fill=white,label=left:{$s_1=(n-2,n)$},draw]{};
\draw (0,4) node [shape=rectangle,fill=white,label={$s_3=(n,n+2)$},draw]{};
\draw (4,0) node [shape=rectangle,fill=white,label=right:{$s_5=(n+2,n)$},draw]{};
\draw (0,-4) node [shape=rectangle,fill=white,label=below:{$s_7=(n,n-2)$},draw]{};

\end{tikzpicture}
\label{fig:Aq}
\caption{Neighbors of $q$ and their coordinates}
\end{figure}

If $C \neq C'$, then $c_{n-1} \in \{r_1,r_2,\hdots,r_7\}$, as shown in Figure \ref{fig:Aq}.
By Proposition \ref{cont_dist}, $d(p,c_{n-1})=n-1$.
If $c_{n-1}$ is pure, then $c_{n-1} \in \{r_2,r_4,r_6\}$.
By Proposition \ref{pure_dist_max}, $d(p,r_2)=d(p,r_4)=d(p,r_6)=n+1 > n-1 = d(p,r_0)$, a contradiction.
If $c_{n-1}$ is mixed, $c_{n-2} \in \{s_1,s_3,s_5,s_7\}$ because COTS-arcs cannot turn at mixed points.
By Proposition \ref{cont_dist}, $d(p,c_{n-2})=n-2$, however, $d(p,s_i) \in \{n,n+2\}$ for $i \in \{1,4,5,7\}$ by Proposition \ref{pure_dist_max}, a contradiction.
\end{proof}

In the next section, we use the above results to learn more about the interiors of Jordan curves.

\section{Digital Jordan Curves with the Khalimsky Topology}

\label{chap:DigJC}

\subsection{Properties of Jordan curves}

The goal of this section is to prove the results we need to show that the space of digital Jordan curves is connected.
Throughout this section, ``Jordan curve'' will mean COTS-Jordan curve, unless otherwise specified.

\begin{prop}\label{JC-even} 
Every COTS-Jordan curve comprises an even number of points.
\end{prop}
\begin{proof}
For sake of contradiction, suppose that $J \subseteq \D$ is a COTS-Jordan curve with $|J|$ odd.
By definition, $J-\{j\}$ is a COTS-arc for all $j \in J$.
Fix a $j \in J$.
Denote $A(j) \cap J$ as $\{j^-,j^+\}$.
By Lemma 5.2(c) of \cite{Khalimsky1990a}, this determines exactly two COTS-arcs with endpoints $j^-,j^+$ that are subsets of $J$.
These are $\set{j^-,j,j^+}$ and $J-\{j\}=\set{j^+,j^{++},\hdots,j^{--},j^-}$.
Because $|J|$ is odd, $|J-\{j\}|$ is even, and as a COTS-arc, it is the homeomorphic image of a COTS, $\varphi:C\to \D$.
Without loss of generality, suppose $\varphi^{-1}\paren{j^-}$ is an open point of $C$, and $\varphi^{-1}\paren{j^+}$ is a closed point of $C$.
Because $\set{j^-,j,j^+}$ is a COTS-arc, it is the image of a COTS $\varphi':C' \to \D$.
Then either $\varphi'^{-1}(j)$ is open, in which case $\{j^-,j\}$ is not connected,
or $\varphi'^{-1}(j)$ is closed, in which case $\{j,j^+\}$ is not connected.
Hence, $|J|$ is even.
\end{proof}

It is important to note that Proposition \ref{JC-even} is true for any digital Jordan curve in a $T_\frac{1}{2}$ digital plane.
Proposition 5 of \cite{Slapal2006} states the following:

\begin{prop*}
A finite subset $C$ of a topological space $(X,u)$ is a simple closed curve in the space if and only if, in the connectedness graph of $(X,u)$, for each point $x \in C$ there are precisely two points of $C$ adjacent to $x$.
\end{prop*}

In a $T_\frac{1}{2}$ space, two points of the same type (e.g., two open points) cannot be adjacent.
If a simple closed curve were to have an odd number of points, then two consecutive points would have to be of the same type, which is not possible.

\begin{defn}
A digital Jordan curve is called {\bf minimal} if it the adjacency set of a point in the digital plane.
\end{defn}

Consider Figures \ref{fig:pure}, \ref{fig:purec}, or \ref{fig:amixed}. Each of these figures displays $x \cup A(x)$ for $x$ open, closed, and mixed, respectively. Deleting the central point $x$ in each figure gives the adjacency set of this point. It is easy to check that these sets satisfy the definitions of a digital Jordan curve. 
Furthermore, these are the only three minimal Jordan curves, up to rotation and translation.
The unique Jordan curve in $\D$ with $\abs{\jint{J}}$ maximal is the adjusted border of the digital plane, which is a Jordan curve by Lemma 5.2(b) of \cite{Khalimsky1990a}.
By ``adjusted border,'' we mean the border of $\D$ such that any mixed cornerpoints have been deleted.
Unless otherwise specified, $B \subset \D$ will represent this maximal Jordan curve.

\begin{lem}\label{lem:nonminpureint}
Every non-minimal Jordan curve $J$ contains at least one pure point in its interior.
\end{lem}
\begin{proof}
First we show that if $J$ is not minimal, then $|\jint{J}| \geq 2$.
If $J$ is not minimal, then $J \neq A(p)$ for some $p \in \D$ that does not touch the border.
Because $\jint{J}\neq \emptyset$, $\abs{\jint{J}}\geq 1$.
Suppose for sake of contradiction that $\abs{\jint{J}}=1$ such that $\jint{J}=\set{q}$ for some $q \in \D$.
Since $J \neq A(q)$, there exists some $j \in J$ such that $j \not\in A(q)$.
By Lemma 5.2(a) of \cite{Khalimsky1990a}, $A(j)- J$ has exactly two components $A$ and $A'$ such that $A \subseteq \jint{J}$ and $A' \subseteq \jext{J}$.
Since $A \subseteq A(j)$ but $j \not \in A(q)$, it must be the case that $q \not\in A$.
Since $A$ is nonempty, there exists an $r \in A \subseteq \jint{J}$ such that $r \neq q$.
Then $\set{q,r}\subseteq \jint{J}$, so $\abs{\jint{J}}\geq 2$ when $J$ is not minimal.

Consider the open subsets $q^\downarrow$ and $r^\downarrow$ of $\D$.
If we assume that $q$ and $r$ are both mixed for sake of contradiction, it must be the case that $q^\downarrow \cap \jint{J}=\{q\}$ and $r^\downarrow\cap \jint{J}=\{r\}$, or else $\jint{J}$ would contain a pure point.
Then $q$ and $r$ are both open points of $\jint{J}$ with the subspace topology.
Since $q$ and $r$ are both open points of $\jint{J}$ and $\D$ is $T_0$, $\set{q,r}$ is a discrete set, so $\jint{J}$ is not one connected component, a contradiction.

Hence, $\jint{J}$ must contain at least one pure point.
\end{proof}

Furthermore, every non-minimal Jordan curve contains at least one mixed point in its interior, or else the Jordan curve enclosing it would turn at a mixed point, which is forbidden by Definition 4.1(iii) of \cite{Khalimsky1990a}.
To see this, consider the following example:

\begin{ex}\label{ex:1mix}
Suppose $\set{p,q}\subseteq \jint{J}$ is a connected set of $\jint{J}$ comprising one open point $p$ and one closed point $q$.
Consider one of the mixed points $r\in A(p)\cap A(q) \subset \D$, as shown in Figure \ref{fig:4pts}.
The dashed lines of this figure connect points through which $J$ could pass.
Since $r \not\in \jint{J}$, $r \in J$ or else $\jint{J} \cup \jext{J}$ would be connected.
Let $A(r)\cap J=\set{r^-,r^+}$ be the two-point discrete subset guaranteed by Lemma 5.2(a) of \cite{Khalimsky1990a}.
Consider $r^-\in \set{r_1,r_2,r_3,r_4}$.
If $r^-,r^+\in\set{r_1,r_4}$, then $\set{r^-,r,r^+}$ is not connected, a contradiction.
Hence $r^-,r^+$ must be pure, so $\set{r^-,r^+}=\set{r_2,r_3}$.
Then $r^- \in A(r^+)$, so $\set{r^-,r,r^+}$ is not a COTS-arc, a contradiction.
Hence every non-minimal Jordan curve must also contain at least one mixed point.

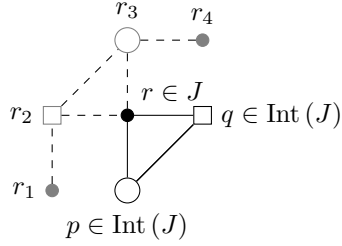
\begin{figure}
    \centering
    \begin{tikzpicture}
    \node[shape=circle, label=below:{$p\in\jint{J}$},draw] (p) at (0,0){};
    \node[shape=rectangle, label=right:{$q\in\jint{J}$},draw] (q) at (1,1){};
    \node[shape=circle,scale=.5,fill=black, label=above right:{$r\in J$},draw] (r) at (0,1){};
    \draw[] (p)--(r)--(q)--(p)--(q);
    
    \node[shape=rectangle, color=gray,label=left:{$r_2$}, draw] (r1) at (-1,1){};
    \node[shape=circle, color=gray,label=above:{$r_3$},draw] (r2) at (0,2){};
    \node[shape=circle,scale=.5,color=gray,fill=gray,label=above:{$r_4$},draw] (r3) at (1,2){};
    \node[shape=circle,scale=.5,color=gray,fill=gray,label=left:{$r_1$},draw] (r4) at (-1,0){};
    \draw[dashed] (r4)--(r1)--(r)--(r2)--(r3);
    \draw[dashed] (r1)--(r2);
    \end{tikzpicture}
    \caption{A 3-point connected subset of $\D$ and some neighbors, for use in Example \ref{ex:1mix}}
    \label{fig:4pts}
\end{figure}

\end{ex}

\begin{lem}\label{lem:nodd} 
Let $J$ be a non-minimal Jordan curve and $p \in \jint{J}$ pure such that $A(p) \cap J \neq \emptyset$.
If $\{j_1,\hdots,j_n\}$ is a connected component of $A(p) \cap J$, then $n$ is odd.
\label{lem:odd}
\end{lem}
\begin{proof}
Let $J$ be a Jordan curve and $p \in \jint{J}$ pure such that $A(p) \cap J \neq \emptyset$.
Let $\{j_1,\hdots,j_n\}$ be a connected component of $A(p) \cap J$.
Since $p$ is pure (and, without loss of generality, closed), $j_1, \hdots, j_n$ alternate between mixed and open.
So if $n$ is even, either $j_1$ or $j_n$ is mixed.
Without loss of generality, suppose $j_n$ is mixed.
Then the point following $j_n$ in $J$ is not in $A(p)$,
so $J$ turns at a mixed point, a contradiction.

Furthermore, $j_1$ and $j_n$ must both be pure, or else $J$ would turn at a mixed point.
\end{proof}

\begin{lem}\label{lem:components} 
If $C$ is a proper subset of a Jordan curve $J$, then $C$ and $J-C$ have the same number of components.
\end{lem}

\begin{proof}
Let $C \subset J$ with components $C_1,C_2,\hdots,C_n$.
Let $D:=J-C$ with components $D_1,D_2,\hdots, D_m$ for some $m\neq n$.
Suppose without loss of generality that $m > n$ (if $m<n$, we may swap the sets $C$ and $D$).
Enumerate the points of $J$ clockwise as $J=\set{j_0,j_1,\hdots,j_{|J|-1}}$ such that $A(j_i)\cap J = \set{j_{(i-1)\bmod{|J|}},j_{(i+1)\bmod{|J|}}}$.
Consider $C_1$, and rotate the enumeration of $J$ such that $j_0$ is the counter-clockwise-most point of $C_1$ and $C_1$ can be written as $\set{j_0,j_1,\hdots,j_{|C_1|-1}}$.
Since $C_1$ cannot be adjacent to another component of $C$, there exists a component $D_1 \subseteq D$ such that $j_{|C_1|}\in D_1$.
Repeating this process gives an ordering such that 
\[C_1 \cup D_1 \cup C_2 \cup D_2 \cup \hdots \cup C_n \cup D_n\]
is a connected subset of $J$.
Take $n':=\sum_{i=1}^n\abs{C_i\cup D_i}$.
Then $j_{n'} \in D_j$ for some $n<j\leq m$.
so $D_{n} \cup D_j$ is a connected subset of $J$, a contradiction.
\end{proof}

\begin{prop}\label{prop:jintk} 
Let $J \subseteq \D$ be a Jordan curve.
If $K \subseteq J\cup\jint{J}$ is a Jordan curve, then $\jint{K} \subseteq \jint{J}$.
\end{prop}
\begin{proof}
Suppose for sake of contradiction there there exists a $k \in \jint{K}$ such that $k \not\in \jint{J}$.
Then $k \in J \cup \jext{J}$.
Because $k \in \jint{K}$, every path from $k$ to the boundary $B$ of $\D$ must pass through $K$.
If $k \in J\cap B$ (i.e., $J$ meets the border of $\D$), then $K \not\subseteq \D$, a contradiction.
Otherwise, $k$ is adjacent to some point $p_1 \in \jext{J}$.
Because $\jext{J}$ meets the border $B\subset \D$, there exists a connected component of $\jext{J}$ containing both $p_1$ and some $b \in B\cap\jext{J}$; call this component $A$.
Since $A$ is connected and finite and therefore path-connected, there exists a path $\alpha:=\{k,p_1,p_2,\hdots,p_{n-1},b\} \subset \{k\}\cup \jext{J}$ from $k$ to $b$ such that each of the $p_i$ are in $\jext{J}$.
Since $K \subset J \cup \jint{J}$ and $(J \cup \jint{J})\cap\jext{J}=\emptyset$, however, no point of $\alpha$ can also be in $K$, contradicting the assumption that every path from $k$ to $B$ runs through $K$.
\end{proof}

\begin{lem}\label{lem:maxcon}
Let $J \subseteq \mathcal{D}$ be a non-minimal Jordan curve.
Fix any pure point $p \in \jint{J}$, and choose $q \in \jint{J}$ such that $d_{\jint{J}}(p,q)$ is maximal.
Then:
\begin{enumerate}
\item[(a)] $A(q)\cap J$ is connected,
\item[(b)] $|A(q)\cap J| \geq 3$ if $q$ is pure, and
\item[(c)] $|A(q)\cap J| \geq 3$ if $q$ is mixed and there are no pure points of maximal distance. 
\end{enumerate}
\end{lem}

The proof of this result was surprisingly elusive.
Let us consider its counterpart in the Hausdorff setting.
Let $J \subset \rr^2$ be a simple closed loop in the plane that satisfies the Jordan curve theorem.
Fix any point $x \in \jint{J} \subset \rr^2$.
Let $n = \sup_{y\in\jint{J}}\set{d(x,y)}$.
Our intuition tells us that if there exists a $y \in \jint{J}$ and an $\varepsilon>0$ such that $D(y,\varepsilon)\cap J = \emptyset$, then $d(x,y)\neq n$.
Since that is not the case in $T_0$ spaces, we present the following.

\begin{proof}[Proof of Lemma \ref{lem:maxcon}]
For the sake of brevity and by abuse of notation, we will drop the subscript from $d_{\jint{J}}(p,p')$ and write $d_{p'}$, where $p$ is the fixed pure point and $p'\in\jint{J}$ is any other point of $\jint{J}$.

\begin{itemize}
\item[(a)] $A(q)\cap J$ is connected:

If $A(q)\cap J = \emptyset$, then it is vacuously connected.
Otherwise, if $d_q=n$ is maximal, and since $\jint{J}$ is arcwise-connected, there exists a (not necessarily unique) shortest COTS-arc $\alpha = \{p,a_1,a_2,\hdots,a_{n-1},q\} \subseteq \D$ from $p$ to $q$.
Suppose $A(q) \cap J$ is nonempty and disconnected.
Then each of $A(q)  -  J$ and $A(q) \cap J$ have the same number of components, which is at least two each, by applying Lemma \ref{lem:components} to the Jordan curve $A(q)$.
Because $\alpha \subseteq \jint{J}$, $a_{n-1} \in A(q)\cap\jint{J}=A(q)-J$.
Let $b$ be a point in a component of $A(q)  -  J$ through which $\alpha$ does not run.
If the shortest arc from $p$ to $b$ runs through $q$, this contradicts the maximality of the distance from $p$ to $q$ by Proposition \ref{cont_dist}.

If the shortest path from $p$ to $b$ does not run through $q$, call that path $\beta'$.
We can construct a second arc, $\beta:= \beta' \cup \{q\}$, that runs from $p$ to $q$ through $b$. 
Note that $\{p,q\} \subseteq \alpha \cap \beta$, 
and $\alpha \cup \beta - \{q\}$ is a COTS-path from $b$ to $a_{n-1}$ that does not run through $\{q\}$ by Proposition \ref{prop:concat-paths}.
For example, we may consider Figure \ref{fig:AqJdis}, where the blue line represents $J \cap A(q)$, and the red line will be defined in the next paragraph.
The points $a_{n-1}$ and $b$ are both of distance $n-1$ from $p$, and are each in a different component of $A(q)-J$.

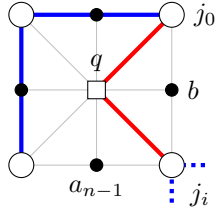
\begin{figure}
    \centering
      \begin{tikzpicture}
\foreach \x in {1,2,3}
	{
		\draw[lightgray, very thin] (\x,1)--(\x,3);
		\draw[lightgray, very thin] (1,\x)--(3,\x);
	}
		\draw[lightgray, very thin] (1,1)--(3,3);
		\draw[lightgray, very thin] (1,3)--(3,1);

	\draw[red, ultra thick] (3,3)--(2,2)--(3,1);
	\draw[blue, ultra thick] (1,1)--(1,2)--(1,3)--(2,3)--(3,3);
	\draw[blue, ultra thick, dotted] (3,0.5)--(3,1)--(3.5,1);

	\draw (2,2) node [shape=rectangle,fill=white,draw,label={$q$}]{};

	\draw (1,1) node [shape=circle,fill=white,draw]{};
	\draw (1,3) node [shape=circle,fill=white,draw]{};
	\draw (3,1) node [shape=circle,fill=white,draw,label=below right:{$j_i$}]{};
	\draw (3,3) node [shape=circle,fill=white,draw,label=right:{$j_0$}]{};

		\draw (1,2) node [shape=circle,fill=black,scale=.5,draw]{};
		\draw (2,1) node [shape=circle,fill=black,scale=.5, label=below:{$a_{n-1}$},draw]{};
		\draw (2,3) node [shape=circle,fill=black,scale=.5,draw]{};
		\draw (3,2) node [shape=circle,fill=black,scale=.5,label=right:{$b$},draw]{};	
\end{tikzpicture}
    \caption{An example of $A(q)\cap J$ being disconnected for $q$ of maximal distance}
    \label{fig:AqJdis}
\end{figure}

Next, we will construct a Jordan curve $K \subset J\cup\jint{J}$ such that $b \in \jint{K}$ and $a_{n-1} \in \jext{K}\cap\jint{J}$.
For example, in Figure \ref{fig:AqJdis}, the new segment of $K$ is depicted by the red line.
Fix a clockwise ordering of $J$.
Let $j_0$ be the clockwise-most point of some component of $A(q)\cap J$.
If $q$ is pure, $j_0$ is pure by Lemma \ref{lem:nodd}.
If $q$ is mixed, $j_0$ is pure because $A(q)$ comprises exclusively pure points.
Since $A(q)\cap J$ is disconnected, there exists a minimal index $i$ such that $j_i \in A(q) \cap J$, and that $j_0$ and $j_i$ are in different components of $A(q)\cap J$.
By the minimality of $i$, $j_i$ is also pure.
Consider $K:=\{j_0,\hdots,j_i\} \cup \{q\}$; we will show that $K$ is a Jordan curve.
Because $\{j_0,\hdots,j_i\}$ is a connected subset of $J$, it is a COTS-arc $\cots{j_0,j_i}$, so $|A(j)\cap \cots{j_0,j_i}|=2$ for $j \in \cots{j_1,j_{i-1}}$, and $|A(j)\cap \cots{j_0,j_i}|=1$ for $j \in \{j_0,j_i\}$.
Consequently, $|A(j)\cap K|=2$ for $j \in \cots{j_1,j_{i-1}}$.
It remains to be shown that $|A(j)\cap K|=2$ for $j \in \{q,j_0,j_i\}$.
In $J$, $A(j_0)\cap J = \set{j_{|J|-1},j_1}$.
We know that $j_{|J|-1}\not\in A(j_i)\cap J$, or else $\set{j_i,j_{|J|-1},j_0}$ would be a connected set of $J$.
Similarly, $j_1 \not\in A(j_i)\cap J$.
Hence, $j_{|J|-1}\not\in A(j_0)\cap K$, so $\abs{A(j_0)\cap K}=\abs{\set{q,j_1}}=2$.
By a symmetrical argument, $\abs{A(j_i)\cap K} = \abs{\set{j_{i-1},q}}=2$.
Lastly, by the minimality of $i$, $|A(q) \cap K| = |\{j_0,j_i\}|=2$, so $K$ is a Jordan curve.

Since $K$ is a Jordan curve, the arc $\{j_0,q,j_i\}\subset K$ determines three components of $A(q)$, each belonging to one of $\jext{K}$, $K$, and $\jint{K}$.
One of these components must contain $b$.
Notice that
\[\bigcup \begin{array}{l}
\set{\textrm{the component of $A(q)\cap J$ containing }j_0}\\
\set{\textrm{the component of $A(q) -J$ belonging to }\jint{K}}\\
\set{\textrm{the component of $A(q)\cap J$ containing }j_i}
\end{array}\]
is a connected subset of the Jordan curve $A(q)$.
To see that there exists a unique component of ($A(q)-J)\cap \jint{K}$, suppose for sake of contradiction that there exists a decomposition of $(A(q)-J)\cap \jint{K}$ into distinct nonempty connected components $K_1,\hdots, K_\ell$.
By Lemma 5.2(a) of \cite{Khalimsky1990a}, $A(q)\cap\jint{K}$ is connected.
Since $K_1 \cup \hdots \cup K_\ell \subset (A(q)-J)\cap\jint{K}$ is disconnected, there exists a subset $J'\subset J$ such that $K_1 \cup \hdots \cup K_\ell \cup J'$ is a connected subset of $A(q)\cap \jint{K}$.
This implies that there exists a point $j \in J \cap \jint{K} \subset J \cap \jint{J}=\emptyset$, a contradiction.
By performing this construction for every choice of $j_0$, we will eventually arrive at a choice of $j_0$ such that $b \in (A(q)-J)\cap \jint{K}$.

Since $K$ is a Jordan curve,
\begin{eqnarray*}
\D &=& \jint{K} \sqcup K \sqcup \jext{K}\\
\D \cap \jint{J} &=& \left( \jint{K} \sqcup K \sqcup \jext{K} \right) \cap \jint{J}\\
\jint{J} &=& (\jint{K} \cap \jint{J}) \sqcup (K \cap \jint{J}) \sqcup (\jext{K} \cap \jint{J})\\
 &=& \jint{K} \sqcup \{q\} \sqcup (\jext{K} \cap \jint{J}).
\end{eqnarray*}

Since $p \neq q$, either $p \in \jint{K}$ or $p \in \jext{K}\cap \jint{J}$.
Recall that by construction, $b \in \jint{K}$ and $a_{n-1} \in \jext{K} \cap \jint{J}$.
Since $a_{n-1}$ and $b$ are in different components of $\jint{J}$, there is no path from $a_{n-1}$ to $b$ that does not run through $q$.
This contradicts that $\alpha \cup \beta - \{q\}$ gives a path from $a_{n-1}$ to $b$.
Hence, if $q$ is of maximal distance from $p$ in $\jint{J}$ and $A(q)\cap J \neq \emptyset$, then $A(q) \cap J$ is connected.

\item[(b)] $|A(q)\cap J| \geq 3$ if $q$ is pure:

Suppose $q$ is pure.
By part (a), $A(q)\cap J$ is one connected component, and by Lemma \ref{lem:nodd}, $|A(q)\cap J|$, is odd, so we only need to prove $|A(q)\cap J| \not\in \{0,1\}$.
If $|A(q) \cap J| \leq 1$ and $q$ is pure, $q$ is one of either Figure \ref{fig:1} or \ref{fig:0}, where the blue line highlights part of $J$ (only in Figure \ref{fig:1}), and the red line highlights part of a new Jordan curve $K$ that will divide $\jint{J}$ into two disjoint components.
If $d(p,q)=n$, the pure points of $A(q)\cap\jint{J}$ must be of distances $n-1$ or $n+1$ by Proposition \ref{prop:pure_dist}.
If any of them are of distance $n+1$, then $q$ is not of maximal distance, a contradiction.
Assume that all pure points in $A(q) \cap \jint{J}$ are of distance $n-1$.

\begin{figure}
  \begin{tikzpicture}
\foreach \x in {1,2,3}
	{
		\draw[lightgray, thin] (\x,1)--(\x,3);
		\draw[lightgray, thin] (1,\x)--(3,\x);
	}
		\draw[lightgray, thin] (1,1)--(3,3);
		\draw[lightgray, thin] (1,3)--(3,1);

	\draw[dotted, blue, ultra thick] (.5,3)--(1,3)--(1,3.5);	
	\draw[dotted, blue, ultra thick] (.5,2.5)--(1,3)--(1.5,3.5);

	\draw[red, ultra thick] (2,2)--(3,1);
	\draw[dotted, red, ultra thick] (3,1)--(3.5,.5);

	\draw (2,2) node [shape=rectangle,fill=white,draw,label={$q$}]{};

	\draw (1,1) node [shape=circle,fill=white,draw,label=left:{$e$}]{};
	\draw (1,3) node [shape=circle,fill=white,draw,label=below right:{$f$}]{};
	\draw (3,1) node [shape=circle,fill=white,draw,label=right:{$c$}]{};
	\draw (3,3) node [shape=circle,fill=white,draw,label=right:{$d$}]{};

		\draw (1,2) node [shape=circle,fill=black,scale=.5,draw]{};
		\draw (2,1) node [shape=circle,fill=black,scale=.5,draw]{};
		\draw (2,3) node [shape=circle,fill=black,scale=.5,draw]{};
		\draw (3,2) node [shape=circle,fill=black,scale=.5,draw]{};	
\end{tikzpicture}
\caption{$|A(q) \cap J|=1$ and $q$ is pure}
\label{fig:1}
\end{figure}%
\begin{figure}
    \centering
\begin{tikzpicture}
\foreach \x in {1,2,3}
	{
		\draw[lightgray, thin] (\x,1)--(\x,3);
		\draw[lightgray, thin] (1,\x)--(3,\x);
	}
		\draw[lightgray, thin] (1,1)--(3,3);
		\draw[lightgray, thin] (1,3)--(3,1);

	\draw[red, ultra thick] (1,3)--(3,1);
	\draw[dotted, red, ultra thick] (.5,3.5)--(3.5,.5);

	\draw (2,2) node [shape=rectangle,fill=white,draw,label={$q$}]{};

	\draw (1,1) node [shape=circle,fill=white,draw,label=left:{$e$}]{};
	\draw (1,3) node [shape=circle,fill=white,draw,label=left:{$f$}]{};
	\draw (3,1) node [shape=circle,fill=white,draw,label=right:{$c$}]{};
	\draw (3,3) node [shape=circle,fill=white,draw,label=right:{$d$}]{};

		\draw (1,2) node [shape=circle,fill=black,scale=.5,draw]{};
		\draw (2,1) node [shape=circle,fill=black,scale=.5,draw]{};
		\draw (2,3) node [shape=circle,fill=black,scale=.5,draw]{};
		\draw (3,2) node [shape=circle,fill=black,scale=.5,draw]{};	
\end{tikzpicture}
\caption{$|A(q) \cap J|=0$ and $q$ is pure}
\label{fig:0}
\end{figure}
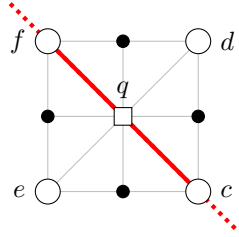  

Label the four pure points of $A(q)$ as  $c,d,e$, and $f$ such that $c$ and $f$ are on the same diagonal, and $d$ and $e$ are on the same diagonal.
(If $|A(q)\cap J|=1$ as in Figure \ref{fig:1}, then $f=A(q)\cap J$.)
Consider a COTS-arc $C=\{c_{-m_1}, \hdots, c_{-1},q,c_1,\hdots,c_{m_2}\}$ extending through $\jint{J}$ whose length is minimal such that $A(J)\cap C = \{c_{-m_1},c_{m_2}\}$.
In the case of $|A(q)\cap J|=1$, the construction of $C$ starts at $q=c_0=c_{-1}=\hdots = c_{-m_1}$.
In the case of $|A(q)\cap J|=0$, $c_{-1}=f$ and $c_1=c$.
Enumerate the points of $J$ as $\{j_0,j_1,\hdots,j_{|J|-1}\}$ such that $j_0$ is the clockwise-most point of $A(c_{-m_1})\cap J$.
If $|A(q)\cap J|=1$, $j_0=f$.
Choose $m'\in \set{1,\hdots,|J|-1}$ minimal such that $|A(j_{m'})\cap C|>0$.
Clearly, $C$ and $\set{j_0,j_1,\hdots,j_{m'}}\subset J$ are each COTS-arcs whose endpoints are adjacent.
By Proposition \ref{prop:path-arc}, $C \cup\{j_1,j_2,\hdots,j_{m'}\}$ contains a COTS-arc $C'$ as a subset that starts at $j_1$ and ends at $q$.
Note the omission of $j_0$ from the construction of $C'$.
Since $C'$ is a COTS-arc, $|A(c) \cap C'|=1$ for $c\in\set{j_1,c_{-m_1}}$, and $|A(c)\cap C'|=2$ for $c \in \cots{j_1,c_{-m_1}}$.
Take $K:= C' \cup \{j_0\}$.
Notice that $A(j_0) \cap K = \{j_1,c_{-m_1}\}$, $A(c_{-m_1})\cap K = \{j_0,c_{-m_1+1}\}$ (or $\{j_0,c_1\}$ when $|A(q)\cap J|=1$), and $A(j_1) \cap K = \{j_0,j_2\}$, so $K$ is a Jordan curve by Lemma 5.2(a) of \cite{Khalimsky1990a}.
Since $K$ is a Jordan curve, it has an interior, $\jint{K}\subset \jint{J}$ by Proposition \ref{prop:jintk}.
We consider three cases: $p \in \jint{K}$, $p \in C$, and $p \in \jext{K}\cap\jint{J}$.

By Lemma 5.2(a) of \cite{Khalimsky1990a}, we know that $A(q)-K$ has exactly two components; let $d \in \jint{K}$ and $e \in \jext{K} \cap \jint{J}$.
Because $d_d=d_e=n-1$, let $D=\{p,d_1,\hdots,d_{n-1},q\}$\footnote{By abuse of notation, $d_i$ here is a point of $D$ for $1\leq i \leq n-1$, and not $d_{\jint{J}}\paren{p,i}$.} and $E = \{p,e_1,\hdots,e_{n-1},q\}$ be COTS-arcs in $\jint{J}$ from $p$ to $q$ such that $d_{n-1}=d$ and $e_{n-1}=e$.
If $p \in \jint{K}$, then the COTS-arc $E$ must intersect the COTS-arc $C$ at some pure point $c_\epsilon = e_{\epsilon'} \in C \cap E$ to traverse from $\jint{K}$ to $\jext{K}\cap\jint{J}$, where $1\leq \epsilon,\epsilon' < n-1$.
Consider the two COTS-arcs $\{c_\epsilon,c_{\epsilon-1},\hdots,c_1,q\}\subseteq C$ and $\{e_{\epsilon'},e_{\epsilon'+1},e_{\epsilon'+2},\hdots,e_{n-1},q\} \subseteq E$.
Each of these COTS-arcs have the same start and end points, however, they are distinct since $e_{n-1}=e \not\in C$.
This contradicts Proposition \ref{pure_unique}, which states that the shortest COTS-arc between two pure points on the same diagonal is unique.
By a similar argument, if $p \in \jext{K}\cap\jint{J}$, we can construct two distinct shortest COTS-arcs between pure points on the same diagonal, again contradicting Proposition \ref{pure_unique}.
Lastly, if $p \in C$, then $D$ and $E$ are distinct shortest COTS-arcs between two pure points on the same diagonal, a contradiction.
Hence $A(q)\cap J \geq 3$ if $q$ is pure.

\bigskip

\item[(c)] $|A(q)\cap J| \geq 3$ if $q$ is mixed and there are no pure points of maximal distance:

If $q$ is mixed and $|A(q)\cap J| < 3$, then $A(q)\cap \jint{J}$ contains at least one open point $o$ and one closed point $c$.
By Proposition \ref{prop:pure_dist}, $d_o\neq d_c$.
Since there are no pure points of maximal distance, $d_o<n$ and $d_c < n$.
Then one of either $o$ or $c$ must be of distance $n-1$ and the other must be of distance $n-2$.
Then $q$ is adjacent to a point of distance $n-2$, so $d(p,q)=n-1 < n$, a contradiction.
Hence $|A(q)\cap J| \geq 3$.

\end{itemize}
\end{proof}

Let us examine an application of this lemma to a Jordan curve:

\begin{ex}
Consider Figure \ref{fig:maxconapp}.
Highlighted in blue is a Jordan curve $J$ such that there exists a point $q \in \jint{J}$ with $\abs{A(q)\cap J}=\abs{j_0}=1$.
Highlighted in red is the COTS-arc $C$ that extends diagonally from $q=c_0$ until it nears another point of $J$, in this case, $c_{m_2}\in A(j_{m'})$.
Highlighted in magenta is the Jordan curve $K$ constructed from the union $C \cup \cots{j_0,j_{m'}}$.
In this example, $p \in \jint{K}$.
If it is to be true that $d(p,e)=d(p,c)=n-1$, then a COTS-arc $E$ of length $n-1$ (shown in green and not necessarily unique) must cross through $C$ to get from $p$ to $e$; let $c_\epsilon$ be a point in this intersection.
By assumption, this creates two COTS-arcs of the same length starting at $c_\epsilon$ and ending at $q$.
Since $q$ and $c_\epsilon$ are pure points on the same diagonal, however, the shortest path between them should be unique.
\begin{figure}
    \centering
    \begin{tikzpicture}
    
    \path[color=blue, ultra thick, draw] (1,1)--(1,9)--(3,9)--(3,11)--(11,11)--(11,1)--(1,1);
    
    \path[color=red, ultra thick, draw] (4,8)--(10,2);
    
    \path[color=magenta, ultra thick, draw] (3.25,9)--(3.25,10.75)--(10.75,10.75)--(10.75,3)--(10,2.25)--(3.25,9);
    
    \path[color=green, ultra thick, draw] (4,8)--(3,7)--(6,4)--(10,4);

    \foreach \x in {1,3,5,7,9,11}{
        \foreach \y in {1,3,5,7,9,11}{
            \draw (\x,\y) node [shape=circle, color=gray, fill=white,draw]{};
        }}
    \foreach \x in {2,4,6,8,10}{
        \foreach \y in {2,4,6,8,10}{
            \draw (\x,\y) node [color=gray,shape=rectangle, fill=white,draw]{};
        }}
    \foreach \x in {1,3,5,7,9,11}{
        \foreach \y in {2,4,6,8,10}{
            \draw (\x,\y) node [shape=circle, color=gray, fill=gray, scale=.5,draw]{};
            \draw (\y,\x) node [shape=circle, color=gray, fill=gray, scale=.5,draw]{};
        }}
    \draw (3,9) node [shape=circle, draw, color=black, fill=white, label=above left:{$f=j_0$}]{};
    \draw (3,10) node [shape=circle, color=black, fill=black, scale=.5, label=left:{$j_1$},draw]{};
    \draw (11,3) node [shape=circle, draw, color=black, fill=white, label=right:{$j_{m'}$}]{};
    
    \draw (4,8) node [shape=rectangle, draw, color=black, fill=white, label=above right:{$q=c_0$}]{};
    \draw (5,7) node [shape=circle, draw, color=black, fill=white, label=above right:{$c$}]{};
    \draw (10,2) node [shape=rectangle, draw, color=black, fill=white, label=below:{$c_{m_2}$}]{};
    \draw (8,4) node [shape=rectangle, draw, color=black, fill=white, label=above right:{$c_{\epsilon}$}]{};

    \draw (3,7) node [shape=circle, draw, color=black, fill=white, label=above:{$e=e_{n-1}$}]{};
    
    \draw (10,4) node [shape=rectangle, draw, color=black, fill=white, label=right:{$p$}]{};
  
    \end{tikzpicture}
    \caption{An application of Lemma \ref{lem:maxcon} to a Jordan curve (in blue), with the adjacency lines of the digital plane omitted}
    \label{fig:maxconapp}
\end{figure}
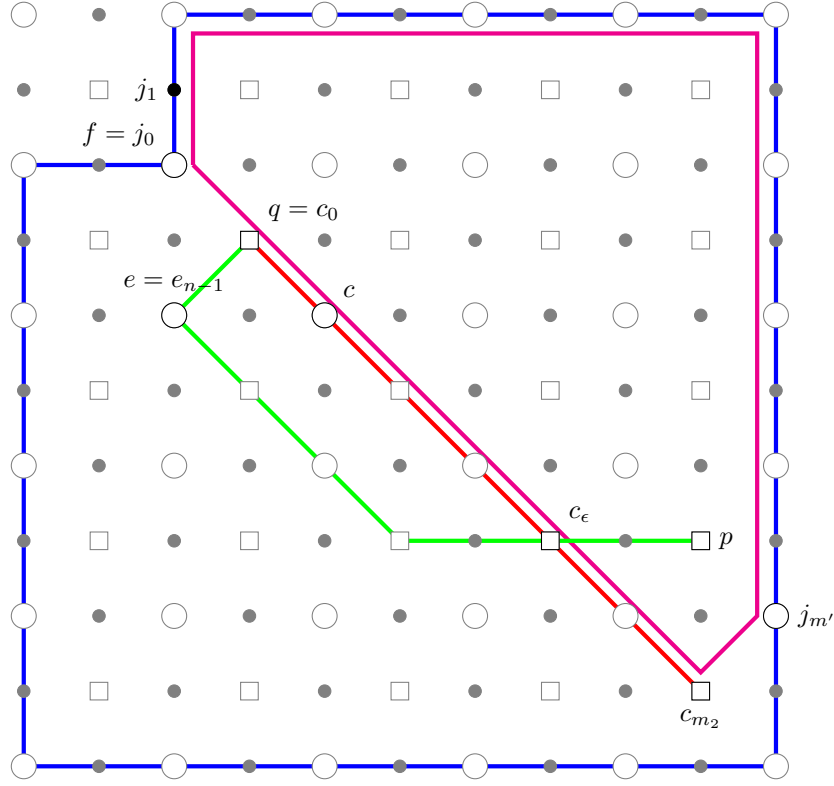
\end{ex}

\begin{cor}\label{cor:weak} 
If $p \in \jint{J}$ and $A(p) \cap J$ is a proper nonempty connected subset of $A(p)$, then $p$ is a weak point of $\jint{J}$.
\end{cor}

\begin{proof}
Without loss of generality, suppose $p$ is closed.
Then $p^\downarrow-\{p\} = A(p)$ is a Jordan curve.
Since $p \in \jint{J}$, $A(p)-J \subseteq \jint{J}$.
Since $A(p)-J$ is a connected subset of a Jordan curve, it is a COTS-arc, so it is contractible in $\jint{J}$, so it is a weak point.

If $p$ is mixed and $A(p) \cap J$ is a proper connected subset, then $|A(p) \cap J| = 1,2,$ or $3$. 
In each case, either $|(p^\downarrow-\{p\})\cap \jint{J}|=1$ or $|(p^\uparrow-\{p\})\cap \jint{J}|=1$, which is contractible since its core is a point.

Hence if $A(p) \cap J$ is connected, then $p$ is a weak point of $\jint{J}$.
\end{proof}

In the next section, we will use Corollary \ref{cor:weak} to prove the following:

\begin{cor}\label{cor:weakcontra}
For $J \subseteq \D$ a digital Jordan curve, $\jint{J}$ is weakly contractible.
\end{cor}

\subsection{Spaces of Digital Jordan Curves}

Our ultimate goal is to understand the set of digital Jordan curves within a finite digital plane as a topological space. How many are there? When are two digital Jordan curves adjacent? Is the space contractible, as it is in the real setting?

\begin{defn}
Given a digital plane $\D$, we define
\[\J(\D) = \left\{ J \subset \D \mid J \textrm{ is a Jordan curve} \right\}\]
to be the set of digital Jordan curves in a digital plane $\D$.
\end{defn}
When the choice of digital plane is obvious or irrelevant, the $\D$ may be dropped from notation.

Recall that for $m$ mixed, $|A(m)| = 4$, and for $p$ pure, $|A(p)|=8$, so not all Jordan curves comprise the same number of points. One might fear that two Jordan curves in $\J$ could only be homotopic if they comprise the same number of points, however, introducing parameterizations of Jordan curves solves this problem.

\begin{prop}\label{prop:pars1}
Every digital Jordan curve admits a parametrization by $S^1$.
\end{prop}
\begin{proof}
By Proposition \ref{JC-even}, we can suppose $|J|=n$ for some even integer $n$.
Denote the points of $J$ as $j_0,j_1,\hdots,j_{n-1}$ such that $A(j_i) \cap J = \{j_{i-1 \bmod{n}},j_{i+1\bmod{n}}\}$ for $0 \leq i <n$.
Without loss of generality, choose the ordering such that $j_1$ is the image of open point of a COTS and $j_n$ is the image of a closed point of a COTS.
Taking $S^1 \cong [0,1]/(0\sim1)$, a parameterization exists as follows:
\[
f^{-1}(j_\ell) = 
\left\{\begin{array}{lr}
\left(\frac{\ell-1}{n}, \frac{\ell}{n}\right), & \ell\textrm{ is odd}\\
\left[\frac{\ell-1}{n}, \frac{\ell}{n}\right], & \ell\textrm{ is even}
\end{array}
\right.
\]
Since $j_i$ is the preimage of an open (closed) point of a COTS for $i$ odd (even), the preimage under $f$ of open points in $J$ is open in $S^1$, so $f:S^1 \to \D$ is a map of the circle into the digital plane whose image is a Jordan curve.

Furthermore, we call this construction the standard parameterization.
\end{proof}

Since the points of a simple closed curve in a $T_\frac{1}{2}$ digital plane $(X,u)$ must alternate between open and closed (there are no mixed points), Proposition \ref{prop:pars1} holds for digital Jordan curves in any $T_\frac{1}{2}$ digital plane.

We will use these parameterizations to define a partial order on $\J$.

\begin{defn}
Given $J, J' \in \J$, we say $J \leq J'$ if and only if there exist parameterizations $f$ of $J$ and $f'$ of $J'$ such that $f(s) \leq f'(s)$ for all $s \in S^1$.
\end{defn}

This preorder then generates a topology on $\J$ whose open sets are generated by the down-sets $J^\downarrow=\set{J' \in \J \mid J' \leq J} \subset \J$.
If $J$ and $J'$ have parameterizations $f$ and $f'$ respectively such that $f(s)\leq f'(s)$ for all $s \in S^1$, then $f \leq f'$ with respect to the pointwise-order, so $f$ and $f'$ are homotopic by Proposition 14 of \cite{Stong1966}.
By abuse of notation, given two Jordan curves $J$ and $J'$ with parameterizations $f$ and $f'$, respectively, we may use $J \leq J'$ and $f\leq f'$ interchangably. 
It is of interest to note that this recovers the compact-open topology of $\D^{S^1}$.
That is, open sets of the compact-open topology are also open under the pointwise-order topology.

Consider $\D^{S^1}=Map(S^1,\D)$ to be the space of continuous maps from $S^1$ into $\D$, equipped with the traditional compact-open topology.
That is, the subbase for the topology is generated by sets of the form
\[S(K,W)=\set{f\in \D^{S^1}\mid f(K)\subseteq W},\]
where $K$ is a compact subset of $S^1$ and $W$ is an open set of $\D$.
We will show that $S(K,W)$ is open with respect to the partial order on $\D$.
Let $f:S^1 \to \D$ be a map in $S(K,W)$, and let $g \leq f$.
That is, $g(s)\leq f(s) \in W \subseteq \D$ for all $s \in K \subseteq S^1$.
Since $W \subseteq \D$ is open, it is a down-set, so $g(s)\leq f(s)\in W$ implies $g(s)\in W$.
Hence, $S(K,W)$ is open with respect to the partial order of $\D$.

To show that $\J$ is connected under this topology, we first show that every digital Jordan curve is homotopic to a minimal Jordan curve about one of its pure interior points, and that there exists a fence of homotopies between any Jordan curve and the boundary $B \subset \D$.

\begin{thm}\label{big} 
There is a fence of homotopies between any Jordan curve $J$ and the smallest Jordan curve about one of its pure interior points.
\end{thm}

\begin{proof}
The proof idea is to collapse a Jordan curve to a minimal Jordan curve about one of its interior points by incrementally removing points from the interior of the Jordan curve until only one is left.
We present Algorithm \ref{alg:shrink}\footnote{Given an ordered set $J$, we use $J[i]$ to refer to the $i$th element of that set, where $0\leq i\leq |J|-1$.} that removes points from $\jint{J}$ in an order determined by how far they are from a pure fixed basepoint $p \in \jint{J}$.


\begin{figure}
\begin{algorithmic}[1]
\Procedure{Shrink}{$J,f,p$}
\If{$|\jint{J}|=1$} \State \Return $(J,f,p)$ \EndIf
\State $n \gets \max\{d_{\jint{J}}(p,q)\}_{q \in \jint{J}}$
\State $S \gets S_{\jint{J}}(p,n)$ \Comment{The points in $\jint{J}$ of distance $n$ from $p$}
\State $S_8 \gets \{s \in S \mid |A(s)|=8\}$ \Comment{The pure points of distance $n$}
\State $S_4 \gets \{s \in S \mid |A(s)|=4\}$ \Comment{The mixed points of distance $n$}
\If{$S_8 \neq \emptyset$}
	\State $q \gets S_8[0]$
\Else
	\State $q \gets S_4[0]$
\EndIf
	\State $J' \gets A(q)\cap J$ \Comment{The neighbors of $q$ in $J$, ordered clockwise}
	\State $J'' \gets J'-(J'[0] \cup J'[|J'|-1])$ \Comment{Remove the endpoints}
	\For{$t \in S^1$}
		\If{$t \in f^{-1}(J'')$}
			\State $g(t) \gets q$
		\Else
			\State $g(t) \gets f(t)$
		\EndIf
	\EndFor
	\State $K \gets (J-J'')\cup\{q\}$
	\State \Return $(K,g,p)$
\EndProcedure
\end{algorithmic}
\caption{Algorithm for shrinking a Jordan curve such that its interior has one less point}\label{alg:shrink}
\end{figure}

\begin{figure}
    \centering
\begin{tikzpicture}
\foreach \x in {1,2,3}
	{
		\draw[lightgray, thin] (\x,1)--(\x,3);
		\draw[lightgray, thin] (1,\x)--(3,\x);
	}
		\draw[lightgray, thin] (1,1)--(3,3);
		\draw[lightgray, thin] (1,3)--(3,1);

	\draw[blue, ultra thick, dashed] (1,1)--(1,3)--(3,3)--(3,1);	

	\draw[red, ultra thick] (1,1)--(2,2)--(3,1);

	\draw (2,2) node [shape=rectangle,fill=white,draw,label={$q$}]{};

	\draw (1,1) node [shape=circle,fill=white,draw]{};
	\draw (1,3) node [shape=circle,fill=white,draw]{};
	\draw (3,1) node [shape=circle,fill=white,draw]{};
	\draw (3,3) node [shape=circle,fill=white,draw]{};

		\draw (1,2) node [shape=circle,fill=black,scale=.5,draw]{};
		\draw (2,1) node [shape=circle,fill=black,scale=.5,draw]{};
		\draw (2,3) node [shape=circle,fill=black,scale=.5,draw]{};
		\draw (3,2) node [shape=circle,fill=black,scale=.5,draw]{};	
\end{tikzpicture}
\caption{$|A(q) \cap J|=7$}
\label{fig:7}
\end{figure}
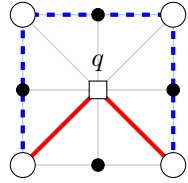

\begin{figure}
    \centering
 \begin{tikzpicture}
\foreach \x in {1,2,3}
	{
		\draw[lightgray, thin] (\x,1)--(\x,3);
		\draw[lightgray, thin] (1,\x)--(3,\x);
	}
		\draw[lightgray, thin] (1,1)--(3,3);
		\draw[lightgray, thin] (1,3)--(3,1);

	\draw[blue, dashed, ultra thick] (1,3)--(3,3)--(3,1);	

	\draw[red, ultra thick] (1,3)--(3,1);

	\draw (2,2) node [shape=rectangle,fill=white,draw,label={$q$}]{};

	\draw (1,1) node [shape=circle,fill=white,draw]{};
	\draw (1,3) node [shape=circle,fill=white,draw]{};
	\draw (3,1) node [shape=circle,fill=white,draw]{};
	\draw (3,3) node [shape=circle,fill=white,draw]{};

		\draw (1,2) node [shape=circle,fill=black,scale=.5,draw]{};
		\draw (2,1) node [shape=circle,fill=black,scale=.5,draw]{};
		\draw (2,3) node [shape=circle,fill=black,scale=.5,draw]{};
		\draw (3,2) node [shape=circle,fill=black,scale=.5,draw]{};	
\end{tikzpicture}
\caption{$|A(q) \cap J|=5$}
\label{fig:5}
\end{figure}
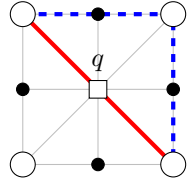

\begin{figure}
    \centering
 \begin{tikzpicture}
\foreach \x in {1,2,3}
	{
		\draw[lightgray, thin] (\x,1)--(\x,3);
		\draw[lightgray, thin] (1,\x)--(3,\x);
	}
		\draw[lightgray, thin] (1,1)--(3,3);
		\draw[lightgray, thin] (1,3)--(3,1);

	\draw[blue, dashed, ultra thick] (1,3)--(3,3);	

	\draw[red, ultra thick] (1,3)--(2,2)--(3,3);

	\draw (2,2) node [shape=rectangle,fill=white,draw,label={$q$}]{};

	\draw (1,1) node [shape=circle,fill=white,draw]{};
	\draw (1,3) node [shape=circle,fill=white,draw]{};
	\draw (3,1) node [shape=circle,fill=white,draw]{};
	\draw (3,3) node [shape=circle,fill=white,draw]{};

		\draw (1,2) node [shape=circle,fill=black,scale=.5,draw]{};
		\draw (2,1) node [shape=circle,fill=black,scale=.5,draw]{};
		\draw (2,3) node [shape=circle,fill=black,scale=.5,draw]{};
		\draw (3,2) node [shape=circle,fill=black,scale=.5,draw]{};	
\end{tikzpicture}
\caption{$|A(q) \cap J|=3$ and $q$ is pure}
\label{fig:3b}
\end{figure}

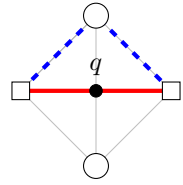
\begin{figure}
    \centering
\begin{tikzpicture}
	\draw[lightgray, thin] (2,1)--(1,2)--(2,3)--(3,2)--(2,1)--(2,3)--(1,2)--(3,2);

	\draw[blue, dashed, ultra thick] (1,2)--(2,3)--(3,2);

	\draw[red, ultra thick] (1,2)--(3,2);
	
	\draw (2,2) node [shape=circle,fill=black,scale=.5,draw,label={$q$}]{};
	\draw (1,2) node [shape=rectangle,fill=white,draw]{};
	\draw (3,2) node [shape=rectangle,fill=white,draw]{};
	\draw (2,1) node [shape=circle,fill=white,draw]{};
	\draw (2,3) node [shape=circle,fill=white,draw]{};

\end{tikzpicture}
\caption{$|A(q)\cap J|=3$ and $q$ is mixed}
\label{fig:3a}
\end{figure}


Figures \ref{fig:7}, \ref{fig:5}, \ref{fig:3b}, and \ref{fig:3a} depict the four possible moves made in Algorithm \ref{alg:shrink}, up to rotation, translation, and parity of the points.
Within each figure, the dashed blue line represents $A(q)\cap J$, and the solid red line depicts part of the Jordan curve output from $\textsc{Shrink}(J,f,p)$ such that $q$ has been removed from its interior.
Of the triple $(K,g,p)=\textsc{Shrink}(J,f,p)$, we need to check that $K$ is a Jordan curve, $f \simeq g$, and $|\jint{K}|<|\jint{J}|$.

\begin{enumerate}
\item $K$ is a Jordan curve:

It is sufficient to check that $|A(k)\cap K|=2$ for all $k \in K$.
By Lemma \ref{lem:maxcon}(a), $J':=A(q)\cap J$ is a connected subset of $J$, and $|J'|$ is odd by Lemma \ref{lem:nodd}.
Let $|J'|=2\ell+1$ for some $\ell \in \{1,2,3\}$.
(We know $\ell \neq 0$ by Lemma \ref{lem:maxcon}(b)-(c).)
By Lemma \ref{lem:components}, $J-J'$ is also a connected subset of $J$.

Assign to $J$ a clockwise ordering $\{j_0,j_1,\hdots,j_{|J|-1}\}$ such that 
\[A(j_i)\cap J = \set{j_{i-1\bmod |J|},j_{i+1\bmod{|J|}}}\]
for all $j_i \in J$ and that $j_0$ is the counter-clockwise-most point of $A(q) \cap J$.
Then $K = \set{j_0,q,j_{2\ell},j_{2\ell+1},\hdots,j_{|J|-1}}$.
To check the ``gluing points,'' see that $|A(j_0)\cap K| = \abs{\set{j_{|J|-1},q}}=2$, and $|A(j_{2\ell}) \cap K| = |\{q,j_{2\ell+1}\}|=2$.
Lastly, $|A(q) \cap K| = |\{j_0,j_{2\ell}\}|=2$.
Hence $K$ is a Jordan curve.

\item $g \simeq f$:

We need to check that either $f(t) \leq g(t)$ for all $t \in S^1$, or that $f(t) \geq g(t)$ for all $t \in S^1$.

If $q$ is pure, suppose without loss of generality that $q$ is closed.
Then $A(q) \cap J = q^\downarrow \cap J$, so $q \geq r$ for all $r \in A(q)\cap J$.
Then $g(t) \geq f(t)$ for all $t \in f^{-1}(J'')$, where $J''$ is $J'$ minus its endpoints.
Since $g(t) = f(t)$ for $t \in S^1-f^{-1}(J'')$, we have $g \geq f$, hence $g \simeq f$.

If $q$ is mixed, then $S_8=\emptyset$. 
By Lemma \ref{lem:maxcon}(c), $|A(q) \cap J|=3$.
Then $|J''|=|J'[1]|=1$.
Suppose without loss of generality that $J'[1]$ is closed such that $q \leq J'[1]$.
Then $g(t) \leq f(t)$ for $t \in f^{-1}(J'')$, and $g(t)=f(t)$ otherwise, so $g \simeq f$.

\item $|\jint{K}|<|\jint{J}|$:

Since $K \subset J \cup \jint{J}$, we have that $\jint{K} \subseteq \jint{J}$ by Proposition \ref{prop:jintk}.
Then $|\jint{K}| \leq |\jint{J}|$, but $q \not\in \jint{J}$, so $|\jint{K}| < |\jint{J}|$.
\end{enumerate}

Because $|\jint{K}| < |\jint{J}|$ after each iteration of the algorithm, it will terminate when $|\jint{K}|=1$.
This iteration process is given in Algorithm \ref{alg:min}.

\begin{figure}
\begin{algorithmic}[1]
\Procedure{Minimalize}{$J,f,p$}
\If{$|\jint{J}|=1$}
    \State \Return $(J,f,p)$
\EndIf
\State $(K,g,p) \gets (J,f,p)$
\While{$|\jint{K}|>1$}
	\State $(K,g,p) \gets \textsc{Shrink}(K,g,p)$
\EndWhile
\State \Return $(K,g,p)$
\EndProcedure
\end{algorithmic}
\caption{Algorithm for shrinking a Jordan curve until it is minimal}\label{alg:min}
\end{figure}

\end{proof}

Iterated applications of this theorem exhibit a fence of homotopies between any digital Jordan curve and some minimal digital Jordan curve about one of the pure points in its interior.
To move between minimal Jordan curves about pure and mixed points, we prove the Proposition \ref{prop:minconn}:

\begin{prop}\label{prop:minconn}
Given any two adjacent points $p,q\in \D$ such that $A(p)\subset \D$ and $A(q)\subset \D$, there exists a homotopy between the Jordan curves $A(p)$ and $A(q)$.
\end{prop}
\begin{proof}
We split this into three cases:
\begin{enumerate}
    \item $p$ is pure and $q$ is mixed,
    \item $p$ is mixed and $q$ is pure, and
    \item $p$ and $q$ are both pure.
\end{enumerate}

For each case (shown in Figures \ref{fig:adjpm}-\ref{fig:adjpp}), we show $A(p)$ with a dashed blue line, and $A(q)$ with a solid line.

\begin{enumerate}
    \item If $p$ is pure and $q$ is mixed, suppose without loss of generality that $p$ is closed.
If $p = (x,y)$, denote its adjacency set as 
\[\{x^-,x,x^+\}\times \{y^-,y,y^+\} - \{(x,y)\};\]
note that this set determines a Jordan curve.
If $q = (u,v)$ we can denote $A(q)$ as 
\[\{(u^-,v),(u,v^+),(u^+,v),(u,v^-)\}.\]
Figure \ref{fig:adjpm} displays $A(p)$ dashed and $A(q)$ solid, where $p$ is the closed point on the left, and $q$ is the mixed point in the middle.
Suppose without loss of generality that $q = (x^+,y)$ such that $p = (u^-,v)$.
Let $f:S^1 \to \D$ be the standard parameterization of $A(p)$ starting at $(x,y^-)$ and traveling clockwise.
We define a parameterization of $A(q)$ as follows:
\[\def\arraystretch{1}
g(t) = \left\{
\begin{array}{ll}
(u^-,v), & t \in f^{-1}((x,y^-)\cup(x^-,y^-)\cup(x^-,y)\cup(x^-,y^+)\cup(x,y^+))\\
(u,v^+), & t \in f^{-1}((x^+,y^+))\\
(u^+,v), & t \in f^{-1}(x^+,y)\\
(u,v^-), & t \in f^{-1}((x^+,y^-))
\end{array}
\right.\]
To see that $g \geq f$ in Figure \ref{fig:adjpm}, observe that
\begin{eqnarray*}
(u^-,v)^\downarrow & \supseteq &\set{(x,y^-),(x^-,y^-),(x^-,y),(x^-,y^+),(x,y^+)}\\
(u,v^+)&=&(x^+,y^+)\\
(u^+,v)& \geq& (x^+,y)\\
(u,v^+) &=& (x^+,y^-).
\end{eqnarray*}
Every point of $A(q)$ is greater than or equal to a point of $A(p)$, and a correspondence between these pairs of points is drawn by their parameterizations.
Hence $g(t)\geq f(t)$ and $A(q) \geq A(p)$ in $\J$.

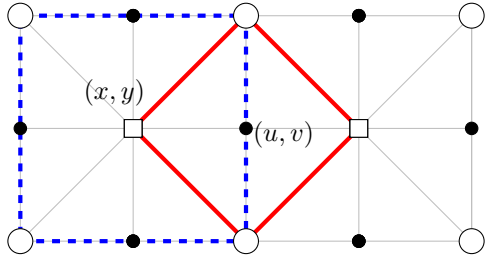
\begin{figure}
    \centering
    \begin{tikzpicture}[scale=1.5]
        \draw[lightgray, thin] (0,0) grid (4,2);
        \draw[lightgray, thin] (0,0)--(2,2)--(4,0)--(4,2)--(2,0)--(0,2);
        \draw[blue,dashed, ultra thick] (0,0)--(2,0)--(2,2)--(0,2)--(0,0);
        \draw[red, ultra thick] (1,1)--(2,0)--(3,1)--(2,2)--(1,1);
        \foreach \x in {0,2,4}{
            \draw (\x,1) node [shape=circle,fill=black,scale=.5,draw]{};
            \foreach \y in {0,2}{
                \draw (\x,\y) node [shape=circle,fill=white,draw]{};
                \draw (\y+1,1) node [shape=rectangle,fill=white,draw] (c\y) {};
                \draw (\y+1,0) node [shape=circle,fill=black,scale=.5,draw]{};
                \draw (\y+1,2) node [shape=circle,fill=black,scale=.5,draw]{};
            }
        }
        \node[label={[shift={(-0.25,0.05)}]{$(x,y)$}}] at (1,1){};
        \node[label={[shift={(0.5,0.-.5)}]{$(u,v)$}}] at (2,1){};
    \end{tikzpicture}
    \caption{The minimal Jordan curves about adjacent pure and mixed points}
    \label{fig:adjpm}
\end{figure}

    \item If $p$ is mixed and $q$ is pure, suppose without loss of generality that $q$ is closed.
    
    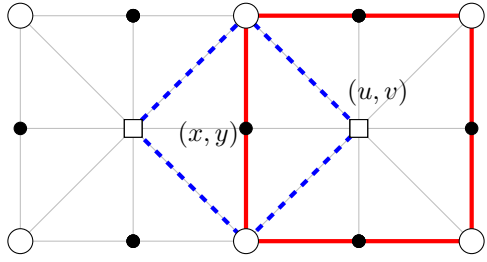
\begin{figure}
    \centering
    \begin{tikzpicture}[scale=1.5]
        \draw[lightgray, thin] (0,0) grid (4,2);
        \draw[lightgray, thin] (0,0)--(2,2)--(4,0)--(4,2)--(2,0)--(0,2);
        \draw[blue, dashed, ultra thick] (1,1)--(2,0)--(3,1)--(2,2)--(1,1);
        \draw[red, ultra thick] (2,0)--(4,0)--(4,2)--(2,2)--(2,0);
        \foreach \x in {0,2,4}{
            \draw (\x,1) node [shape=circle,fill=black,scale=.5,draw]{};
            \foreach \y in {0,2}{
                \draw (\x,\y) node [shape=circle,fill=white,draw]{};
                \draw (\y+1,1) node [shape=rectangle,fill=white,draw] (c\y) {};
                \draw (\y+1,0) node [shape=circle,fill=black,scale=.5,draw]{};
                \draw (\y+1,2) node [shape=circle,fill=black,scale=.5,draw]{};
            }
        }
        \node[label={[shift={(0.25,0.05)}]{$(u,v)$}}] at (3,1){};
        \node[label={[shift={(-0.5,0.-.5)}]{$(x,y)$}}] at (2,1){};
    \end{tikzpicture}
    \caption{The minimal Jordan curves about adjacent mixed and pure points}
    \label{fig:adjmp}
\end{figure}
    
Assume $p = (x,y)$ and $q = (u,v)$, and denote their adjacency sets as above.
Suppose without loss of generality that $q = (x^+,y)$ and consequently that $p = (u^-,v)$.
In Figure \ref{fig:adjmp}, the dashed diamond is the Jordan curve $A(p)$ about the central mixed point $p$, and the solid square is the Jordan curve $A(q)$ about the closed point on the right, $q$.
Let $f:S^1 \to \D$ be the standard parameterization of $A(p)$ starting at $(x^+,y)$ and traveling clockwise, as described in Proposition \ref{prop:pars1}.
We define a parameterization $g:[0,1]/0\sim1 \to A(q)$ as follows: 

\[\def\arraystretch{1}
g(t) = \left\{
\begin{array}{ll}
(u,v^+), & t \in \left[0,\frac{1}{20}\right]\subset f^{-1}((x^+,y))\\
(u^+,v^+), & t \in \left(\frac{1}{20},\frac{1}{10}\right)\subset f^{-1}((x^+,y))\\
(u^+,v), & t \in \left[\frac{1}{10},\frac{3}{20}\right]\subset f^{-1}((x^+,y))\\
(u^+,v^-), & t \in \left(\frac{3}{20}, \frac{1}{5}\right)\subset f^{-1}((x^+,y))\\
(u,v^-), & t \in \left[\frac{1}{5},\frac{1}{4}\right]\subset f^{-1}((x^+,y))\\
(u^-,v^-), & t \in f^{-1}((x,y^-))\\
(u^-,v), & t \in f^{-1}((x^-,y))\\
(u^-,v^+), & t \in f^{-1}((x,y^+))
\end{array}
\right.\]
As in the argument for (1), we may find a correspondence between points in $A(p)$ and the points of $A(q)$ to which they are ``sent.''
In this case, $g(t) \leq f(t)$ for all $t \in [0,1]/0\sim1$, so $g \simeq f$ and $A(q)\leq A(p)\in\J$.

    \item If $p$ and $q$ are both pure, suppose without loss of generality that $p$ is closed and $q \in A(p)$ is open.
If $p = (x,y)$, denote its adjacency set as above.
Define the adjacency set for the pure point $q=(u,v)$ similarly.
Suppose without loss of generality that $q=(x^+,y^+)$ such that $p=(u^-,v^-)$.
Let $f:[0,1]/0\sim1 \to A(p)$ be the standard parameterization of $A(p)$ starting at $(x^-,y^-)$, traveling clockwise.
See, for example, that $f^{-1}((x^-,y^-)) = \left(0,\frac{1}{8}\right)$.
We define a parameterization of $A(q)$ as follows:

\[\def\arraystretch{1}
g(t) = \left\{
\begin{array}{ll}
(u^-,v^-), &t \in f^{-1}((x,y^-)\cup(x^-,y^-)\cup(x^-,y))\\
(u^-,v),& t \in f^{-1}((x^-,y^+))\\
(u^-,v^+),& t \in f^{-1}((x,y^+))\\
(u,v^+), & t \in \left(\frac{1}{2},\frac{13}{24}\right) \subset f^{-1}((x^+,y^+))\\
(u^+,v^+),&t \in \left[\frac{13}{24},\frac{7}{12}\right]\subset f^{-1}((x^+,y^+))\\
(u^+,v),&t \in \left(\frac{7}{12}, \frac{5}{8}\right) \subset f^{-1}((x^+,y^+))\\
(u^+,v^-), & t \in f^{-1}((x^+,y))\\
(u,v^-), & t \in f^{-1}((x^+,y^-))
\end{array}
\right.\]
To see that $g(t)\geq f(t)$ here, consider the following comparison of points from Figure \ref{fig:adjpp}:
\begin{eqnarray*}
(u^-,v^-)^\downarrow &\supseteq & \set{(x,y^-),(x^-,y^-),(x^-,y)}\\
(u^-,v)&\geq&(x^-,y^+)\\
(u^-,v^+) & \geq & (x,y^+)\\
(u,v^+)&\geq&(x^+,y^+)\\
(u^+,v^+)&\geq&(x^+,y^+)\\
(u^+,v)&\geq&(x^+,y^+)\\
(u^+,v^-)&\geq &(x^+,y)\\
(u,v^-)&\geq& (x^+,y^-).
\end{eqnarray*}
This construction yields a continuous function $g:S^1 \to \D$ whose image is $A(q)$.
Furthermore, $g(t) \geq f(t)$ for all $t \in S^1$, so $A(q) \geq A(p)$ in $\J$.

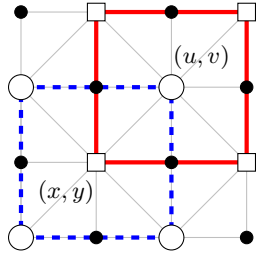
\begin{figure}
    \centering
    \begin{tikzpicture}[yscale=-1]
	\draw[lightgray, thin] (1,1) grid (4,4);
    \draw[lightgray, thin] (1,2)--(3,4)--(4,3)--(2,1)--(1,2)--(1,4)--(4,1);
    \draw[blue, dashed, ultra thick] (1,2)--(3,2)--(3,4)--(1,4)--(1,2);
    \draw[red, ultra thick] (2,1)--(4,1)--(4,3)--(2,3)--(2,1);
    \draw (1,2) node [shape=circle,fill=white,draw]{};
    \draw (1,4) node [shape=circle,fill=white,draw]{};
    \draw (3,2) node [shape=circle,fill=white,draw]{};
    \draw (3,4) node [shape=circle,fill=white,draw]{};
    \draw (2,3) node [shape=rectangle,fill=white,draw]{};
    \draw (2,1) node [shape=rectangle,fill=white,draw]{};
    \draw (4,3) node [shape=rectangle,fill=white,draw]{};
    \draw (4,1) node [shape=rectangle,fill=white,draw]{};
    \draw (2,4) node [shape=circle,fill=black,scale=.5,draw]{};
    \draw (4,4) node [shape=circle,fill=black,scale=.5,draw]{};
    \draw (1,3) node [shape=circle,fill=black,scale=.5,draw]{};
    \draw (3,3) node [shape=circle,fill=black,scale=.5,draw]{};
    \draw (2,2) node [shape=circle,fill=black,scale=.5,draw]{};
    \draw (4,2) node [shape=circle,fill=black,scale=.5,draw]{};
    \draw (1,1) node [shape=circle,fill=black,scale=.5,draw]{};
    \draw (3,1) node [shape=circle,fill=black,scale=.5,draw]{};
    
    \node[label={[shift={(-.4,-.8)}]\small{$(x,y)$}}] at (2,3){};
        \node[label={[shift={(0.4,0)}]\small{$(u,v)$}}] at (3,2){};
\end{tikzpicture}
    \caption{The minimal Jordan curves about two adjacent pure points}
    \label{fig:adjpp}
\end{figure}
\end{enumerate}

Hence, for any two points $p,q$ with $q \in A(p)$, $A(p) \simeq A(q)$.
\end{proof}

It is worth noting that in the proof of Proposition \ref{prop:minconn} above, the homotopy between $A(p)$ and $A(q)$ passes through only those two Jordan curves.
That is, the homotopy is within the subspace of minimal Jordan curves.
This allows us to prove Theorem \ref{main}, which asserts that the space of digital Jordan curves is path-connected.

\begin{proof}[Proof of Theorem \ref{main}.]
Fix some $q\in\D$ such that the Jordan curve $A(q)\subset\D$ as well.
We will prove that there exists a path from any Jordan curve $J \in \J$ to $A(q)$.
Let $J \in J$ be any Jordan curve, and $f$ its parameterization. 
If $J$ is minimal, there exists a path from $J$ to $A(q)$ by Proposition \ref{prop:minconn} and we are done.
If $J$ is not minimal, there exists a pure point $p \in \jint{J}$ by Lemma \ref{lem:nonminpureint}.
Consider Algorithm \ref{alg:min}.
By Theorem \ref{big}, $\textsc{Minimalize}(J,f,p)$ yields a path from $J$ to $A(p)$.
By Proposition \ref{prop:minconn} again, there exists a path from $A(p)$ to $A(q)$.
Hence $\J$ is connected.
\end{proof}

\begin{thm}
\label{thm:DT0}
Given a Khalimsky plane $\D$, $\J(\D)$ is $T_0$.
\end{thm}
\begin{proof}
Suppose for sake of contradiction that $\J$ is not $T_0$.
Then there exist two Jordan curves $J,J' \in J$ such that $J \leq J'$ and $J' \leq J$ but $J \neq J'$.
If $J \leq J'$, then there exists a parameterization $f$ of $J$ and $f'$ of $J'$ such that $f(t)\leq f'(t)$ for all $t \in S^1$.
Similarly, $f'(t) \leq f(t)$ for all $t \in S^1$.
Since $\D$ is $T_0$, however, $f'(t) \leq f(t)$ and $f(t)\leq f'(t)$ implies $f(t)=f'(t)$ for all $t \in S^1$, so $J = J'$.
\end{proof}

\begin{proof}[Proof of Corollary \ref{cor:weakcontra}]
After each application of Algorithm \ref{alg:shrink} in Algorithm \ref{alg:min}, a point $q$ is removed from $\jint{J}$.
For each $q$ removed, $A(q) \cap J$ is connected, so by Corollary \ref{cor:weak}, $q$ is a weak point of $\jint{J}$.
Consider $(K,g,p) = \textsc{Shrink}(J,f,p)$ such that $q$ is the point that has been deleted from $\jint{J}$.
By Proposition 4.2.4 of \cite{Barmak2012}, the inclusion map
\[\iota:\jint{K} = \jint{J}-\{q\} \hookrightarrow \jint{J}\]
is a weak homotopy equivalence.
Since Algorithm \ref{alg:min} removes weakpoints from $\jint{J}$ until $|\jint{J}|=1$, it follows that $\iota:\{p\} \hookrightarrow\jint{J}$ is a weak homotopy equivalence, so $\jint{J}$ is weakly contractible.
\end{proof}

We suspect that $\jint{J}$ is also contractible, however, the points removed in Algorithm \ref{alg:shrink} are not always beat, so we cannot exhibit a sequence of beat points to remove such that at each step, what remains is still the interior of a Jordan curve.

\begin{thm}\label{thm:parD2} 
For all $J \in \J(\D)$, there exists a continuous surjection 
\[f:D^2 \to J\cup\jint{J}\]
such that $\partial D^2$ is sent to $J$.
\end{thm}

It is worth noting that $f^{-1}(J) \neq \partial D^2$, necessarily.

\begin{proof}
Given any non-minimal Jordan curve $J_0 \in \J$, Algorithm \ref{alg:min} yields a fence $J_0 \lessgtr J_1 \lessgtr \hdots \lessgtr J_n$, where $J_n=A(p)$ for some pure $p \in \D-B$.
Let $f_0,f_1,\hdots,f_n$ be their parameterizations by $S^1$, respectively.
By construction, in each step from $f_i$ to $f_{i+1}$, one of four moves is performed, up to parity and rotation of the points.
See Figures \ref{fig:7}, \ref{fig:5}, \ref{fig:3b}, and \ref{fig:3a} for visual representations of these moves.
We will show that the fence $f_0 \lessgtr f_1 \lessgtr \hdots \lessgtr f_n$ extends to a fence $\hat{f}_n \lessgtr \hat{f}_{n-1} \lessgtr \hdots \lessgtr \hat{f}_0$ such that $\hat{f}_i:D^2 \to J_i \cup \jint{J_i}$ is a parameterization with $\hat{f}_i|_{S^1} = J_i$.

First, we show that the minimal Jordan curve and its interior, $A(p)\cup\{p\}$, admits a parameterization by $D^2$ such that the boundary is mapped to $A(p)$.
Consider $\paren{0,0}\in D^2 = \set{\paren{x,y}\mid x^2+y^2 \leq 1} \subset \rr^2$.
If $p$ is open, define $\hat{f}_n$ such that
\[\hat{f}_n^{-1}(p):=\set{\paren{x,y}\mid x^2+y^2 < \frac{1}{2}}.\]
Then there exists a homeomorphism $\varphi:S^1\times\cots{\frac{1}{2},1}\to D^2-\hat{f}_n^{-1}(p)$.
Using the parameterization $f_n:S^1\to J_n$ from Theorem \ref{prop:pars1}, take
\[\hat{f}_n^{-1}(j) := \varphi\paren{f_n^{-1}(j)\times\cots{\frac{1}{2},1}},\]
for all $j\in J_n=A(p)$.
Since $p$ is open, the points of $J_n$ are either closed or mixed.
If $j \in J_n$ is closed, then $f_n^{-1}(j)$ is a closed subset of $S^1$,
so $\varphi\paren{f_n^{-1}(j)\times \cots{\frac{1}{2},1}}$ is a closed subset of $D^2$.
Hence, $\hat{f}_n$ is continuous and ${\hat{f}_n}|_{S^1} = J_n$.
Furthermore, notice that for all $j \in J_n$, $\hat{f}_n^{-1}(j)$ is homeomorphic the product of two intervals.
A similar construction works if $p$ is closed.
Hence, $A(p)$ admits a parameterization by $D^2$ for $p$ pure.

Suppose that $J_i\cup\jint{J_i}$ admits a parameterization $\hat{f}_i$ by $D^2$ such that $\hat{f}_i|_{S^1}=J_i$ and $\hat{f}_i^{-1}(x)$ is homeomorphic to the product of two intervals for all $x \in \img{\hat{f}_i}$.
Consider $J_{i-1}$.
The move from $J_{i-1}$ to $J_i$ removes one point from $\jint{J_{i-1}}$ to construct $J_i$; call it $q$.
If $q$ is pure, suppose without loss of generality that $q$ is closed.
The move from $J_{i-1}$ to $J_i$ is one of Figures \ref{fig:7}, \ref{fig:5}, or \ref{fig:3b}.
Because $q \in J_i$, $\hat{f}_i^{-1}(q)$ is a closed subset of $D^2$ that is homeomorphic to $[a,b]\times[\varepsilon,1]$ for some $\cots{a,b} \subset S^1$ and  $0<\varepsilon<1$ by the inductive hypothesis.
Denote the homeomorphism by $\varphi:\cots{a,b}\times \cots{\varepsilon,1} \to \hat{f}_i^{-1}(q)\subset D^2$ such that $\varphi|_{\cots{a,b}\times\set{1}}=\hat{f}_i^{-1}(q) \cap S^1$.
If $q$ is closed, then $f_{i-1}\paren{f_i^{-1}(q)}$ is a closed subset of $J_{i-1}$.
By Lemma \ref{lem:maxcon}, $f_{i-1}\paren{f_i^{-1}(q)}$ is a connected subset of $J_{i-1}$, so it is a COTS-arc by Lemma 5.2(a) of \cite{Khalimsky1990a}.
By Lemma \ref{lem:nodd}, $\abs{f_{i-1}\paren{f_i^{-1}(q)}}$ is odd. 
Denote those points $\cots{j_1,\hdots,j_\ell}\subset J_{i-1}$, and recall that $j_k$ is mixed if $k$ is odd and open if $k$ is even for $j_k \in \cots{j_1,j_\ell}$.
We define $\hat{f}_{i-1}:D^2 \to J_{i-1}\cup\jint{J_{i-1}}$ by:
\[\hat{f}_{i-1}(t)=\left\{\begin{array}{ll}
    j_1, & t \in \varphi\paren{\cots{a,a+\frac{b-a}{\ell}}\times \left(\frac{1+\varepsilon}{2},1\right] } \\
    j_2, & t \in \varphi\paren{\paren{a+\frac{b-a}{\ell},a+\frac{2(b-a)}{\ell}}\times \left(\frac{1+\varepsilon}{2},1\right]} \\
    \vdots&\\
    j_\ell, & t \in \varphi\paren{\cots{a+\frac{(\ell-1)(b-a)}{\ell},b}\times \left(\frac{1+\varepsilon}{2},1\right]} \\
    q, & t \in \varphi\paren{\cots{a,b} \times \cots{\varepsilon,\frac{1+\varepsilon}{2}}}\\
    \hat{f}_i(t),&\textrm{ else.}
\end{array}
\right.\]
First, we check that $\hat{f}_{i-1}$ is continuous.
If $q$ is closed, $\hat{f}_{i-1}^{-1}(q)=\varphi\paren{\cots{a,b} \times \cots{\varepsilon,\frac{1+\varepsilon}{2}}}$, which is a closed subset of $D^2$.
Since $q$ is closed, $j_k \in \cots{j_1,j_\ell}$ is open for $k$ even.
By construction, $\hat{f}_{i-1}^{-1}\paren{j_k} = \varphi\paren{\paren{a+\frac{(k-1)(b-a)}{\ell},a+\frac{k(b-a)}{\ell)}}\times \left(\frac{1+\varepsilon}{2},1\right]}$, which is an open subset of $D^2$ since $\paren{a+\frac{(k-1)(b-a)}{\ell},a+\frac{k(b-a)}{\ell}}\times\set{1} = S^1 \cap \hat{f}_{i-1}^{-1}\paren{j_k}$.

If $q$ is mixed, the move from $J_{i-1}$ to $J_i$ is shown in Figure \ref{fig:3a}, up to parity of the points.
If $J_i$ is the solid Jordan curve as depicted in Figure \ref{fig:3a}, $f_i^{-1}(q)=(a,b)\subset S^1$ is an open subset.
Because $\hat{f}_i^{-1}\paren{q^\uparrow}$ is a closed subset of $D_2$, $\hat{f}_i^{-1}(q)$ must be homeomorphic to something of the form $(a,b) \times [\varepsilon,1]$ for some $0<\varepsilon<1$.
As before, denote the homeomorphism by $\varphi:(a,b)\times [\varepsilon,1] \to \hat{f}_i^{-1}(q) \subset D^2$.
By the inductive hypothesis, this satisfies $\varphi_{(a,b)\times\{1\}} = \hat{f}_i^{-1}(q)\cap S^1$.
If $q$ is as in Figure \ref{fig:3a}, then $f_{i-1}\paren{f_i^{-1}(q)}$ is a single open point of $J_{i-1}$.
We define $\hat{f}_{i-1}:D^2 \to J_{i-1}\cup\jint{J_{i-1}}$ as follows.
\[\hat{f}_{i-1}(t) = \left\{\begin{array}{rl}
   f_{i-1}\paren{f_i^{-1}(q)},  & t \in \varphi\paren{(a,b)\times \left(\frac{1+\varepsilon}{2},1\right]} \\
    q, & t \in \varphi\paren{(a,b)\times \cots{\varepsilon,\frac{1+\varepsilon}{2}}} \\
    \hat{f}_i(t), & \textrm{else.}
\end{array} \right.\]
Since the preimage of $f_{i-1}\paren{f_i^{-1}(q)}$ is open, and since $\hat{f}_{i-1}^{-1} = \hat{f}_i^{-1}$ on all other open points of $J_{i-1}\cup\jint{J_{i-1}}$, $\hat{f}_{i-1}$ is continuous.
Lastly, $\hat{f}_{i-1}\paren{\varphi\paren{(a,b)\times\{1\}}} \supset f_{i-1}\paren{f_i^{-1}(q)} \cap J_{i-1}$.
Hence $\hat{f}_{i-1}|_{S_1}=J_{i-1}$.

Lastly, we show that $\hat{f}_{i-1}\simeq \hat{f}_i$.
Since $\hat{f}_{i-1}(t)=\hat{f}_i(t)$ almost everywhere, we only need to check that $\hat{f}_{i-1}(t) \simeq \hat{f}_i(t)$ for $t \in \hat{f}_{i}^{-1}\paren{q}$.
If $q$ is as shown in Figures \ref{fig:7}, \ref{fig:5}, \ref{fig:3b}, or \ref{fig:3a}, then $r \leq q$ for all $r \in f_{i-1}\paren{f_i^{-1}(q)}$.
Since $\hat{f}_{i-1}\paren{\hat{f}_i^{-1}(q)}= \{q\}\cup A(q)\cap J_{i-1}$, it follows that $\hat{f}_{i-1} \leq \hat{f}_i$.

Given a fence of Jordan curves $J_1 \lessgtr J_{2} \lessgtr \hdots \lessgtr J_n$ as generated by Algorithm \ref{alg:min}, $J_n \neq A(p)$ for $p$ mixed.
It remains to be shown there exists a parameterization $\hat{f}:D^2 \to A(p)\cup\{p\}$.
Let $A(p)\cup\{p\}$ as shown in Figure \ref{fig:amixed}.
If $p=(x,y)$, denote the points of $A(p)$ as $\set{(x,y^+),(x^+,y),(x,y^-),(x^-,y)}$.
Let $\varphi:I \times I \to D^2$ be a homeomorphism such that $\varphi_{\partial(I \times I)}=S^1$.
Then we define a map $\hat{f}:D^2 \to A(p)\cup\{p\}$ as follows.
\[\hat{f}(t)=\left\{\begin{array}{rl}
    (x,y^+), & t \in \varphi\paren{[0,1]\times\cots{\frac{2}{3},1}} \\
    (x^+,y), & t \in \varphi\paren{\left(\frac{2}{3},1\right]\times\paren{\frac{1}{3},\frac{2}{3}}} \\
    (x,y^-), & t \in \varphi\paren{[0,1]\times\cots{1,\frac{1}{3}}} \\
    (x^-,y), & t \in \varphi\paren{\left[0,\frac{1}{3}\right)\times\paren{\frac{1}{3},\frac{2}{3}}} \\
    (x,y), & \textrm{else.}
\end{array}\right.\]
It is easy to check that $\hat{f}$ is continuous and that $\hat{f}|_{S^1}=A(p)$.

\end{proof}

\subsection{Enumerating Jordan Curves}

While explicitly computing the topological complexity of a space of digital Jordan curves may be difficult, we can get estimates by showing a correspondence with other spaces, or by enumerating the Jordan curves and counting the maximal elements in the space's Hasse diagram.

\begin{defn}
Let
\[\J_1(\D) = \left\{J \subset \D \mid |\jint{J}|=1 \right\}\]
denote the \textbf{space of minimal Jordan curves}.
Equivalently, \[\J_1(\D) = \set{J \in \J(\D) \mid J=A(p) \textrm{ for some } p \in \D-B}.\]
\end{defn}

\begin{thm}
$\tc{}{\J_1(\D)}=1$.
\end{thm}
\begin{proof} 
The proof sketch is as follows: we show $\J_1(\D)$ is contractible by showing $\J_1(\D)$ is homeomorphic to a digital plane and applying Theorem 1 of \cite{Farber2001}.

First we show $\J_1(\D)$ is contractible.
Given an $m\times n$ digital plane $\D$, if $J \subset \D$ is a minimal Jordan curve, then $\jint{J}$ lies is the $(m-2)\times(n-2)$ digital plane $\D'\subseteq \D-B$.
Clearly, $\D'$ is contractible because it is the product of two COTS.
Consider $\op{\D'}$.
By Proposition \ref{prop:dualx}, $\op{\D'}$ is also a digital plane, whose open and closed points have been swapped.
For example, if $X$ is the digital plane shown in Figure \ref{fig:pure}, then $\op{X}$ is the digital plane shown in Figure \ref{fig:purec}.
By abuse of notation, $\op{p}\in \op{\D'}$ will refer to the point of $\op{\D'}$ that has the same coordinates and neighbors as $p \in \D'$, but with the opposite ordering.

There exists an inclusion map $\iota:\J_1(\D) \hookrightarrow \op{\D'}$ given by $\iota(J) \mapsto \jint{J} \in \op{\D'}$.
That is, $\iota(A(p)) \mapsto \{\op{p}\}$.
To see that $\iota$ is continuous, consider $J \leq K$ in $\J_1(\D)$.
Because $J$ and $K$ are minimal, each are the border of one of Figures \ref{fig:pure}, \ref{fig:purec}, or \ref{fig:amixed}, up to rotation and translation.
The three cases in the proof of Proposition \ref{prop:minconn} demonstrate the three ways two minimal Jordan curves be adjacent to one another.
If $J \leq K$ in $\J_1(\D)$, each are of the form $A(p)$ and $A(q)$, respectively.
Since $\J_1(\D) \subseteq \J(\D)$ is $T_0$, if $J$ and $K$ are distinct, then $J < K$, in fact.
Then there are three cases for $J=A(p) < A(q) = K$:
\begin{enumerate}
    \item $p$ is closed and $q$ is mixed
    \item $p$ is closed and $q$ is open
    \item $p$ is mixed and $q$ is open.
\end{enumerate}
It is easy to see that in each of the three cases above, $p > q$, since closed points are greater than mixed points are greater than open points with respect to the order topology on $\D$.
Hence, if $A(p) < A(q)$ in $\J_1(\D)$, then $p > q \in \D'$, so $\op{p} < \op{q} \in \op{\D'}$, so $\iota(A(p)) < \iota(A(q)) \in \op{\D'}$. 
To see that $\iota$ is surjective, consider a point $\op{p} \in \op{\D'}$.
Since $\D' \subseteq \D-B$, $p$ is a point of $\D'$ such that $A(p) \subset \D$, so $A(p) \in \J_1(\D)$.

Next, we consider an inverse map $\alpha:\op{\D'} \to \D' \to \J_1(\D)$ define by $\alpha(\op{p})=A(p)\in \J_1(\D)$.
Consider $\op{p} < \op{q}$ to be two distinct comparable points in $\op{\D'}$.
By an argument similar to the previous case:
\begin{eqnarray*}
\op{p}&<&\op{q} \in \op{\D'}\\
p&>&q \in \D'\\
A(p)&<&A(q)\in \J_1(\D).
\end{eqnarray*}

This shows that $\alpha$ is order-preserving and therefore continuous.
Then $\alpha \circ \iota (J)=\alpha(\op{\jint{J}})=A(\jint{J})=J$,
and $\iota\circ\alpha(\op{p})=\iota(A(p))=\op{p}$.
Hence $\alpha\circ\iota \simeq 1_{\J_1(\D)}$ and $\iota\circ\alpha \simeq 1_{\D'}$.
Then $\J_1(\D)$ is homeomorphic to a contractible space, and the result follows.
\end{proof}

Enumerating Jordan curves and determining the topology of the resulting space is a straightforward way to explicitly determine the topological complexity.
Just as we showed the space of minimal Jordan curves was homotopy equivalent to a contractible space, we can do the same for Jordan curves in $4 \times 4$ and $5 \times 5$ digital planes.

All COTS $X$ of length 4 are homeomorphic because, up to reflection, they must each start with an open point and end with a closed point.
(Notice that if a COTS $X$ has an odd number of points, then the first and last points must be of the same type, so $X \not\cong X^{op}$.)
Consequently, there is a unique $4\times 4$ digital plane equipped with the Khalimsky topology, up to rotation.

\begin{figure}
    \centering
    \begin{tikzpicture}
	\draw[lightgray, thin] (1,1) grid (4,4);
    \draw[lightgray, thin] (1,2)--(3,4)--(4,3)--(2,1)--(1,2)--(1,4)--(4,1);
    \draw (1,2) node [shape=circle,fill=white,draw]{};
    \draw (1,4) node [shape=circle,fill=white,draw]{};
    \draw (3,2) node [shape=circle,fill=white,draw]{};
    \draw (3,4) node [shape=circle,fill=white,draw]{};
    \draw (2,3) node [shape=rectangle,fill=white,draw]{};
    \draw (2,1) node [shape=rectangle,fill=white,draw]{};
    \draw (4,3) node [shape=rectangle,fill=white,draw]{};
    \draw (4,1) node [shape=rectangle,fill=white,draw]{};
    \draw (2,4) node [shape=circle,fill=black,scale=.5,draw]{};
    \draw (4,4) node [shape=circle,fill=black,scale=.5,draw]{};
    \draw (1,3) node [shape=circle,fill=black,scale=.5,draw]{};
    \draw (3,3) node [shape=circle,fill=black,scale=.5,draw]{};
    \draw (2,2) node [shape=circle,fill=black,scale=.5,draw]{};
    \draw (4,2) node [shape=circle,fill=black,scale=.5,draw]{};
    \draw (1,1) node [shape=circle,fill=black,scale=.5,draw]{};
    \draw (3,1) node [shape=circle,fill=black,scale=.5,draw]{};
\end{tikzpicture}
    \caption{The $4\times 4$ Khalimsky digital plane}
    \label{fig:4x4}
\end{figure}
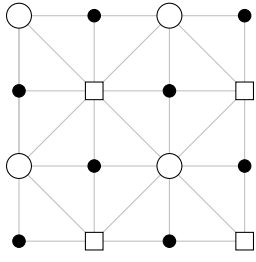

This digital plane $\D_{4\times4}$ shown in Figure \ref{fig:4x4} has four points whose adjacency neighborhoods (i.e., their minimal Jordan curves) are subsets of $\D_{4\times4}$, so $|\J(\D_{4\times4})|\geq 4$.
The adjusted border is also a Jordan curve, whose interior is all four points mentioned above.
All in all, $|\J(\D_{4\times4})|=11$, and we've displayed the Jordan curves in Figure \ref{fig:J4x4}.

\begin{figure}
\centering
\resizebox{!}{7in}{
\begin{tikzpicture}
	\node (a) at (0,0) {
	\begin{tikzpicture}[open/.style={shape=circle,fill=white,scale=.75,draw}, mixed/.style={shape=circle,fill=gray,scale=.5,draw}, closed/.style={shape=rectangle,fill=white,draw}]

	\node[closed] (c1) at (2,3){};

	\node[mixed,above=of c1] (m1){};
	\node[mixed,below=of c1] (m2){};
	\node[mixed,left=of c1] (m3) {};
	\node[mixed,right=of c1] (m4) {};

	\node[open,above left=of c1] (o0) {};
	\node[open,above right=of c1] (o2){};
	\node[open,below left=of c1] (o3){};
	\node[open,below right=of c1] (o1) {};

	\node[closed, above right=of o1] (c2) {};
	\node[closed, below right=of o1] (c4) {};
	\node[closed, below left=of o1] (c3) {};
	\node[mixed,right=of o1] (m5){};
	\node[mixed,right=of c3] (m6){};
	\node[mixed,above=of c2] (m7){};
	\node[mixed,left=of c3] (m8){};

	\draw[lightgray, thin, draw] (o0)--(m3)--(o3)--(m8)--(c3)--(m2)--(c1)--(m1)--(o2)--(m4)--(o1)--(m6)--(c4)--(m5)--(c2)--(m7)--(o2)--(m1)--(o0)--(m3)--(c1)--(m4)--(c2)--(m5)--(o1)--(m2)--(o3)--(m8)--(c3)--(m6)--(c4)--(o1)--(c1)--(o0)--(m1)--(o2)--(c1)--(o3)--(c3)--(o1)--(c2)--(o2);
	\draw[blue,ultra thick] (c1)--(m2)--(c3)--(m6)--(c4)--(m5)--(c2)--(m4)--(c1);
\end{tikzpicture}
};
    \node[below left=of a] (b) {
    \begin{tikzpicture}[open/.style={shape=circle,fill=white,scale=.75,draw}, mixed/.style={shape=circle,fill=gray,scale=.5,draw}, closed/.style={shape=rectangle,fill=white,draw}]

	\node[closed] (c1) at (2,3){};

	\node[mixed,above=of c1] (m1){};
	\node[mixed,below=of c1] (m2){};
	\node[mixed,left=of c1] (m3) {};
	\node[mixed,right=of c1] (m4) {};

	\node[open,above left=of c1] (o0) {};
	\node[open,above right=of c1] (o2){};
	\node[open,below left=of c1] (o3){};
	\node[open,below right=of c1] (o1) {};

	\node[closed, above right=of o1] (c2) {};
	\node[closed, below right=of o1] (c4) {};
	\node[closed, below left=of o1] (c3) {};
	\node[mixed,right=of o1] (m5){};
	\node[mixed,right=of c3] (m6){};
	\node[mixed,above=of c2] (m7){};
	\node[mixed,left=of c3] (m8){};

	\draw[lightgray, thin, draw] (o0)--(m3)--(o3)--(m8)--(c3)--(m2)--(c1)--(m1)--(o2)--(m4)--(o1)--(m6)--(c4)--(m5)--(c2)--(m7)--(o2)--(m1)--(o0)--(m3)--(c1)--(m4)--(c2)--(m5)--(o1)--(m2)--(o3)--(m8)--(c3)--(m6)--(c4)--(o1)--(c1)--(o0)--(m1)--(o2)--(c1)--(o3)--(c3)--(o1)--(c2)--(o2);
	\draw[blue,ultra thick] (o3)--(c3)--(m6)--(c4)--(m5)--(c2)--(m4)--(c1)--(o3);
\end{tikzpicture}};
    \node[below right=of a] (c) {
    \begin{tikzpicture}[open/.style={shape=circle,fill=white,scale=.75,draw}, mixed/.style={shape=circle,fill=gray,scale=.5,draw}, closed/.style={shape=rectangle,fill=white,draw}]

	\node[closed] (c1) at (2,3){};

	\node[mixed,above=of c1] (m1){};
	\node[mixed,below=of c1] (m2){};
	\node[mixed,left=of c1] (m3) {};
	\node[mixed,right=of c1] (m4) {};

	\node[open,above left=of c1] (o0) {};
	\node[open,above right=of c1] (o2){};
	\node[open,below left=of c1] (o3){};
	\node[open,below right=of c1] (o1) {};

	\node[closed, above right=of o1] (c2) {};
	\node[closed, below right=of o1] (c4) {};
	\node[closed, below left=of o1] (c3) {};
	\node[mixed,right=of o1] (m5){};
	\node[mixed,right=of c3] (m6){};
	\node[mixed,above=of c2] (m7){};
	\node[mixed,left=of c3] (m8){};

	\draw[lightgray, thin, draw] (o0)--(m3)--(o3)--(m8)--(c3)--(m2)--(c1)--(m1)--(o2)--(m4)--(o1)--(m6)--(c4)--(m5)--(c2)--(m7)--(o2)--(m1)--(o0)--(m3)--(c1)--(m4)--(c2)--(m5)--(o1)--(m2)--(o3)--(m8)--(c3)--(m6)--(c4)--(o1)--(c1)--(o0)--(m1)--(o2)--(c1)--(o3)--(c3)--(o1)--(c2)--(o2);
	\draw[blue,ultra thick] (c1)--(m2)--(c3)--(m6)--(c4)--(m5)--(c2)--(o2)--(c1);
\end{tikzpicture}};
    \node[below right=of b] (d) {
    \begin{tikzpicture}[open/.style={shape=circle,fill=white,scale=.75,draw}, mixed/.style={shape=circle,fill=gray,scale=.5,draw}, closed/.style={shape=rectangle,fill=white,draw}]

	\node[closed] (c1) at (2,3){};

	\node[mixed,above=of c1] (m1){};
	\node[mixed,below=of c1] (m2){};
	\node[mixed,left=of c1] (m3) {};
	\node[mixed,right=of c1] (m4) {};

	\node[open,above left=of c1] (o0) {};
	\node[open,above right=of c1] (o2){};
	\node[open,below left=of c1] (o3){};
	\node[open,below right=of c1] (o1) {};

	\node[closed, above right=of o1] (c2) {};
	\node[closed, below right=of o1] (c4) {};
	\node[closed, below left=of o1] (c3) {};
	\node[mixed,right=of o1] (m5){};
	\node[mixed,right=of c3] (m6){};
	\node[mixed,above=of c2] (m7){};
	\node[mixed,left=of c3] (m8){};

	\draw[lightgray, thin, draw] (o0)--(m3)--(o3)--(m8)--(c3)--(m2)--(c1)--(m1)--(o2)--(m4)--(o1)--(m6)--(c4)--(m5)--(c2)--(m7)--(o2)--(m1)--(o0)--(m3)--(c1)--(m4)--(c2)--(m5)--(o1)--(m2)--(o3)--(m8)--(c3)--(m6)--(c4)--(o1)--(c1)--(o0)--(m1)--(o2)--(c1)--(o3)--(c3)--(o1)--(c2)--(o2);
	\draw[blue,ultra thick] (c1)--(o3)--(c3)--(m6)--(c4)--(m5)--(c2)--(o2)--(c1);
\end{tikzpicture}};
    \node[below left=of d] (e) {
    \begin{tikzpicture}[open/.style={shape=circle,fill=white,scale=.75,draw}, mixed/.style={shape=circle,fill=gray,scale=.5,draw}, closed/.style={shape=rectangle,fill=white,draw}]

	\node[closed] (c1) at (2,3){};

	\node[mixed,above=of c1] (m1){};
	\node[mixed,below=of c1] (m2){};
	\node[mixed,left=of c1] (m3) {};
	\node[mixed,right=of c1] (m4) {};

	\node[open,above left=of c1] (o0) {};
	\node[open,above right=of c1] (o2){};
	\node[open,below left=of c1] (o3){};
	\node[open,below right=of c1] (o1) {};

	\node[closed, above right=of o1] (c2) {};
	\node[closed, below right=of o1] (c4) {};
	\node[closed, below left=of o1] (c3) {};
	\node[mixed,right=of o1] (m5){};
	\node[mixed,right=of c3] (m6){};
	\node[mixed,above=of c2] (m7){};
	\node[mixed,left=of c3] (m8){};

	\draw[lightgray, thin, draw] (o0)--(m3)--(o3)--(m8)--(c3)--(m2)--(c1)--(m1)--(o2)--(m4)--(o1)--(m6)--(c4)--(m5)--(c2)--(m7)--(o2)--(m1)--(o0)--(m3)--(c1)--(m4)--(c2)--(m5)--(o1)--(m2)--(o3)--(m8)--(c3)--(m6)--(c4)--(o1)--(c1)--(o0)--(m1)--(o2)--(c1)--(o3)--(c3)--(o1)--(c2)--(o2);
	\draw[blue,ultra thick] (c1)--(o3)--(c3)--(o1)--(c1);
\end{tikzpicture}};
    \node[below=of d] (f) {
    \begin{tikzpicture}[open/.style={shape=circle,fill=white,scale=.75,draw}, mixed/.style={shape=circle,fill=gray,scale=.5,draw}, closed/.style={shape=rectangle,fill=white,draw}]

	\node[closed] (c1) at (2,3){};

	\node[mixed,above=of c1] (m1){};
	\node[mixed,below=of c1] (m2){};
	\node[mixed,left=of c1] (m3) {};
	\node[mixed,right=of c1] (m4) {};

	\node[open,above left=of c1] (o0) {};
	\node[open,above right=of c1] (o2){};
	\node[open,below left=of c1] (o3){};
	\node[open,below right=of c1] (o1) {};

	\node[closed, above right=of o1] (c2) {};
	\node[closed, below right=of o1] (c4) {};
	\node[closed, below left=of o1] (c3) {};
	\node[mixed,right=of o1] (m5){};
	\node[mixed,right=of c3] (m6){};
	\node[mixed,above=of c2] (m7){};
	\node[mixed,left=of c3] (m8){};

	\draw[lightgray, thin, draw] (o0)--(m3)--(o3)--(m8)--(c3)--(m2)--(c1)--(m1)--(o2)--(m4)--(o1)--(m6)--(c4)--(m5)--(c2)--(m7)--(o2)--(m1)--(o0)--(m3)--(c1)--(m4)--(c2)--(m5)--(o1)--(m2)--(o3)--(m8)--(c3)--(m6)--(c4)--(o1)--(c1)--(o0)--(m1)--(o2)--(c1)--(o3)--(c3)--(o1)--(c2)--(o2);
	\draw[blue,ultra thick] (o0)--(m3)--(o3)--(c3)--(m6)--(c4)--(m5)--(c2)--(o2)--(m1)--(o0);
\end{tikzpicture}};
    \node[below right=of d] (g) {
    \begin{tikzpicture}[open/.style={shape=circle,fill=white,scale=.75,draw}, mixed/.style={shape=circle,fill=gray,scale=.5,draw}, closed/.style={shape=rectangle,fill=white,draw}]

	\node[closed] (c1) at (2,3){};

	\node[mixed,above=of c1] (m1){};
	\node[mixed,below=of c1] (m2){};
	\node[mixed,left=of c1] (m3) {};
	\node[mixed,right=of c1] (m4) {};

	\node[open,above left=of c1] (o0) {};
	\node[open,above right=of c1] (o2){};
	\node[open,below left=of c1] (o3){};
	\node[open,below right=of c1] (o1) {};

	\node[closed, above right=of o1] (c2) {};
	\node[closed, below right=of o1] (c4) {};
	\node[closed, below left=of o1] (c3) {};
	\node[mixed,right=of o1] (m5){};
	\node[mixed,right=of c3] (m6){};
	\node[mixed,above=of c2] (m7){};
	\node[mixed,left=of c3] (m8){};

	\draw[lightgray, thin, draw] (o0)--(m3)--(o3)--(m8)--(c3)--(m2)--(c1)--(m1)--(o2)--(m4)--(o1)--(m6)--(c4)--(m5)--(c2)--(m7)--(o2)--(m1)--(o0)--(m3)--(c1)--(m4)--(c2)--(m5)--(o1)--(m2)--(o3)--(m8)--(c3)--(m6)--(c4)--(o1)--(c1)--(o0)--(m1)--(o2)--(c1)--(o3)--(c3)--(o1)--(c2)--(o2);
	\draw[blue, ultra thick] (o2)--(c1)--(o1)--(c2)--(o2);
\end{tikzpicture}};
    \node[below=of f] (h) {
    \begin{tikzpicture}[open/.style={shape=circle,fill=white,scale=.75,draw}, mixed/.style={shape=circle,fill=gray,scale=.5,draw}, closed/.style={shape=rectangle,fill=white,draw}]

	\node[closed] (c1) at (2,3){};

	\node[mixed,above=of c1] (m1){};
	\node[mixed,below=of c1] (m2){};
	\node[mixed,left=of c1] (m3) {};
	\node[mixed,right=of c1] (m4) {};

	\node[open,above left=of c1] (o0) {};
	\node[open,above right=of c1] (o2){};
	\node[open,below left=of c1] (o3){};
	\node[open,below right=of c1] (o1) {};

	\node[closed, above right=of o1] (c2) {};
	\node[closed, below right=of o1] (c4) {};
	\node[closed, below left=of o1] (c3) {};
	\node[mixed,right=of o1] (m5){};
	\node[mixed,right=of c3] (m6){};
	\node[mixed,above=of c2] (m7){};
	\node[mixed,left=of c3] (m8){};

	\draw[lightgray, thin, draw] (o0)--(m3)--(o3)--(m8)--(c3)--(m2)--(c1)--(m1)--(o2)--(m4)--(o1)--(m6)--(c4)--(m5)--(c2)--(m7)--(o2)--(m1)--(o0)--(m3)--(c1)--(m4)--(c2)--(m5)--(o1)--(m2)--(o3)--(m8)--(c3)--(m6)--(c4)--(o1)--(c1)--(o0)--(m1)--(o2)--(c1)--(o3)--(c3)--(o1)--(c2)--(o2);
	\draw[blue,ultra thick] (o0)--(m3)--(o3)--(c3)--(o1)--(c2)--(o2)--(m1)--(o0);
\end{tikzpicture}};
    \node[below left=of h] (i) {
    \begin{tikzpicture}[open/.style={shape=circle,fill=white,scale=.75,draw}, mixed/.style={shape=circle,fill=gray,scale=.5,draw}, closed/.style={shape=rectangle,fill=white,draw}]

	\node[closed] (c1) at (2,3){};

	\node[mixed,above=of c1] (m1){};
	\node[mixed,below=of c1] (m2){};
	\node[mixed,left=of c1] (m3) {};
	\node[mixed,right=of c1] (m4) {};

	\node[open,above left=of c1] (o0) {};
	\node[open,above right=of c1] (o2){};
	\node[open,below left=of c1] (o3){};
	\node[open,below right=of c1] (o1) {};

	\node[closed, above right=of o1] (c2) {};
	\node[closed, below right=of o1] (c4) {};
	\node[closed, below left=of o1] (c3) {};
	\node[mixed,right=of o1] (m5){};
	\node[mixed,right=of c3] (m6){};
	\node[mixed,above=of c2] (m7){};
	\node[mixed,left=of c3] (m8){};

	\draw[lightgray, thin, draw] (o0)--(m3)--(o3)--(m8)--(c3)--(m2)--(c1)--(m1)--(o2)--(m4)--(o1)--(m6)--(c4)--(m5)--(c2)--(m7)--(o2)--(m1)--(o0)--(m3)--(c1)--(m4)--(c2)--(m5)--(o1)--(m2)--(o3)--(m8)--(c3)--(m6)--(c4)--(o1)--(c1)--(o0)--(m1)--(o2)--(c1)--(o3)--(c3)--(o1)--(c2)--(o2);
	\draw[blue, ultra thick] (o0)--(m3)--(o3)--(c3)--(o1)--(m4)--(o2)--(m1)--(o0);
\end{tikzpicture}};
    \node[below right=of h] (j) {
    \begin{tikzpicture}[open/.style={shape=circle,fill=white,scale=.75,draw}, mixed/.style={shape=circle,fill=gray,scale=.5,draw}, closed/.style={shape=rectangle,fill=white,draw}]

	\node[closed] (c1) at (2,3){};

	\node[mixed,above=of c1] (m1){};
	\node[mixed,below=of c1] (m2){};
	\node[mixed,left=of c1] (m3) {};
	\node[mixed,right=of c1] (m4) {};

	\node[open,above left=of c1] (o0) {};
	\node[open,above right=of c1] (o2){};
	\node[open,below left=of c1] (o3){};
	\node[open,below right=of c1] (o1) {};

	\node[closed, above right=of o1] (c2) {};
	\node[closed, below right=of o1] (c4) {};
	\node[closed, below left=of o1] (c3) {};
	\node[mixed,right=of o1] (m5){};
	\node[mixed,right=of c3] (m6){};
	\node[mixed,above=of c2] (m7){};
	\node[mixed,left=of c3] (m8){};

	\draw[lightgray, thin, draw] (o0)--(m3)--(o3)--(m8)--(c3)--(m2)--(c1)--(m1)--(o2)--(m4)--(o1)--(m6)--(c4)--(m5)--(c2)--(m7)--(o2)--(m1)--(o0)--(m3)--(c1)--(m4)--(c2)--(m5)--(o1)--(m2)--(o3)--(m8)--(c3)--(m6)--(c4)--(o1)--(c1)--(o0)--(m1)--(o2)--(c1)--(o3)--(c3)--(o1)--(c2)--(o2);
	\draw[blue, ultra thick] (o0)--(m3)--(o3)--(m2)--(o1)--(c2)--(o2)--(m1)--(o0);
\end{tikzpicture}};
    \node[below right=of i] (k) {
    \begin{tikzpicture}[open/.style={shape=circle,fill=white,scale=.75,draw}, mixed/.style={shape=circle,fill=gray,scale=.5,draw}, closed/.style={shape=rectangle,fill=white,draw}]

	\node[closed] (c1) at (2,3){};

	\node[mixed,above=of c1] (m1){};
	\node[mixed,below=of c1] (m2){};
	\node[mixed,left=of c1] (m3) {};
	\node[mixed,right=of c1] (m4) {};

	\node[open,above left=of c1] (o0) {};
	\node[open,above right=of c1] (o2){};
	\node[open,below left=of c1] (o3){};
	\node[open,below right=of c1] (o1) {};

	\node[closed, above right=of o1] (c2) {};
	\node[closed, below right=of o1] (c4) {};
	\node[closed, below left=of o1] (c3) {};
	\node[mixed,right=of o1] (m5){};
	\node[mixed,right=of c3] (m6){};
	\node[mixed,above=of c2] (m7){};
	\node[mixed,left=of c3] (m8){};

	\draw[lightgray, thin, draw] (o0)--(m3)--(o3)--(m8)--(c3)--(m2)--(c1)--(m1)--(o2)--(m4)--(o1)--(m6)--(c4)--(m5)--(c2)--(m7)--(o2)--(m1)--(o0)--(m3)--(c1)--(m4)--(c2)--(m5)--(o1)--(m2)--(o3)--(m8)--(c3)--(m6)--(c4)--(o1)--(c1)--(o0)--(m1)--(o2)--(c1)--(o3)--(c3)--(o1)--(c2)--(o2);
	\draw[blue, ultra thick] (o0)--(m3)--(o3)--(m2)--(o1)--(m4)--(o2)--(m1)--(o0);
\end{tikzpicture}};
    
    \draw (a)--(b)--(d)--(f)--(h)--(i)--(k)--(j)--(h)--(f)--(d)--(c)--(a);
    \draw (b)--(e)--(i);
    \draw (c)--(g)--(j);
\end{tikzpicture}}
    \caption{The Hasse diagram of $\J(\D_{4\times 4})$}
    \label{fig:J4x4}
\end{figure}

Intuitively, we can see how the pointwise order topology dictates the structure of the Hasse diagram in Figure \ref{fig:J4x4}.
If $J < K$ in $\J\paren{\D_{4\times 4}}$, then $J$ either has fewer closed points or more open points than $K$.
In Proposition \ref{prop:closedmax}, we will formalize what the maximal and minimal elements of a space of Jordan curves looks like.

\begin{thm}\label{cor:J41} 
For the $4 \times 4$ digital plane $\D_{4\times4}$, $\tc{}{\J(\D_{4\times4})}=1$.
\end{thm}

\begin{proof}
Since the Hasse diagram of $\J(\D_{4\times4})$ has a unique maximal element, $\J(\D_{4\times4})$ is contractible, so $\tc{}{\J(\D_{4\times4})}=1$ by Theorem 1 of \cite{Farber2001}.

\end{proof}

When a finite path-connected space $X$ has a unique maximal element $x_0 \in X$, the motion planner on $X$ sends a pair of start and end points $(a,b) \in X \times X$ to the path from $a$ to $x_0$ to $b$.
While this motion planner is continuous and will work for any space with a unique maximal element, the paths it generates are not necessarily intuitive.
Figure \ref{fig:J4p1} displays a path between two Jordan curves in $\J\paren{\D_{4\times 4}}$ as constructed in Corollary \ref{cor:J41}.
If we label the Jordan curves in Figure \ref{fig:J4p1} from left to right as $J$, $K$, and $L$, notice that $\abs{J \cap K}=\abs{K \cap L}=2$, despite the fact that $\abs{J \cap L}=6$.\footnote{These are intersections are as subsets of $\D_{4\times 4}$, not as singletons in $\J\paren{\D_{4 \times 4}}$.}
A more intuitive path from $J$ to $L$ might be the one shown in Figure \ref{fig:J4p2}.
Notice in that path that $\abs{J\cap K}=\abs{K \cap L} =7$.
In \cite{Blaszczyk2018}, they describe \textbf{efficient topological complexity}, for which the length of the motion planner is taken into account.
Although efficient topological complexity is defined for only smooth compact orientable Riemannian manifolds, it may be of interest to define an efficient notion of combinatorial complexity that minimizes the height traveled by the motion planner in the Hasse diagram.
For example, if $X$ is a $T_0$ space with associated Hasse diagram $\mathcal{H}$ and height function $h:\mathcal{H}\to \zz_{\geq 0}$, let $U \subseteq X\times X$ admit a motion planner $s$.
An ideal motion planner would minimize $\abs{h\paren{s(a,b)(t_1)}-h\paren{s(a,b)(t_2)}}$ for all $(a,b)\in U$ and $t_1,t_2 \in S^1$.

\begin{figure}
    \centering
    \begin{tikzpicture}
        \node (left) at (0,0) {
    \begin{tikzpicture}[open/.style={shape=circle,fill=white,scale=.75,draw}, mixed/.style={shape=circle,fill=gray,scale=.5,draw}, closed/.style={shape=rectangle,fill=white,draw}]

	\node[closed] (c1) at (2,3){};

	\node[mixed,above=of c1] (m1){};
	\node[mixed,below=of c1] (m2){};
	\node[mixed,left=of c1] (m3) {};
	\node[mixed,right=of c1] (m4) {};

	\node[open,above left=of c1] (o0) {};
	\node[open,above right=of c1] (o2){};
	\node[open,below left=of c1] (o3){};
	\node[open,below right=of c1] (o1) {};

	\node[closed, above right=of o1] (c2) {};
	\node[closed, below right=of o1] (c4) {};
	\node[closed, below left=of o1] (c3) {};
	\node[mixed,right=of o1] (m5){};
	\node[mixed,right=of c3] (m6){};
	\node[mixed,above=of c2] (m7){};
	\node[mixed,left=of c3] (m8){};

	\draw[lightgray, thin, draw] (o0)--(m3)--(o3)--(m8)--(c3)--(m2)--(c1)--(m1)--(o2)--(m4)--(o1)--(m6)--(c4)--(m5)--(c2)--(m7)--(o2)--(m1)--(o0)--(m3)--(c1)--(m4)--(c2)--(m5)--(o1)--(m2)--(o3)--(m8)--(c3)--(m6)--(c4)--(o1)--(c1)--(o0)--(m1)--(o2)--(c1)--(o3)--(c3)--(o1)--(c2)--(o2);
	\draw[blue, ultra thick] (o0)--(m3)--(o3)--(c3)--(o1)--(m4)--(o2)--(m1)--(o0);
\end{tikzpicture}};
        \node (mid) at (5,0) {
        \begin{tikzpicture}[open/.style={shape=circle,fill=white,scale=.75,draw}, mixed/.style={shape=circle,fill=gray,scale=.5,draw}, closed/.style={shape=rectangle,fill=white,draw}]

	\node[closed] (c1) at (2,3){};

	\node[mixed,above=of c1] (m1){};
	\node[mixed,below=of c1] (m2){};
	\node[mixed,left=of c1] (m3) {};
	\node[mixed,right=of c1] (m4) {};

	\node[open,above left=of c1] (o0) {};
	\node[open,above right=of c1] (o2){};
	\node[open,below left=of c1] (o3){};
	\node[open,below right=of c1] (o1) {};

	\node[closed, above right=of o1] (c2) {};
	\node[closed, below right=of o1] (c4) {};
	\node[closed, below left=of o1] (c3) {};
	\node[mixed,right=of o1] (m5){};
	\node[mixed,right=of c3] (m6){};
	\node[mixed,above=of c2] (m7){};
	\node[mixed,left=of c3] (m8){};

	\draw[lightgray, thin, draw] (o0)--(m3)--(o3)--(m8)--(c3)--(m2)--(c1)--(m1)--(o2)--(m4)--(o1)--(m6)--(c4)--(m5)--(c2)--(m7)--(o2)--(m1)--(o0)--(m3)--(c1)--(m4)--(c2)--(m5)--(o1)--(m2)--(o3)--(m8)--(c3)--(m6)--(c4)--(o1)--(c1)--(o0)--(m1)--(o2)--(c1)--(o3)--(c3)--(o1)--(c2)--(o2);
	\draw[blue,ultra thick] (c1)--(m2)--(c3)--(m6)--(c4)--(m5)--(c2)--(m4)--(c1);
\end{tikzpicture}};
    \node (right) at (10,0) {
    \begin{tikzpicture}[open/.style={shape=circle,fill=white,scale=.75,draw}, mixed/.style={shape=circle,fill=gray,scale=.5,draw}, closed/.style={shape=rectangle,fill=white,draw}]

	\node[closed] (c1) at (2,3){};

	\node[mixed,above=of c1] (m1){};
	\node[mixed,below=of c1] (m2){};
	\node[mixed,left=of c1] (m3) {};
	\node[mixed,right=of c1] (m4) {};

	\node[open,above left=of c1] (o0) {};
	\node[open,above right=of c1] (o2){};
	\node[open,below left=of c1] (o3){};
	\node[open,below right=of c1] (o1) {};

	\node[closed, above right=of o1] (c2) {};
	\node[closed, below right=of o1] (c4) {};
	\node[closed, below left=of o1] (c3) {};
	\node[mixed,right=of o1] (m5){};
	\node[mixed,right=of c3] (m6){};
	\node[mixed,above=of c2] (m7){};
	\node[mixed,left=of c3] (m8){};

	\draw[lightgray, thin, draw] (o0)--(m3)--(o3)--(m8)--(c3)--(m2)--(c1)--(m1)--(o2)--(m4)--(o1)--(m6)--(c4)--(m5)--(c2)--(m7)--(o2)--(m1)--(o0)--(m3)--(c1)--(m4)--(c2)--(m5)--(o1)--(m2)--(o3)--(m8)--(c3)--(m6)--(c4)--(o1)--(c1)--(o0)--(m1)--(o2)--(c1)--(o3)--(c3)--(o1)--(c2)--(o2);
	\draw[blue, ultra thick] (o0)--(m3)--(o3)--(m2)--(o1)--(c2)--(o2)--(m1)--(o0);
\end{tikzpicture}}; 
    \draw[->, ultra thick] (left) -- (mid);
    \draw[->, ultra thick] (mid) -- (right);
    \end{tikzpicture}
    \caption{A continuous path between two Jordan curves given by the motion planner on a space with a maximal element}
    \label{fig:J4p1}
\end{figure}
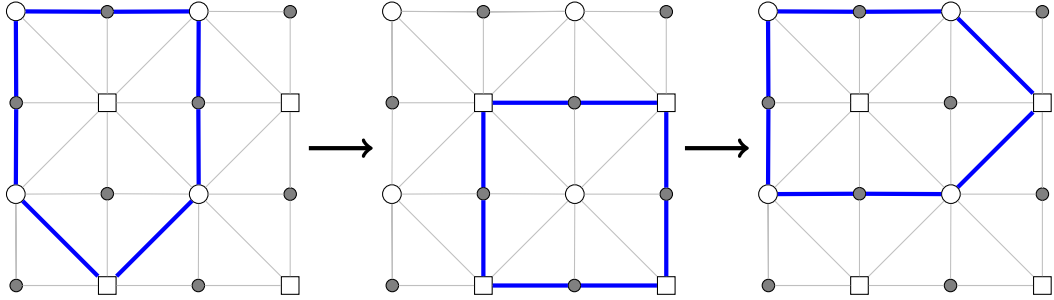

\begin{figure}
    \centering
    \begin{tikzpicture}
        \node (left) at (0,0) {
    \begin{tikzpicture}[open/.style={shape=circle,fill=white,scale=.75,draw}, mixed/.style={shape=circle,fill=gray,scale=.5,draw}, closed/.style={shape=rectangle,fill=white,draw}]

	\node[closed] (c1) at (2,3){};

	\node[mixed,above=of c1] (m1){};
	\node[mixed,below=of c1] (m2){};
	\node[mixed,left=of c1] (m3) {};
	\node[mixed,right=of c1] (m4) {};

	\node[open,above left=of c1] (o0) {};
	\node[open,above right=of c1] (o2){};
	\node[open,below left=of c1] (o3){};
	\node[open,below right=of c1] (o1) {};

	\node[closed, above right=of o1] (c2) {};
	\node[closed, below right=of o1] (c4) {};
	\node[closed, below left=of o1] (c3) {};
	\node[mixed,right=of o1] (m5){};
	\node[mixed,right=of c3] (m6){};
	\node[mixed,above=of c2] (m7){};
	\node[mixed,left=of c3] (m8){};

	\draw[lightgray, thin, draw] (o0)--(m3)--(o3)--(m8)--(c3)--(m2)--(c1)--(m1)--(o2)--(m4)--(o1)--(m6)--(c4)--(m5)--(c2)--(m7)--(o2)--(m1)--(o0)--(m3)--(c1)--(m4)--(c2)--(m5)--(o1)--(m2)--(o3)--(m8)--(c3)--(m6)--(c4)--(o1)--(c1)--(o0)--(m1)--(o2)--(c1)--(o3)--(c3)--(o1)--(c2)--(o2);
	\draw[blue, ultra thick] (o0)--(m3)--(o3)--(c3)--(o1)--(m4)--(o2)--(m1)--(o0);
\end{tikzpicture}};
        \node (mid) at (5,0) {
        \begin{tikzpicture}[open/.style={shape=circle,fill=white,scale=.75,draw}, mixed/.style={shape=circle,fill=gray,scale=.5,draw}, closed/.style={shape=rectangle,fill=white,draw}]

	\node[closed] (c1) at (2,3){};

	\node[mixed,above=of c1] (m1){};
	\node[mixed,below=of c1] (m2){};
	\node[mixed,left=of c1] (m3) {};
	\node[mixed,right=of c1] (m4) {};

	\node[open,above left=of c1] (o0) {};
	\node[open,above right=of c1] (o2){};
	\node[open,below left=of c1] (o3){};
	\node[open,below right=of c1] (o1) {};

	\node[closed, above right=of o1] (c2) {};
	\node[closed, below right=of o1] (c4) {};
	\node[closed, below left=of o1] (c3) {};
	\node[mixed,right=of o1] (m5){};
	\node[mixed,right=of c3] (m6){};
	\node[mixed,above=of c2] (m7){};
	\node[mixed,left=of c3] (m8){};

	\draw[lightgray, thin, draw] (o0)--(m3)--(o3)--(m8)--(c3)--(m2)--(c1)--(m1)--(o2)--(m4)--(o1)--(m6)--(c4)--(m5)--(c2)--(m7)--(o2)--(m1)--(o0)--(m3)--(c1)--(m4)--(c2)--(m5)--(o1)--(m2)--(o3)--(m8)--(c3)--(m6)--(c4)--(o1)--(c1)--(o0)--(m1)--(o2)--(c1)--(o3)--(c3)--(o1)--(c2)--(o2);
	\draw[blue, ultra thick] (o0)--(m3)--(o3)--(m2)--(o1)--(m4)--(o2)--(m1)--(o0);
\end{tikzpicture}};
    \node (right) at (10,0) {
    \begin{tikzpicture}[open/.style={shape=circle,fill=white,scale=.75,draw}, mixed/.style={shape=circle,fill=gray,scale=.5,draw}, closed/.style={shape=rectangle,fill=white,draw}]

	\node[closed] (c1) at (2,3){};

	\node[mixed,above=of c1] (m1){};
	\node[mixed,below=of c1] (m2){};
	\node[mixed,left=of c1] (m3) {};
	\node[mixed,right=of c1] (m4) {};

	\node[open,above left=of c1] (o0) {};
	\node[open,above right=of c1] (o2){};
	\node[open,below left=of c1] (o3){};
	\node[open,below right=of c1] (o1) {};

	\node[closed, above right=of o1] (c2) {};
	\node[closed, below right=of o1] (c4) {};
	\node[closed, below left=of o1] (c3) {};
	\node[mixed,right=of o1] (m5){};
	\node[mixed,right=of c3] (m6){};
	\node[mixed,above=of c2] (m7){};
	\node[mixed,left=of c3] (m8){};

	\draw[lightgray, thin, draw] (o0)--(m3)--(o3)--(m8)--(c3)--(m2)--(c1)--(m1)--(o2)--(m4)--(o1)--(m6)--(c4)--(m5)--(c2)--(m7)--(o2)--(m1)--(o0)--(m3)--(c1)--(m4)--(c2)--(m5)--(o1)--(m2)--(o3)--(m8)--(c3)--(m6)--(c4)--(o1)--(c1)--(o0)--(m1)--(o2)--(c1)--(o3)--(c3)--(o1)--(c2)--(o2);
	\draw[blue, ultra thick] (o0)--(m3)--(o3)--(m2)--(o1)--(c2)--(o2)--(m1)--(o0);
\end{tikzpicture}}; 
    \draw[->, ultra thick] (left) -- (mid);
    \draw[->, ultra thick] (mid) -- (right);
    \end{tikzpicture}
    \caption{A more intuitive path than the one given in Figure \ref{fig:J4p1}}
    \label{fig:J4p2}
\end{figure}

We have also counted all of the Jordan curves that can exist in a $5\times5$ digital plane with an open border.
In Figures \ref{fig:87-1} through \ref{fig:87-4}, we display all 87 Jordan curves in $\J\paren{\D_{5\times5}}$, hand-drawn.
The adjacency lines are not drawn in, however, they remain the same as in Figure \ref{fig:plane}.
The Jordan curves are roughly in order from the maximal elements to the minimal elements.
The maximal element of $\J\paren{\D_{5\times5}}$ (the top-left Jordan curve in Figure \ref{fig:87-1}) is the Jordan curve comprising closed and mixed points, which is the adjacency set of the central open point of $\D_{5\times 5}$.
We can formalize these ideas with the following.

\begin{prop}\label{prop:closedmax} 
If $J \in \J$ is a Jordan curve containing no open points of $\D$, then $J$ is a maximal element of $\J$.
If $J$ contains no closed points, then it is a minimal element of $\J$.
\end{prop}

It is worth nothing that in Proposition \ref{prop:closedmax}, ``minimal'' refers to the $J$ having height 0 in the Hasse diagram of $\J$, and not as a member of $\J_1(\D)$.

\begin{proof}
Let $\D$ be a Khalimsky digital plane and $\J:=\J\paren{\D}$ its space of digital Jordan curves.
Let $J \in \J$ with parameterization $f$ such that $J$ has no open points.
Suppose there exists a Jordan curve $K \in \J$ with parameterization $g$ such that $K \geq J$ in $\J$.
Then $g(t)\geq f(t)$ for all $t \in S^1$.
Since $J$ has no open points, it comprises only mixed points and closed points.
Because the closed points $c_i \in J$ are maximal in $\D$, $g \geq f$ implies $g\paren{f^{-1}\paren{c_i}}\geq c_i$, so $g(t)=f(t)$ for $t \in f^{-1}\paren{c_i}$.
Then we must have $g(t)\geq f(t)$ for $t \in f^{-1}\paren{m}$ for some mixed point $m \in J$, that is, $g\paren{f^{-1}\paren{m}} \in m^\uparrow$.
Because Jordan curves cannot turn at mixed points, $m^\uparrow =:\set{c^-,m,c^+} \subset J$ as well.
Then $g \geq f$ implies $g\paren{f^{-1}\paren{m}} \in \set{c^-,m,c^+}$.
If $m \not\in \textrm{im}(g)$, then $\set{c^-,c^+}$ does not determine two unique COTS-arcs belonging to $K$, contradicting Lemma 5.2(c) of \cite{Khalimsky1990a}.

A similar argument shows that the Jordan curves containing no closed points are minimal elements of $\J$.
\end{proof}

We conjecture that the converse of Proposition \ref{prop:closedmax} is true, however, this has yet to be shown.
The shape of the maximal and minimal elements described above in fact show a correspondence to polyominoes, or cycles in a square graph.
In \cite{Golomb1954}, they define an {\bf $n$-omino} (or polyomino) to be a simply connected set of $n$ squares of a chessboard that are ``rook-wise connected.''
Viewing $\D$ as a subset of $\zz\times\zz$, the closed points of $\D$ are the lattice points of $2\zz \times 2\zz$.
\v{S}lapal calls this the ``square graph of type 2,'' and proves in \cite{Slapal2006} that any cycle in this graph is a Jordan curve.
The number of cycles in an $n\times n$ grid, which we will denote $c(n)$, is shown in Table \ref{tab:polyn} (see \cite{OEISA140517} for the table up through $n=26$).
By Proposition \ref{prop:closedmax}, if a digital plane $\D$ contains an $n \times n$ lattice of closed points, then $\J(\D)$ has at least $c(n-1)$ maximal elements.
Similarly, if $\D$ contains an $n \times n$ lattice of open points, then $\J(\D)$ has at least $c(n-1)$ minimal elements.
We conjecture that these inequalities are, in fact, equalities.
If that is the case, then the following will hold:
\[\tc{}{\J(\D)} \leq \ls{\J(\D)}^2 \leq c(n)^2.\]

\begin{table}[]
    \centering
    \begin{tabular}{r|l}
       $n$  & $c(n)$ \\ \hline
        0 &0\\
        1 & 1\\
        2 & 13\\
        3 & 213\\
        4 & 9349\\
        5 & 1222363\\
        6 & 487150371\\
        7 & 603841648931\\
        8 & 2318527339461265\\
        9 & 27359264067916806101\\
        10 & 988808811046283595068099
    \end{tabular}
    \caption{The number of cycles in an $n\times n$ grid for $n\in \set{0,1,\hdots,10}$}
    \label{tab:polyn}
\end{table}

There is a $2\times 2$ lattice of closed points in $\D_{5 \times 5}$ (see Figure \ref{fig:plane}), and so the maximal elements of $\J\paren{\D_{5\times5}}$ correspond to the simply connected polyominoes in a $1\times1$ grid, of which there is only one.
There is a $3\times 3$ lattice of open points in $\D_{5\times 5}$, which corresponds to a $2\times2$ grid with nine vertices.
By \cite{OEISA118797}, a polyomino that is not simply-connected must contain at least seven tiles.
Consequently, every polyomino in a $2\times 2$ grid corresponds to a Jordan curve in $\D'_{5\times 5}$, of which there are thirteen.
If $\D_{n \times n}$ is a digital plane with open cornerpoints and $n$ odd, then by Table \ref{tab:polyn}, $\J\paren{\D_{n\times n}}$ has at least $c\paren{\frac{n-1}{2}}$ maximal elements and at least $c\paren{\frac{n+1}{2}}$ minimal elements.
As noted in \cite{Tanaka2018}, all known motion planners on finite spaces are defined on categorical sets.

The four elements following the maximal element in Figure \ref{fig:87-1} complete the top two rows of the Hasse diagram of $\J\paren{\D_{5\times5}}$, which is shown in Figure \ref{fig:5x5top}.
The thirteen minimal elements are those at the ends of Figures \ref{fig:87-3} and \ref{fig:87-4}.

\begin{figure}
    \centering
    \includegraphics[page=1, clip, ,trim=.7in .25in 5.25in 0in , width=4in,angle=0]{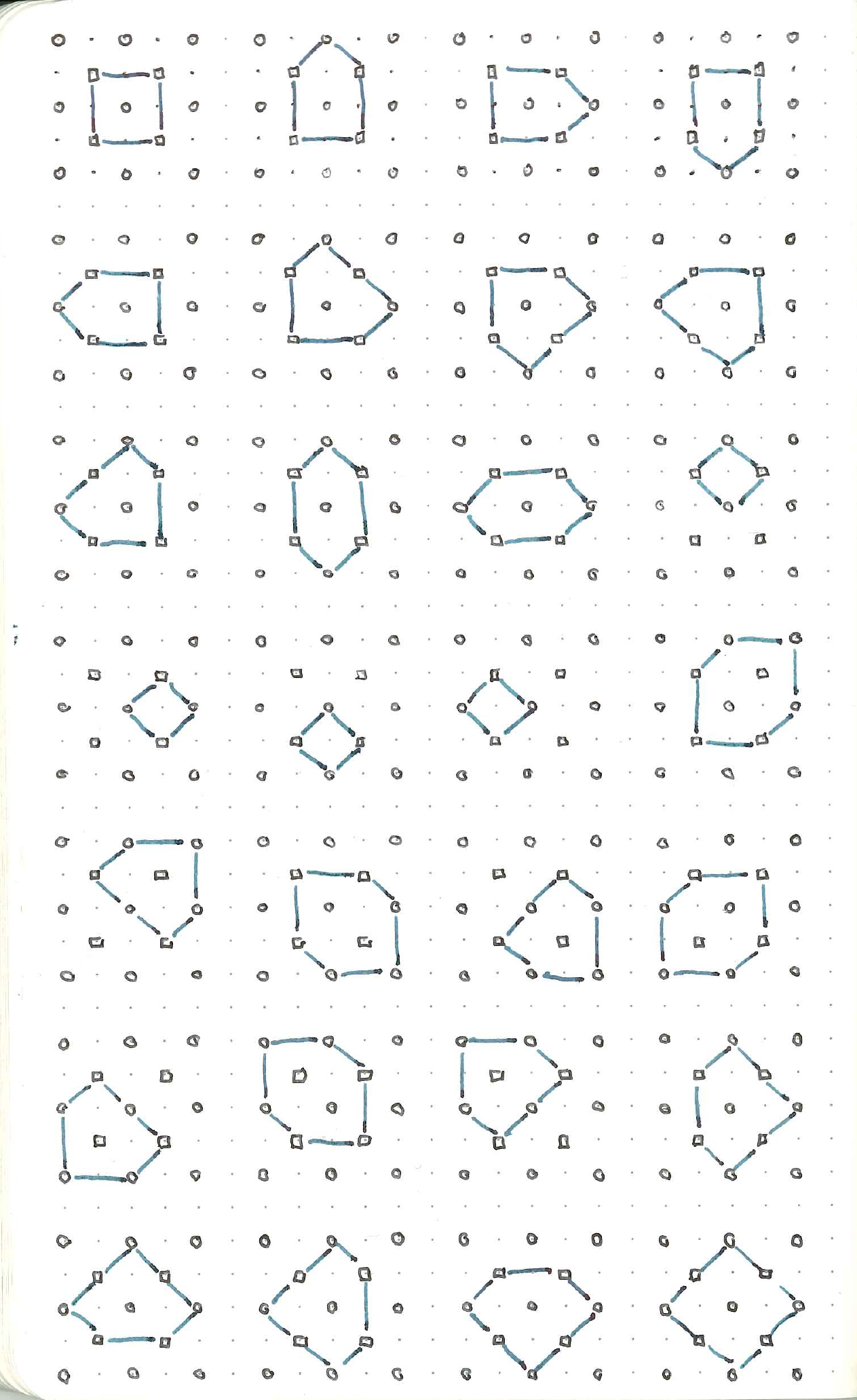}
    \caption{Jordan curves in $\J\paren{\D_{5\times5}}$ (Figure 1 of 4)}
    \label{fig:87-1}
\end{figure}

\begin{figure}
    \centering
    \includegraphics[page=1, clip, trim = 6in .3in .3in 0in,angle=0, width=4in]{5x5-87.pdf}
    \caption{Jordan curves in $\J\paren{\D_{5\times5}}$ (Figure 2 of 4)}
    \label{fig:87-2}
\end{figure}

\begin{figure}
    \centering
    \includegraphics[page=2, clip, trim=.6in .25in 5.5in .1in , width=4in,angle=-1]{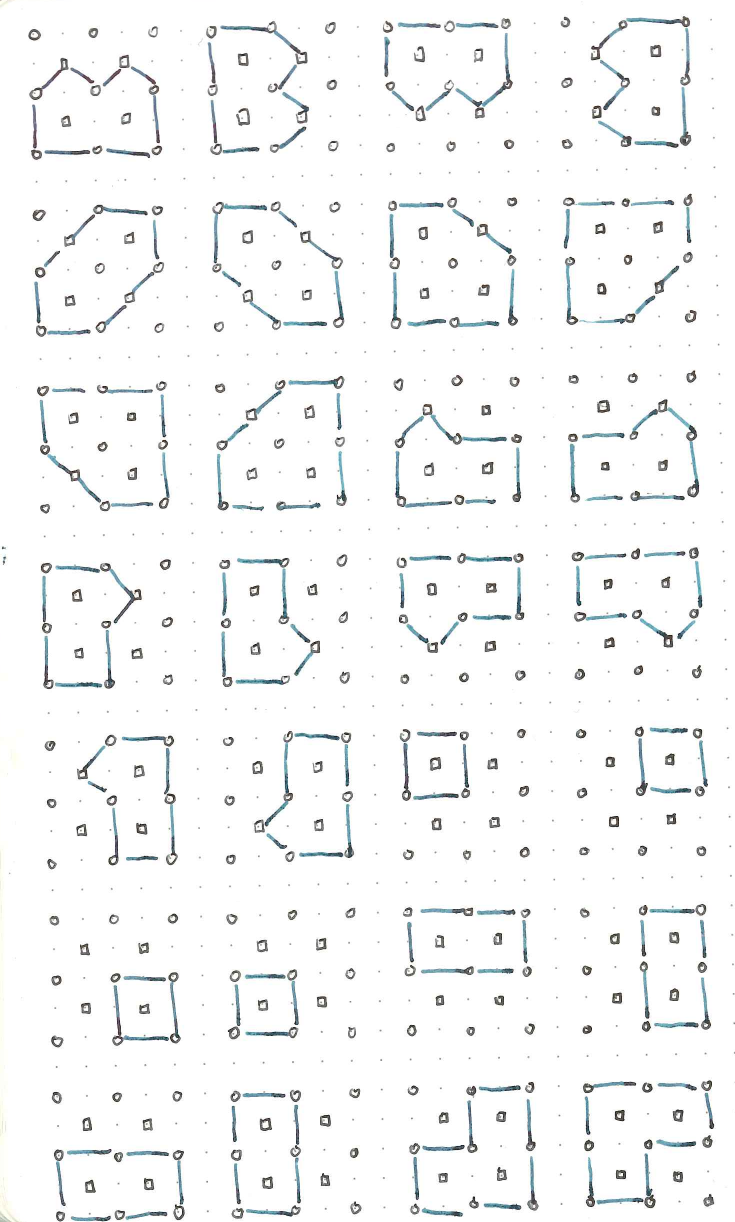}
    \caption{Jordan curves in $\J\paren{\D_{5\times5}}$ (Figure 3 of 4)}
    \label{fig:87-3}
\end{figure}

\begin{figure}
    \centering
    \includegraphics[page=2, clip, trim = 5.8in .3in 1.5in 7in,angle=-1, width=4in]{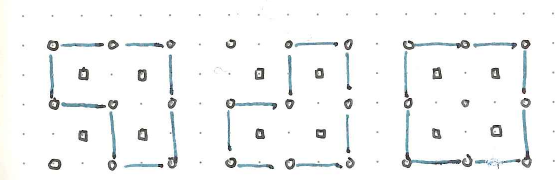}
    \caption{Jordan curves in $\J\paren{\D_{5\times5}}$ (Figure 4 of 4)}
    \label{fig:87-4}
\end{figure}

\begin{figure}
    \centering
    \resizebox{5in}{!}{
\begin{tikzpicture}
\node (l0) at (0,0){\begin{tikzpicture}
\five
\draw[blue, ultra thick] (2,2)--(4,2)--(4,4)--(2,4)--(2,2);
\end{tikzpicture}};

\node (l1a) at (-9,-10){\begin{tikzpicture}
\five
\draw[blue, ultra thick] (2,2)--(4,2)--(4,4)--(3,5)--(2,4)--(2,2);
\end{tikzpicture}};

\node (l1b) at (-3,-10){\begin{tikzpicture}
\five
\draw[blue, ultra thick] (2,2)--(4,2)--(4,4)--(2,4)--(1,3)--(2,2);
\end{tikzpicture}};

\node (l1c) at (3,-10){\begin{tikzpicture}
\five
\draw[blue, ultra thick] (2,2)--(4,2)--(5,3)--(4,4)--(2,4)--(2,2);
\end{tikzpicture}};

\node (l1d) at (9,-10){\begin{tikzpicture}
\five
\draw[blue, ultra thick] (2,2)--(3,1)--(4,2)--(4,4)--(2,4)--(2,2);
\end{tikzpicture}};

\draw[black, draw] (l0)--(l1a);\draw (l0)--(l1b);\draw (l0)--(l1c);\draw (l0)--(l1d);
\end{tikzpicture}}
    \caption{The top two rows of the Hasse diagram of $\J\paren{\D_{5\times 5}}$}
    \label{fig:5x5top}
\end{figure}

Even if we cannot yet visualize the Hasse diagram of a given space of digital Jordan curves, we can sometimes enumerate its elements to help understand the structure.

\begin{thm}\label{thm:J3n}
If $\D_{3\times n}$ is a $3 \times n$ digital plane, $\left|\J(\D_{3\times n})\right|=\frac{(n-1)(n-2)}{2}$.
\end{thm}
\begin{proof}
For any Jordan curve $J \in \J(\D_{3\times n})$, $\jint{J}$ is a subset of the $1 \times (n-2)$ COTS nested inside of $\D_{3\times n}$; call it $C_{n-2} := \cots{c_1,c_2,\hdots,c_{n-2}}$.
The connected subsets of $C_{n-2}$ are in bijection with subsets of consecutive integers of $\{1,2,\hdots,n-2\}$, of which there are $\frac{(n-1)(n-2)}{2}$ (see \cite{OEIS000217}).
We will show that any connected subset of $C_{n-2}$ determines a unique Jordan curve in $\J(\D_{3\times n})$.

Let $C_{i,j}=\cots{c_i,c_j} \subseteq C_{n-2}$ for some $1\leq i < j \leq n-2$.
We will show that 
\[J:=\left(\bigcup_{c \in C_{i,j}} A(c)\right)-C_{i,j}\]
is a Jordan curve in $\J(\D_{3\times n})$.
Notice that $J$ is a subset of a $3 \times (j-i+3)$ digital plane inside $\D_{3\times n}$.
By Lemma 5.2(b) of \cite{Khalimsky1990a}, the adjusted border of $\D_{3 \times (j-i+3)}$ is a Jordan curve whose interior is $C_{i,j}$.
\end{proof}

Theorem \ref{thm:J3n} paves the way for establishing a lowerbound for the number of Jordan curves in a $3m \times n$ plane.
A $6\times 6$ digital plane, for example, has at least $2\cdot 5 \cdot 4 = 40$ Jordan curves.
Notice that this gives a better lowerbound than $c(n)$, which only tells us that $\D_{6\times6}$ has $c(2)=13$ minimal elements and $c(2)=13$ maximal elements.
This is a gross underestimate, however, as we have shown above that $\abs{\J\paren{\D_{5\times5}}}=87$.

\section{Concluding Remarks}
\label{chap:app}

For the first time in literature, this article goes beyond the practice of studying properties of digital images and instead explores properties of collections of digital images as a whole.
By approaching digital topology from this angle, we hope to establish a correspondence between image processing algorithms and paths in spaces of digital images.
Paths in a space of digital images represent a sequence of images to pass through in navigating from one image to another image.
In this section we justify our approach to solving these problems, and look into some applications of these results beyond topological complexity.

\subsection{Behavior Under Different Topologies}
\label{sec:diftop}

Here, we explore how the results of Sections \ref{chap:COTS} and \ref{chap:DigJC} behave under different digital topologies.
In Section \ref{sec:ottop}, we presented three topologies on $\zz^2$ that are not the Khalimsky topology.
Those are the Marcus-Wyse topology of \cite{Marcus1970}, and the topologies $(\zz^2,w)$ and $(\zz^2,\hat{w})$ from \cite{Slapal2006}.
In \cite{Slapal2006}, they show that all three of those topologies are $T_\frac{1}{2}$.
Recall that a finite space $X$ is \textbf{$T_\frac{1}{2}$} if and only if for all $x \in X$, either $x^\uparrow=\{x\}$ or $x^\downarrow=\{x\}$.
That is, every point of $x$ is either open or closed.

\begin{proof}[Proof of Theorem \ref{thm:onlychoice}]
Let $\D \subset \zz^2$ be a sufficiently large finite rectangular lattice (i.e., there are at least two Jordan curves $J,J'\subset \D$).
Let $\tau$ be one of the topologies mentioned in Section \ref{sec:ottop} that is not the Khalimsky topology.
That is, $\tau$ is either the Marcus-Wyse topology from \cite{Marcus1970}, or $w$ or $\hat{w}$ from \cite{Slapal2006}.
See Figures \ref{fig:marcus}, \ref{fig:Zw}, and \ref{fig:Zw'}, respectively, for tiles of these planes.
As proven in \cite{Slapal2006}, $(\D,\tau)$ is $T_\frac{1}{2}$.
For the remainder of this proof, we will take $\D:=(\D,\tau)$, and $\J:=\J((\D,\tau))$.

If every point of $\D$ is either open or closed, then the Hasse diagram of $\D$ is of height one.
Furthermore, every point of $\D$ has at least two $\tau$-adjacent neighbors, so $x^\downarrow-\set{x}$ and $x^\uparrow-\set{x}$ are either empty or discrete for all $x \in \D$.
Hence, there are no beat points in $\D$, so it is a minimal finite space, and in particular, it is not contractible by Corollary 4 of \cite{Stong1966}.
Then by Remark 3.3.1 of \cite{Barmak2012}, $\greal{\D}$ is weakly homotopy equivalent to a wedge of $n=\chi(\D)-1$ circles, where, by abuse of notation,\footnote{Not to be confused with $\chi(K)$, which is the poset of simplices of a simplicial complex $K$, ordered by inclusion. The Euler characteristic of a finite $T_0$ space $X$ is given by the number of points in $X$ minus the number of edges in the Hasse diagram of $X$.} $\chi(\D)$ is the Euler characteristic of $\D$.
We will denote this space $\bigvee_n S^1 \stackrel{we}{\simeq} \D$, where $x_0 \in \bigvee_n S^1$ is the basepoint of the wedge.
Because $\bigvee_n S^1 \stackrel{we}{\simeq} \D$, there exists an isomorphism $\pi_1\paren{\greal{\D},x_0} \cong \pi_1\paren{\D,x_0}$.

Let $J,J' \in \J$ be two distinct Jordan curves with parameterizations $f,f':S^1 \to \D$, respectively.
For the remainder of this proof, we will refer to $J$ and $J'$ by their parameterizations $f$ and $f'$.
Consider a spanning tree $T \subset \D$.
Since $T \subset \D$ is contractible, $\img{f}\not\subset T$ and $\img{f'}\not\subset T$.
Then there exist 1-chains $\set{j_1,j_2} \subset \img{f}-T$ and $\set{j_1,j_2} \subset \img{f'}-T$ such that $\set{j_1,j_2} \neq \set{j_1',j_2'}$.
Since $\set{j_1,j_2} \neq \set{j_1',j_2'}$, $\greal{\set{j_1,j_2}} \neq \greal{\set{j_1',j_2'}}$ in $\bigvee_n S^1$.
Then $\greal{f}$ and $\greal{f'}$ are in different homotopy equivalence classes of $\pi_1(\greal{D},x_0)$, so $\greal{f}\not\simeq \greal{f'}$.
Since $\pi_1\paren{\greal{\D},x_0} \cong \pi_1\paren{\D,x_0}$, $f\not\simeq f'$ as well.
Then there exists no path between $J$ and $J'$ in $\J$.

\end{proof}

In particular, the proof of Theorem \ref{thm:onlychoice} will work for any digital topology that treats the digital plane as a graph and digital Jordan curves as cycles in that graph.
\footnote{It is worth noting that we are not considering any digital planes equipped with only a pretopology, or with the discrete topology.}
This agrees with our intuition that an appropriate digital plane should be weakly homotopy equivalent to a rectangle in $\rr^2$.
The geometric realization of a finite Khalimsky plane is shown in Figure \ref{fig:grealD}, and the geometric realization of a tile of $\paren{\zz^2,w}$ is shown in Figures \ref{fig:grealZw} and \ref{fig:SgrealZw}.
It becomes apparent that any appropriate digital plane must have all three of open, closed, and mixed points.
The trait that prevents Jordan curves from turning at acute angles is in fact necessary!
\begin{cor}
Let $\D$ be a finite topological space modeling a digital plane.
If $\J(\D)$ is connected, then the height of the Hasse diagram of $\D$ is at least two.
\end{cor}

\begin{figure}
    \centering
    \begin{tikzpicture}
    \draw[black, fill=gray] (1,1)--(5,1)--(5,5)--(1,5)--(1,1);
    \foreach \x in {1,2,3,4,5}
	{
		\draw[black] (\x,1)--(\x,5);
		\draw[black] (1,\x)--(5,\x);
	}
\foreach \x in {1,3,5}
	{
		\draw[black] (\x,1)--(5,6-\x);
		\draw[black] (\x,5)--(5,\x);
		\draw[black] (1,\x)--(6-\x,5);
		\draw[black] (\x,1)--(1,\x);
	}
    \end{tikzpicture}
    \caption{The geometric realization of Figure \ref{fig:plane}}
    \label{fig:grealD}
\end{figure}
\begin{figure}
    \centering
    \begin{tikzpicture}
    \draw[black] (0,0)--(0,4)--(4,4)--(4,0)--(0,0)--(1,1)--(3,1)--(3,3)--(1,3)--(1,1)--(4,4);
    \draw[black] (0,4)--(4,0);
    \draw[black] (0,1)--(1,2)--(0,3);
    \draw[black] (4,1)--(3,2)--(4,3);
    \draw[black] (1,4)--(2,3)--(3,4);
    \draw[black] (1,0)--(2,1)--(3,0);
    
    \draw[green, ultra thick] (1,2)--(0,1)--(0,0)--(4,4)--(4,3);
    \draw[green, ultra thick] (0,3)--(0,4)--(4,0)--(4,1)--(3,2);
    \draw[green, ultra thick] (1,4)--(2,3)--(3,4)--(4,4);
    \draw[green, ultra thick] (0,0)--(1,0)--(2,1)--(3,0);
    \end{tikzpicture}
    \caption{The geometric realization of Figure \ref{fig:Zw} with a spanning tree highlighted in green}
    \label{fig:grealZw}
\end{figure}

\begin{figure}
    \centering
    \begin{tikzpicture}
   \begin{polaraxis}[grid=none, axis lines=none]
     \addplot[mark=none,domain=0:360,samples=500] {cos(x*8)};
   \end{polaraxis}
   \end{tikzpicture}
    \caption{A wedge of 16 circles homotopy equivalent to Figure \ref{fig:grealZw} in which we've quotiented by the green spanning tree}
    \label{fig:SgrealZw}
\end{figure}
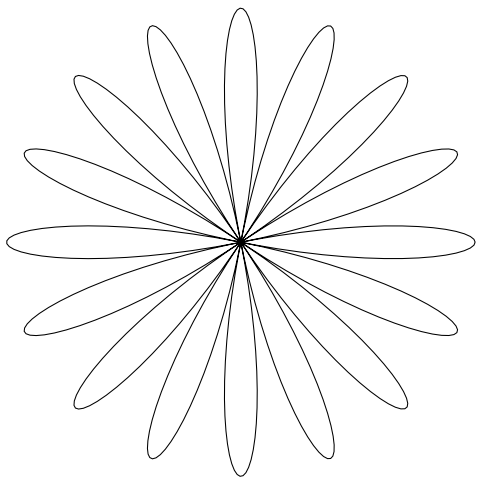

\subsection{Grayscale Images}

Fuzzy topology was first defined in \cite{Zadeh1965} as a way of measuring the ``degree of membership'' of a point in a set.
In \cite{Rosenfeld1979}, Rosenfeld expands on this idea by defining fuzzy topology for a digital plane.
This yields a means of describing grayscale images, as opposed to simple black-and-white ones.
In standard digital topology, an image is a subset $A$ of a digital plane $\D$ such that the points of $A$ are considered black, and the points of $\D-A$ are considered white.
Equivalently, there exists a function $\sigma:\D \to \{0,1\}$ such that $\sigma(a)=1$ for all $a \in A$, and $\sigma(a')=0$ for all $a' \in \D-A$.
To account for a range of shades of gray, Rosenfeld defines a fuzzy image to be a function $\sigma:\D \to [0,1]$, where points of $\D$ are mapped to some value of gray in $[0,1]$.
Introducing such a grayscale on $\J$, however, would trivialize the space.

To see this, consider the following.
Let $\mathcal{I}:=\set{\sigma:\D\to[0,1]}$ be the set of fuzzy images in a finite Khalimsky plane $\D$.
Because the images $\sigma$ do not necessarily need to be continuous, we may consider $\mathcal{I} = [0,1]^{|\D|}$, as opposed to $[0,1]^\D$, which would force all open points to be the same color as their closure.
Since $\D$ is countable, there exists an ordering $p_1,p_2,\hdots,p_{|\D|}$ of the points of $\D$.
Then there exists a bijection $\varphi$ from $\mathcal{I} \to [0,1]^{|\D|}$ given by $\varphi(\sigma) = \paren{\sigma(p_1),\sigma(p_2),\hdots,\sigma\paren{p_{|\D|}}} \in [0,1]^{|\D|}$.
In this way, $\mathcal{I}$ is contractible, and there exists a path from any image $\sigma:\D \to [0,1]$ to the all-white image  $\omega =(0,0,\hdots,0)\in [0,1]^{|\D|}$.
Explicitly, we have $H:\mathcal{I}\times[0,1] \to (0,0,\hdots,0)$ given by $H(\sigma,t)=(1-t)\sigma + t\omega$.
At $t=0$, $H(\sigma,0)=\sigma=1_{\mathcal{I}}$.
At $t=1$, $H(\sigma,1)=\omega$ is the constant map.

In section 4.1 of \cite{Rosenfeld1979}, they define {\bf plateaus} to be maximal connected subsets of $\D$ on which $\sigma$ has a constant value.
In future work, we could expand on this concept by restricting ourselves to plateaus that can be written in the form $J \cup \jint{J}$ for some $J \subset \D$.
This would ultimately allow us to characterize the space of more complicated digital images, rather than just a space of digital Jordan curves.
To account for the finite shades of gray that may appear on a computer screen, we would also like to consider finite models of the grayscale.
For example, we may take $C_n$ to be a COTS of length $n$ representing $n$ shades of gray, and a function space $\set{\sigma:\D \to C_n}$.
Since $\D$ is the product of two COTS, we may be able to consider digital grayscale images as subsets of product of three COTS.

\subsection{Digital 3-Space}

Studying objects in digital 3-space is just as prevalent, if not more so, than studying objects in a digital plane.
The same interests still apply: feature detection, region segmentation, etc.
In Section 4 of \cite{Eckhardt2003}, they describe five topologies on $\zz^3$ that may be used for digital topology; there are two obvious topologies that come to mind.
In the spirit of Khalimsky, digital 3-space could be interpreted as the product of three COTS, and it would inherit the product topology.
We get another topology from the argument of \cite{Marcus1970}, which extends to $\zz^d$ for all $d \geq 1$.
Digital 3-space is currently of interest in medical imaging.
In \cite{Hajri2015}, they use digital 3-space in screening for tumors. 
In \cite{Saha2013}, they combine digital 3-space with fuzzy topology for use in medical imaging.

\subsection{Character Recognition}

A character in a 12-point font typically lies inside a $16\times16$ pixel box.
We can associate this to a $33\times33$ Khalimsky digital plane, $\D_{33\times33}$ whose four cornerpoints are closed.
In this way, the closed points represent vertices of pixels; the mixed points represent edges of pixels; the open points represent the interiors of pixels.
Using that interpretation, Figure \ref{fig:char7} shows $\D_{33\times33}$ in which the contractible character ``7'' has been drawn.
In solid red lines, we have outlined every pixel that contains a portion of the ``7,'' joining two pixels if they have an edge in common.
It is easy to observe that the bold lines do not outline a Jordan curve.
We have added dashed red lines to include additional pixels that will leave us with a Jordan curve.
Using the language of \cite{Golomb1954}, this resulting Jordan curve would be a $27$-omino.

\begin{figure}
    \centering
    \begin{tikzpicture}[scale=.5]
    \draw[blue] (5.5,2.5)--(9.5,13.5)--(4.5,13.5)--(3.5,12.5);
    \draw[blue] (5.5,8.5)--(10.5,8.5);
    \draw[gray, thin] (0,0) grid (16,16);
    \draw[red, ultra thick] (5,2)--(6,2)--(6,3)--(6,4)--(7,4)--(7,5)--(7,6)--(8,6)--(8,7)--(8,8)--(11,8)--(11,9)--(9,9)--(9,12)--(10,12)--(10,14)--(4,14)--(4,13)--(3,13)--(3,12)--(4,12)--(4,13)--(9,13)--(9,12)--(8,12)--(8,9)--(5,9)--(5,8)--(7,8)--(7,7)--(6,7)--(6,6)--(6,4)--(5,4)--(5,2);
    \draw[red, ultra thick, dashed] (6,3)--(7,3)--(7,4);
    \draw[red, ultra thick, dashed] (8,13)--(8,12);
    \draw[red, ultra thick, dashed] (3,13)--(3,14)--(4,14);
    \end{tikzpicture}
    \caption{A pixelated character ``7''}
    \label{fig:char7}
\end{figure}

After adjusting the red outline such that it is a Jordan curve, we are left with a character that does not resemble a ``7'' as well as before.
For this reason, it is easy to see the benefits of the many topologies on $\zz^2$ described by \v{S}lapal, which allow for Jordan curves that turn at sharper angles.
The space $\J\paren{\D_{33\times33}}$ contains all Jordan curves in a $16\times16$ pixel plane, some of which we might recognize as characters.
Declare Figure \ref{fig:char7} to be the archetypal character ``7,'' $J_7$, and consider $J \in \J\paren{\D_{33\times33}}$ to be any other candidate ``7.''
Given a motion planner $s:\paren{\J\paren{\D_{33\times33}}} \longrightarrow \paren{\J\paren{\D_{33\times33}}}^{[0,1]}$, we may consider $\abs{\img{{s\paren{J_7,J}}}}$ to be the {\bf distance} between $J_7$ and $J$ in $\J\paren{\D_{33\times33}}$.
In this way, we may measure the similarity of two characters in a given space of digital Jordan curves.
An obvious shortcoming of this approach is that, as of now, it only applies to characters that are contractible in the plane.
For example, the character ``6'' cannot be written as the union of a Jordan curve and its interior.
In \cite{Khalimsky1990}, they consider {\bf robust scenes} that are partitions of the Khalimsky digital plane into regions separated by COTS-arcs and Jordan curves.
Although this would allow for characters that are not contractible, the digital Jordan curves in that paper must also be closed sets of the digital plane, preventing the space from being connected.
In \cite{Frank1996}, they discuss the use of pretopologies for character recognition, which we may also choose to incorporate in our approach.


\begin{thebibliography}{10}

\bibitem{Hajri2015}
M~Al~Hajri, Karim Belaid, and Lamia Jaafar, \emph{On khalimsky topology and
  applications on the digital image segmentation}, Applied Mathematical
  Sciences \textbf{9} (2015), 3687--3701.

\bibitem{Alexandroff1937}
P.S. Alexandroff, \emph{Diskrete ra\"{u}me}, Rec. Math. [Mat. Sbornik] N.S.
  \textbf{2(44)} (1937), no.~3, 501---519.

\bibitem{Barmak2012}
Jonathan~A. Barmak, \emph{Algebraic topology of finite topological spaces and
  applications}, Lecture Notes in Mathematics, vol. 2032, Springer, Heidelberg,
  2011. \MR{3024764}

\bibitem{Berg1975}
Gordon~O. Berg, W.~Julian, R.~Mines, and F.~Richman, \emph{The constructive
  {J}ordan curve theorem}, Rocky Mountain J. Math. \textbf{5} (1975), 225--236.
  \MR{0410701}

\bibitem{Blaszczyk2018}
Zbigniew B\l~aszczyk and Jos\'{e}~Gabriel Carrasquel-Vera, \emph{Topological
  complexity and efficiency of motion planning algorithms}, Rev. Mat. Iberoam.
  \textbf{34} (2018), no.~4, 1679--1684. \MR{3896245}

\bibitem{Boxer1999}
Laurence Boxer, \emph{A classical construction for the digital fundamental
  group}, J. Math. Imaging Vision \textbf{10} (1999), no.~1, 51--62.
  \MR{1692842}

\bibitem{Boxer2005}
\bysame, \emph{Properties of digital homotopy}, J. Math. Imaging Vision
  \textbf{22} (2005), no.~1, 19--26. \MR{2138582}

\bibitem{Boxer2017a}
Laurence Boxer and P~Staecker, \emph{Homotopy relations for digital images},
  Note di Matematica \textbf{37} (2017), 99--126.

\bibitem{Chassery1979}
Jean-Marc Chassery, \emph{Connectivity and consecutivity in digital pictures},
  Computer Graphics and Image Processing \textbf{9} (1979), no.~3, 294 -- 300.

\bibitem{Eckhardt1994}
Ulrich Eckhardt and Longin~Jan Latecki, \emph{Digital topology}, Digital
  Topology, 1994.

\bibitem{Eckhardt2003}
\bysame, \emph{Topologies for the digital spaces $\mathbb{Z}^2$ and
  $\mathbb{Z}^3$}, Computer Vision and Image Understanding \textbf{90} (2003),
  no.~3, 295 -- 312.

\bibitem{El-Atik2002}
A.~El-Fattah El-Atik, M.~E. Abd El-Monsef, and E.~I. Lashin, \emph{On finite
  {$T_0$} topological spaces}, Proceedings of the {N}inth {P}rague
  {T}opological {S}ymposium (2001), Topol. Atlas, North Bay, ON, 2002,
  pp.~75--90. \MR{1906830}

\bibitem{Farber2001}
Michael Farber, \emph{Topological complexity of motion planning}, Discrete
  Comput. Geom. \textbf{29} (2003), no.~2, 211--221. \MR{1957228}

\bibitem{Fernandez-Ternero2018}
D.~Fern\'{a}ndez-Ternero, E.~Mac\'{i}as-Virg\'{o}s, E.~Minuz, and J.~A.
  Vilches, \emph{Discrete topological complexity}, Proc. Amer. Math. Soc.
  \textbf{146} (2018), no.~10, 4535--4548. \MR{3834677}

\bibitem{Frank1996}
L.~{Frank} and E.~{Hubert}, \emph{Pretopological approach for supervised
  learning}, Proceedings of 13th International Conference on Pattern
  Recognition, vol.~4, Aug 1996, pp.~256--260 vol.4.

\bibitem{Golomb1954}
S.~W. Golomb, \emph{Checker boards and polyominoes}, Amer. Math. Monthly
  \textbf{61} (1954), 675--682. \MR{0067055}

\bibitem{Gonzalez2017}
Jes\'{u}s Gonz\'{a}lez, \emph{Simplicial complexity: piecewise linear motion
  planning in robotics}, New York J. Math. \textbf{24} (2018), 279--292.
  \MR{3778506}

\bibitem{Gonzalez2010}
Jes\'{u}s Gonz\'{a}lez, B\'{a}rbara Guti\'{e}rrez, and Sergey Yuzvinsky,
  \emph{Higher topological complexity of subcomplexes of products of spheres
  and related polyhedral product spaces}, Topol. Methods Nonlinear Anal.
  \textbf{48} (2016), no.~2, 419--451. \MR{3642766}

\bibitem{Hatcher2002}
Allen Hatcher, \emph{Algebraic topology}, Cambridge University Press,
  Cambridge, 2002. \MR{1867354}

\bibitem{OEISA140517}
OEIS~Foundation Inc., \emph{The on-line encyclopedia of integer sequences},
  http://oeis.org/A140517, 2019.

\bibitem{OEISA118797}
\bysame, \emph{The on-line encyclopedia of integer sequences},
  http://oeis.org/A118797, 2019.

\bibitem{OEIS000217}
\bysame, \emph{The on-line encyclopedia of integer sequences},
  http://oeis.org/A000217, 2019.

\bibitem{Iwase2012}
Norio Iwase and Michihiro Sakai, \emph{Topological complexity is a fibrewise
  {L}-{S} category}, Topology Appl. \textbf{157} (2010), no.~1, 10--21.
  \MR{2556074}

\bibitem{Kandola2018}
Shelley {Kandola}, \emph{{The Topological Complexity of Finite Models of
  Spheres}}, arXiv e-prints (2018), arXiv:1812.07604.

\bibitem{Karaca2018}
\.{I}smet Karaca and Melih \.{I}s, \emph{Digital topological complexity
  numbers}, Turkish J. Math. \textbf{42} (2018), no.~6, 3173--3181.
  \MR{3885444}

\bibitem{Khalimsky1990}
Efim Khalimsky, Ralph Kopperman, and Paul~R. Meyer, \emph{Boundaries in digital
  planes}, J. Appl. Math. Stochastic Anal. \textbf{3} (1990), no.~1, 27--55.
  \MR{1051772}

\bibitem{Khalimsky1990a}
\bysame, \emph{Computer graphics and connected topologies on finite ordered
  sets}, Topology Appl. \textbf{36} (1990), no.~1, 1--17. \MR{1062180}

\bibitem{Kiselman2000}
Christer Kiselman, \emph{Digital jordan curve theorems}, Digital Jordan Curve
  Theorems, 09 2000.

\bibitem{Kong1992}
T.~Y. Kong, A.~W. Roscoe, and A.~Rosenfeld, \emph{Concepts of digital
  topology}, Topology Appl. \textbf{46} (1992), no.~3, 219--262, Special issue
  on digital topology. \MR{1198732}

\bibitem{Kovalevsky1989}
V.A Kovalevsky, \emph{Finite topology as applied to image analysis}, Computer
  Vision, Graphics, and Image Processing \textbf{45} (1989), no.~2, 266.

\bibitem{LS34}
L.~Lusternik and L.~Schnirelmann, \emph{M\'{e}thodes topologiques dans les
  probl\`{e}mmes variationnels}, Actualit\'{e}s scientifiques et industrielles
  ; 188, Hermann, Paris, 1934 (fre).

\bibitem{Marcus1970}
Dan Marcus and Cleveland State University Problem~Solving Group, \emph{5712},
  The American Mathematical Monthly \textbf{77} (1970), no.~10, 1119--1119.

\bibitem{McCord1966}
Michael~C. McCord, \emph{Singular homology groups and homotopy groups of finite
  topological spaces}, Duke Math. J. \textbf{33} (1966), 465--474. \MR{0196744}

\bibitem{Ptak1997}
Pavel Ptak, Helmut Kofler, and Walter Kropatsch, \emph{Digital topologies
  revisited: An approach based on the topological point-neighbourhood},
  Discrete Geometry for Computer Imagery (Berlin, Heidelberg) (Ehoud Ahronovitz
  and Christophe Fiorio, eds.), Springer Berlin Heidelberg, 1997, pp.~151--159.

\bibitem{Rosenfeld2007}
Azriel Rosenfeld, \emph{Digital topology}, Amer. Math. Monthly \textbf{86}
  (1979), no.~8, 621--630. \MR{546174}

\bibitem{Rosenfeld1979}
Azriel Rosenfeld, \emph{Fuzzy digital topology}, Information and Control
  \textbf{40} (1979), no.~1, 76 -- 87.

\bibitem{Rudyak2010}
Yuli~B. Rudyak, \emph{On higher analogs of topological complexity}, Topology
  Appl. \textbf{157} (2010), no.~5, 916--920. \MR{2593704}

\bibitem{Saha2013}
Punam~K. Saha, \emph{Fuzzy digital topology and geometry and their applications
  to medical imaging}, Pattern Recognition and Machine Intelligence (Berlin,
  Heidelberg) (Pradipta Maji, Ashish Ghosh, M.~Narasimha Murty, Kuntal Ghosh,
  and Sankar~K. Pal, eds.), Springer Berlin Heidelberg, 2013, pp.~13--29.

\bibitem{Stong1966}
R.~E. Stong, \emph{Finite topological spaces}, 1966, pp.~325--340. \MR{0195042}

\bibitem{Tanaka2018}
Kohei Tanaka, \emph{A combinatorial description of topological complexity for
  finite spaces}, Algebr. Geom. Topol. \textbf{18} (2018), no.~2, 779--796.
  \MR{3773738}

\bibitem{Slapal2006}
Josef \v{S}lapal, \emph{Digital {J}ordan curves}, Topology Appl. \textbf{153}
  (2006), no.~17, 3255--3264. \MR{2260583}

\bibitem{Slapal2009}
\bysame, \emph{Jordan curve theorems with respect to certain pretopologies on
  $\zz^2$}, Discrete Geometry for Computer Imagery (Berlin, Heidelberg)
  (Sre{\v{c}}ko Brlek, Christophe Reutenauer, and Xavier Proven{\c{c}}al,
  eds.), Springer Berlin Heidelberg, 2009, pp.~252--262.

\bibitem{Schwarz1958}
A.~S. \v{S}varc, \emph{The genus of a fiber space}, Dokl. Akad. Nauk SSSR
  (N.S.) \textbf{119} (1958), 219--222. \MR{0102812}

\bibitem{Zadeh1965}
L.A. Zadeh, \emph{Fuzzy sets}, Information and Control \textbf{8} (1965),
  no.~3, 338 -- 353.

\end{thebibliography}

\end{document}